\documentclass[12pt,oneside,reqno]{amsart} 
\usepackage{etex}
\usepackage[T2A]{fontenc}
\usepackage[cp1251]{inputenc}
\usepackage{pifont}
\usepackage[dvips]{graphicx}
\usepackage{amsmath,latexsym,amsthm,cmap,indentfirst,setspace,ccaption,amssymb,color,amscd,mathrsfs,hyperref,soul}
\usepackage{array,enumitem,tikz,multirow,wasysym}

\usetikzlibrary{trees,arrows}
\tikzstyle{level 1}=[level distance=30mm, sibling distance=30mm]
\tikzstyle{level 2}=[level distance=30mm, sibling distance=15mm]
\tikzstyle{level 3}=[level distance=20mm]

\usepackage{fourier}

\usetikzlibrary{graphs,graphs.standard}

\newcount\nodecount
\tikzgraphsset{
	declare={subgraph N}%
	{
		[/utils/exec={\global\nodecount=0}]
		\foreach \nodetext in \tikzgraphV
		{  [/utils/exec={\global\advance\nodecount by1}, 
			parse/.expand once={\the\nodecount/\nodetext}] }
	},
	declare={subgraph C}%
	{
		[cycle, /utils/exec={\global\nodecount=0}]
		\foreach \nodetext in \tikzgraphV
		{  [/utils/exec={\global\advance\nodecount by1}, 
			parse/.expand once={\the\nodecount/\nodetext}] }
	}
}

\usepackage{amssymb,lipsum}
\usepackage{float,xypic}
\usepackage{cmap,longtable}

\usepackage{geometry}
\usepackage{caption, subcaption}

\usepackage{verbatim}
\hypersetup{
	bookmarks=true,         
	unicode=true,          
	pdftoolbar=true,        
	pdfmenubar=true,        
	pdffitwindow=false,     
	pdfnewwindow=true,      
	colorlinks=true,       
	linkcolor=blue,          
	citecolor=blue,        
	filecolor=magenta,      
	urlcolor=cyan           
}

\DeclareMathOperator{\Cr}{Cr}
\DeclareMathOperator{\He}{Hess}
\DeclareMathOperator{\tr}{tr}
\DeclareMathOperator{\diag}{diag}

\DeclareMathOperator{\rk}{rk}
\DeclareMathOperator{\Pic}{Pic}

\DeclareMathOperator{\Fix}{Fix}
\DeclareMathOperator{\Aut}{Aut}

\DeclareMathOperator{\Sp}{Sp}

\DeclareMathOperator{\Ker}{Ker}
\DeclareMathOperator{\PGL}{PGL}
\DeclareMathOperator{\GL}{GL}
\DeclareMathOperator{\SL}{SL}
\DeclareMathOperator{\PSL}{PSL}
\DeclareMathOperator{\SO}{SO}
\DeclareMathOperator{\PO}{PO}
\DeclareMathOperator{\PSp}{PSp}
\DeclareMathOperator{\OO}{O}
\DeclareMathOperator{\Eu}{Eu}
\DeclareMathOperator{\E}{E}
\DeclareMathOperator{\D}{D}
\DeclareMathOperator{\A}{A}

\DeclareMathOperator{\Proj}{Proj}
\DeclareMathOperator{\Bir}{Bir}
\DeclareMathOperator{\sgn}{sgn}
\DeclareMathOperator{\ord}{ord}
\DeclareMathOperator{\J}{J}
\DeclareMathOperator{\Center}{Z}
\DeclareMathOperator{\Char}{char}
\DeclareMathOperator{\Real}{Re}
\DeclareMathOperator{\Symm}{Sym}
\DeclareMathOperator{\Imag}{Im}
\DeclareMathOperator{\Gal}{Gal}
\DeclareMathOperator{\Hom}{Hom}
\DeclareMathOperator{\Heis}{\mathcal{H}}

\theoremstyle{plain}
\newtheorem{thm}{Theorem}[section]
\newtheorem{lem}[thm]{Lemma} 
 
\newtheorem{prop}[thm]{Proposition} 

\theoremstyle{definition}
 
\newtheorem{rem}[thm]{Remark}
\newtheorem{ex}[thm]{Example}

\newtheorem*{convention*}{Convention}

\makeatletter
\g@addto@macro{\endabstract}{\@setabstract}
\newcommand{\authorfootnotes}{\renewcommand\thefootnote{\@fnsymbol\c@footnote}}%
\makeatother

\let\oldtocsection=\tocsection

\let\oldtocsubsection=\tocsubsection

\renewcommand{\tocsection}[2]{\hspace{0em}\oldtocsection{#1}{#2}}
\renewcommand{\tocsubsection}[2]{\hspace{4em}\oldtocsubsection{#1}{#2}}

\begin{document}
	
	\begin{center}
		\LARGE 
		Automorphisms of real del Pezzo surfaces\\ and the real plane Cremona group \par \bigskip
		
		\normalsize
		\authorfootnotes
		\large Egor Yasinsky\footnote{yasinskyegor@gmail.com} \par \bigskip
		\normalsize
		Universit\"{a}t Basel \\Departement Mathematik und Informatik\\
		Spiegelgasse 1, 4051 Basel, Switzerland\par \bigskip
		
	\end{center}

\newcommand{\CC}{\mathbb C} 
\newcommand{\FF}{\mathbb F}
\newcommand{\RR}{\mathbb R}
\newcommand{\ZZ}{\mathbb Z}
\newcommand{\PP}{\mathbb P}
\newcommand{\kk}{\mathbb{k}}
\newcommand{\Sph}{\mathbb S}
\newcommand{\Torus}{\mathbb T}
\newcommand{\RP}{\mathbb{RP}}
\newcommand{\Sym}{\mathfrak S}
\newcommand{\Alt}{\mathfrak A}
\newcommand{\Dih}{\mathrm D}
\newcommand{\BDih}{\mathrm{BD}}
\newcommand{\QD}{\mathrm{QD}}
\newcommand{\Mgr}{\mathrm M}
\newcommand{\Quat}{\mathrm Q}
\newcommand{\XC}{X_{\mathbb C}}
\newcommand{\Weyl}{\mathscr W}
\newcommand{\CCC}{\mathscr C}
\newcommand{\RRR}{\mathscr R}
\newcommand{\QQ}{\mathcal Q}
\newcommand{\id}{\mathrm{id}}
\newcommand{\pt}{\mathrm{pt}}
\newcommand{\DP}{\mathcal{D}}
\newcommand{\CB}{\mathcal{C}}
\newcommand{\Quad}{Q}

	\newcommand*\conj[1]{\overline{#1}}
	
	\begin{abstract}
		We study automorphism groups of real del Pezzo surfaces, concentrating on finite groups acting with invariant Picard number equal to one. As a result, we obtain a vast part of classification of finite subgroups in the real plane Cremona group.
	\end{abstract}

\setcounter{tocdepth}{2}

{
	\hypersetup{linkcolor=black}
\tableofcontents
}
	
\section{Introduction}

\subsection{The classification problem} This paper is devoted to the study of finite automorphism groups of real del Pezzo surfaces. Our main motivation is the classification of finite  subgroups of the real plane Cremona group; hence this paper may be viewed as a follow-up paper to \cite{Yas}. Recall that the Cremona group $\Cr_n(\kk)=\Bir(\PP_\kk^n)$ is the group of birational automorphisms of the $n$-dimensional projective space over a field $\kk$. The finite subgroups of $\Cr_1(\kk)\cong\PGL_2(\kk)$ have been known since Klein's time (see Lemma \ref{lem: PGL subg} and \cite{BeaPGL}). By contrast, the complete classification of finite subgroups of $\Cr_2(\kk)$ for $\kk=\overline{\kk}$ was obtained by I.~Dolgachev and V.~Iskovskikh only in 2009 and involves different hard techniques of modern birational geometry, such as Mori theory, equvariant resolution of singularities, etc. For the exposition of these results, as well as some historical notes, we refer the reader to the original papers \cite{blancabelian} (case of abelian subgroups) and \cite{di}. 

Much less is known for algebraically non-closed fields or $n\geqslant 3$. Classification of finite subgroups of $\Cr_2(\RR)$ was initiated by the author in \cite{Yas} where subgroups of odd order were classified up to conjugacy. The goal of this paper is to extend these results much further and to classify all finite groups acting minimally on real del Pezzo surfaces (see below). As will be explained below, this gives a vast part of classification of finite subgroups in $\Cr_2(\RR)$.

As for the case $n\geqslant 3$, $\kk=\CC$, the classification seems out of reach at the moment. There are some partial results, see e.g. \cite{ProSimple}, \cite{pro2}. Alternatively, one can try looking at things from a different point of view using the notion of {\it Jordan property} introduced in \cite{Pop-Makar}. Recall that an abstract group $\Gamma$ is called {\it Jordan} if there exists a positive integer $m$ such that every finite subgroup $G\subset \Gamma$ contains a normal abelian subgroup $A\triangleleft G$ of index at most $m$. The minimal such $m$ is called the {\it Jordan constant} of $\Gamma$ and is denoted by $\J(\Gamma)$. There is a remarkable result\footnote{It was initially proved modulo so-called Borisov-Alexeev-Borisov conjecture, which was settled in any dimension in \cite{BAB}.} \cite{PS-BAB}\label{thm: PS Jordan}:
\begin{thm}[Yu. Prokhorov, C. Shramov]
	Let $\Char\kk=0$. Then $\Cr_n(\kk)$ is Jordan for each $n\geqslant 1$.
\end{thm}

This theorem allows, at least theoretically, to classify finite subgroups of Cremona groups <<up to abelian subgroups>>. Indeed, we know that for each extension
\[
1\to A\to G\to G/A\to 1,
\]
where $A\subset G$ is a normal abelian subgroup, the sizes of $G/A$ are uniformly bounded by a universal constant depending only on $n$ and $\kk$. How large can be the list of possible $G/A$, i.e. what are precise values of $\J(\Cr_n(\kk))$? There are some results in this direction.

\begin{thm}\cite[Theorem 1.9,1.10]{YasJord}
	One has
	\[
	\J(\Cr_2(\CC))=7200,\ \ \J(\Cr_2(\RR))=120.
	\]
\end{thm}

\begin{thm}\cite[Theorem 1.2.4]{PS-J3}
	Suppose that the field $\kk$ has characteristic 0. Then one has
	\[
	\J(\Cr_3(\kk))\leqslant 107\ 495\ 424.
	\]
\end{thm}

\subsection{$G$-surfaces}\label{subsection: setup} Let us briefly recall a general strategy of classification of finite subgroups in $\Cr_2(\kk)$. Throughout this paper $G$ denotes a finite group. Let $\kk$ be a perfect field. We use the standard language of $G$-varieties (see e.g. \cite{di} or \cite{Yas}). The modern approach to classification is based on the following observations:

\begin{itemize}
	\item For any finite subgroup $G\subset\Cr_2(\kk)$ there exists a $\kk$-rational smooth projective surface $X$, an injective homomorphism $\iota: G\rightarrow \Aut_{\kk}(X)$ and a birational $G$-equivariant $\kk$-map $\psi: X\dasharrow\PP^2_\kk$, such that
	\[
	G=\psi\circ\iota (G)\circ{\psi}^{-1}
	\]
	This process of passing from a birational action of $G$ on $\PP_\kk^2$ to a regular action on $X$ is usually called the {\it regularization} of the $G$-action. On the other hand, for a $\kk$-rational $G$-surface $X$ a birational map $\psi: X\dasharrow\PP^2_\kk$ yields an injective homomorphism \[
	i_\psi: G\to\Cr_2(\kk),\ \ g\mapsto \psi\circ g\circ \psi^{-1}.
	\]
	Moreover, two subgroups of $\Cr_2(\kk)$ are conjugate if and only if the corresponding $G$-surfaces are birationally equivalent. So, there is a natural bijection between the conjugacy classes of finite subgroups $G\subset\Cr_2(\kk)$ and birational isomorphism classes of smooth $\kk$-rational $G$-surfaces $(X,G)$.
	
	\item For any projective geometrically smooth $G$-surface $X$ over $\kk$ there exists a birational $G$-equivariant $\kk$-morphism $X\to X_{\rm min}$ where the $G$-surface $X_{\min}$ is $G$-minimal. The latter means that any birational $G$-equivariant $\kk$-morphism $X_{\rm min}\to Z$ is an isomorphism. If the surface $X$ is additionally $\overline{\kk}$-rational, then one of the following holds \cite[Theorem 5]{di-perf}:
	
	\begin{enumerate}
		\item $X_{\rm min}$ admits a conic bundle structure with ${\Pic(X)^G\cong\mathbb{Z}^2}$;
		\item $X_{\rm min}$ is a del Pezzo surface with ${\Pic(X)^G\cong\mathbb{Z}}$. 
	\end{enumerate}
\end{itemize}

So, the classification of finite subgroups of $\Cr_2(\kk)$ is equivalent to birational classification of minimal pairs $(X,G)$ described above. The goal of this paper is to describe all the minimal pairs $(X,G)$ with $X$ a real del Pezzo surface, i.e. to complete the study of the first case in the previous dichotomy.

\vspace{0.3cm}

\subsection{Some comments on the conic bundle case}

The reader may wonder why do we focus only on the case of del Pezzo surfaces in this paper. The following example can serve as a partial explanation (or rather an excuse). Namely, it shows that there exist infinitely many pairwise non-conjugate involutions in $\Cr_2(\RR)$, which are all conjugate over $\CC$. So, the classification of finite subgroups up to conjugacy in $\Cr_2(\RR)$ is a much more subtle question. For the philosophy of $\kk$-birational unboundedness of conic bundles quotients standing behind this example see \cite{TrepalinCB}.

\begin{ex}
	Consider the surface
	\[
	Z_n:\ \ \ x^2\prod_{k=1}^{2n}(t_0^2+k^2t_1^2)+y^2t_0^{4n}+z^2t_1^{4n}=0
	\]
	in $\Proj\RR[x,y,z]\times\Proj\RR[t_0,t_1]\cong\PP_\RR^2\times\PP_\RR^1$. The projection to $\PP_\RR^1$-factor defines a structure of a conic bundle on $\pi: Z_n\to\PP^1$. Its geometrically singular fibers lie over the points $p_k=[ik:1]$, $\overline{p}_k=[-ik:1]$ (here $i=\sqrt{-1}$) and are given by $y^2+z^2=0$.
	
	Let $g_n\in\Aut(\PP_\RR^1)$ be the involution $[t_0:t_1]\mapsto [-t_0:t_1]$. The complex involution $\sigma$ and the automorphism $g_n$ act on $Z_n$ as shown on Figure \ref{pic:conic bundle}.
	\begin{figure}[h]
		\centering
		\includegraphics[width=.6\linewidth]{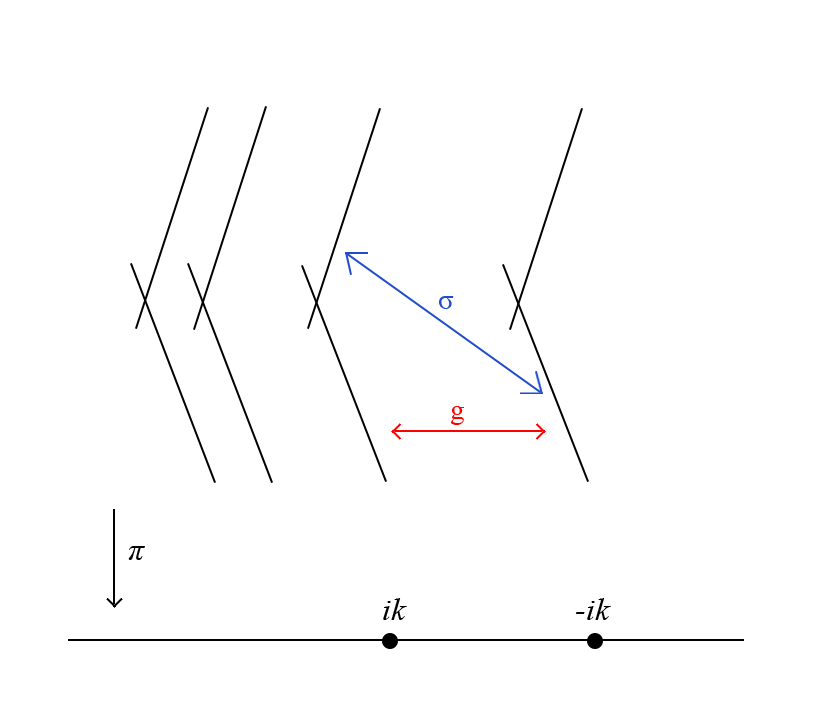}
		\caption{Involutions on $Z_n$.}
		\label{pic:conic bundle}
	\end{figure}
	Note that 
	\begin{enumerate}
		\item Irreducible components of all singular fibers of $Z_n$ can be $\Gamma$-equivariantly contracted on a conic bundle without singular fibers, hence $Z_n$ is rational over $\RR$. In particular, $g_n\in\Cr_2(\RR)$.
		\item $Z_n$ is $\langle g_n\rangle$-minimal. On the other hand, $Z_n\otimes\CC$ is not $\langle g_n\rangle$-minimal over $\CC$, as we can contract disjoint irreducible components of all singular fibers onto some Hirzebruch surface equivariantly. Using elementary transformations between Hirzebruch surfaces (or just \cite[Theorem 1]{blancabelian}), we observe that all $g_n$ are conjugate in $\Cr_2(\CC)$.
		\item The surface $X_n=Z_n/\langle g_n\rangle $ has a structure of a conic bundle with $2n$ singular fibers, and irreducible components in each fiber are complex conjugate. In particular $X_n$ is $\RR$-minimal. Thus $X_n$ is not rational over $\RR$ when $n>3$ (e.g. by Iskovskikh's rationality criterion, see \cite[\S 4]{isk-1}). 
	\end{enumerate}
	Consider two finite subgroups $G_1,\ G_2\subset\Cr_2(\RR)$ with regularizations $(Y_1,G_1)$ and $(Y_2,G_2)$ respectively. Assuming that $G_1$ is conjugate to $G_2$, there exists a common equivariant resolution $Y\to Y_1$, $Y\to Y_2$ such that the actions of $G_1$ and $G_2$ coincide on $Y$. Therefore, $Y_1/G_1$ is birational to $Y_2/G_2$. However, for $n,m>3$ the conic bundles $X_n$ and $X_m$ are not pairwise birational to each other (see e.g. \cite[Theorem 1.6]{isk-4} or \cite[Theorem 4.3]{kol}). Therefore, involutions $g_n$ and $g_m$ are not conjugate in $\Cr_2(\RR)$. 
\end{ex}

This paper is organised as follows. Section \ref{sec: prelim} recalls some basic facts about del Pezzo surfaces, their topology and relation to Weyl groups; it also gathers some auxiliary results about Sarkisov program and classical linear groups that will be used later. The reader may skip this section and return to it later, if needed. In Sections \ref{sec: dP8}-\ref{section: dP1} we study groups acting on real del Pezzo surfaces $X$ with $K_X^2\geqslant 3,\ K_X^2\ne 7,9$. The cases $K_X^2=9$ and $K_X^2=7$ are trivial. Indeed, a del Pezzo surface of degree 7 is never $G$-minimal, and a real del Pezzo surface $X$ of degree 9 with $X(\RR)\ne\varnothing$ is isomorphic to $\PP_\RR^2$, so finite groups acting on it are well known, see Lemma \ref{lem: PGL subg}. 

In comparison with the case $\kk=\CC$, we have to deal with {\it real forms} of del Pezzo surfaces (i.e. non-isomorphic real surfaces that become isomorphic over $\CC$). Here we face an additional difficulty, since the complete classification of possible automorphism groups of del Pezzo surfaces is available only over the field of complex numbers; in fact, this classification was heavily used in the work of Dolgachev and Iskovskikh. So, in Sections \ref{sec: dP8}-\ref{sec: dP4} (i.e. $K_X^2\geqslant 4$) we generally adapt the following strategy to classification: for each ($\RR$-rational) real form of a del Pezzo surface $X$, we study the group $\Aut(X)$ (giving its precise description in many cases), and then determine possible finite groups $G\subset\Aut(X)$ that can act minimally on $X$. To find such groups $G$, we usually investigate the action of $\Gal(\CC/\RR)\times G$ on $X\otimes\CC$; for high degree del Pezzo surfaces, we look directly at the intersection graph of $(-1)$-curves, which is easy to analyze in these cases. For low degree surfaces ($K_X^2\leqslant 3$), our approach becomes more combinatorial. Both real structure $\sigma$ on $X$ and automorphisms $G\subset\Aut(X)$ can be considered as elements of the Weyl group $\Weyl$ associated to $X$ (see \S \ref{subsec: prelim}). Using the classification of conjugacy classes in $\Weyl$, we determine possible pairs $(\sigma, G)$ such that the action of $\langle\sigma\rangle\times G$ on $X\otimes\CC$ is minimal. In many cases we work with explicit equations of $X$ and $G$ (for example, in Section \ref{sec: dP3} we adapt for our purposes Sylvester's classical approach to cubic surfaces).

In Appendixes \ref{app: simple groups} and \ref{app: non-solvable} we focus on some special classes of finite subgroups in $\Cr_2(\RR)$ (being motivated by the study of those in \cite{Tsygankov}, \cite{ProSimple}, \cite{Quasi}) and in particular classify non-solvable finite groups acting on real geometrically rational surfaces. Our goal is to demonstrate that: (1) this classification can be obtained independently of the ``complete'' classification of all finite subgroups and (2) the corresponding list is considerably shorter than in the case $\kk=\CC$. Finally, for the reader's convenience, some technical information about real invariants of some finite groups is included in Appendix \ref{appendix}.

\subsection{Notation and conventions} We use the following notation and conventions.
\begin{itemize}
	\item \textcolor{black}{In this paper, we say that a real del Pezzo surface $X$ is {\it $G$-minimal}, or simply {\it $G$ is minimal} (when it acts on $X$), if and only if any birational $G$-morphism $X\to Y$ of $G$-surfaces is an isomorphism. Further, we say that $X$ is {\it strongly\footnote{This term is not conventional. We use it only for brevity.} $G$-minimal}, or {\it $G$ is strongly minimal}, if and only if $\rk\Pic(X)^G=1$. Clearly, strong $G$-minimality implies $G$-minimality, but not vice versa (consider e.g. $X=\PP_\RR^1\times\PP_\RR^1$ with $G$ acting preserving the factors).}
	\item Moreover, all del Pezzo surfaces are assumed to be $\RR$-rational (if not stated otherwise), and in particular their real loci $X(\RR)$ are not empty. The latter condition implies that 
	\[
	\Pic(X_\CC)^\Gamma=\Pic(X),\ \ \ \ \text{hence}\ \Pic(X_\CC)^{\Gamma\times G}=\Pic(X)^G,
	\]
	where $X_\CC=X\otimes\CC$, and $\Gamma$ is the Galois group $\Gal(\CC/\RR)$ generated by the involution $\sigma$. \textcolor{black}{Therefore, a real del Pezzo surface $X$ is strongly $G$-minimal if and only if $X_\CC$ is strongly $\Gamma\times G$-minimal.} 
	\item We denote by $Q_{r,s}$ the smooth quadric hypersurface 
	\[
	\{[x_1:\ldots :x_{r+s}]: x_1^2+\ldots +x_r^2-x_{r+1}^2-\ldots -x_{r+s}^2=0\}\subset\PP_\RR^{r+s-1}.
	\] 
	\item For a real del Pezzo surface $X$, we denote by $X(a,b)$ the blow-up of $X$ at $a$ real points and $b$ pairs of complex conjugate points. We shall mostly use $\PP_\RR^2$, $\Quad_{3,1}$ or $\Quad_{2,2}$ as $X$.
	\item $\ZZ/n$ or simply $n$ is a cyclic group of order $n$;
	\item $\Dih_{n}$ is a dihedral group of order $2n$;
	\item $\BDih_{n}=\langle a,x\ |\ a^{2n}=1,\ x^2=a^n,\ xax^{-1}=a^{-1}\rangle$ is the binary dihedral group of order $2n$;
	\item $\Sym_n$ is a symmetric group on n-letters.
	\item $A\triangle_D B$ is the diagonal product of $A$ and $B$ over their common homomorphic image $D$, i.e. the subgroup of $A\times B$ of pairs $(a,b)$ such that $\alpha(a)=\beta(b)$ for some epimorphisms $\alpha: A\to D$, $\beta: B\to D$.
	\item $A_\bullet B$ is an extension of $B$ by $A$;
	\item When running the Sarkisov program (e.g. as in Proposition \ref{prop: dP8 lineariz}) we denote by $\mathscr{D}_d$ (resp. $\mathscr{C}_d$) a del Pezzo surface (resp. a conic bundle) of degree $d$ (resp. with $d=8-K_X^2$ singular fibers).
	\item $I$ or $I_n$ denotes the identity matrix of size $n\times n$.
\end{itemize}

\addtocontents{toc}{\protect\setcounter{tocdepth}{1}}
\subsection*{Acknowledgments} The author would like to thank Andrey Trepalin for numerous useful discussions and explanation of the results of \cite{Trepalin19}, and the anonymous referee whose suggestions helped to improve both the exposition and the results of this paper. The author is also grateful to J\'{e}r\'{e}my Blanc, Yuri Prokhorov and Constantin Shramov for their valuable comments. The author acknowledges support by the Swiss National Science Foundation Grant ``Birational transformations of threefolds'' 200020\_178807.
\addtocontents{toc}{\protect\setcounter{tocdepth}{2}}

\section{Some auxiliary results}\label{sec: prelim}

\subsection{A quick look at (real) del Pezzo surfaces}\label{subsec: prelim}

Let us briefly overview some important tools that shall be used in this paper. For a more comprehensive account see e.g. \cite{cag} or \cite{cubicforms}. For the Minimal Model Program over $\RR$ and its relation to the topology of real rational surfaces see \cite{kol}. 

In this paper we are interested in the embedding of finite groups into $\Cr_2(\RR)$, hence we focus on $\RR$-rational surfaces in the first place. When $X$ is a non-singular real projective algebraic surface its set of real points $X(\RR)$ will be always regarded as a compact two-dimensional $C^\infty$-manifold with the usual Euclidean topology. The following characterization of $\RR$-rational del Pezzo surfaces will be useful for us. 

\begin{prop}\label{prop: dP criterion of R-rationality}
	Let $X$ be a smooth real del Pezzo surface. Then $X$ is $\RR$-rational if and only if $X(\RR)$ is nonempty and connected.
\end{prop}
\begin{proof}
	The result is classical and follows from \cite[Theorem 1.9, Theorem 2.2, Lemma 3.2]{kol}.
\end{proof}
\begin{rem}\label{rem: topology of dP}
	In fact, for an $\RR$-rational del Pezzo surface $X$, its real locus $X(\RR)$ is diffeomorphic to one of the following manifolds:
	\begin{enumerate}
		\item $\Sph^2$ if $X\cong\Quad_{3,1}(0,b)$;
		\item $\Torus^2$ if $X\cong\Quad_{2,2}(0,b)$;
		\item $N_g=\#_g\RP^2$ if $X\cong\PP_\RR^2(a,b)$ where $g=a+1$ and $1\leqslant g\leqslant9$.
	\end{enumerate}
See \cite{kol} for details.
\end{rem}

Another powerful tool for studying del Pezzo surfaces is the {\it Weyl groups}. Let $X_\CC$ be a complex del Pezzo surface of degree $d\leq 6$, obtained by blowing up $\PP_\CC^2$ in $r=9-d$ points. The group $\Pic{\XC}\cong\ZZ^{r+1}$ has a basis $e_0,\ e_1,\ldots,e_r$, where $e_0$ is the pull-back of the class of a line on $\PP_\CC^2$, and $e_i$ are the classes of exceptional curves. Put
\[
\Delta_r=\{s\in\Pic(\XC):\ s^2=-2,\ s\cdot K_{\XC}=0 \}.
\]
Then $\Delta_r$ is a root system in the orthogonal complement to $K_{\XC}^{\bot}\subset\Pic(\XC)\otimes\RR$. As usual, one can associate with $\Delta_r$ the Weyl group $\Weyl(\Delta_r)$. Depending on degree $d$, the type of $\Delta_r$ and the size of $\Weyl(\Delta_r)$ are the following:

\setlength{\extrarowheight}{3pt}
\begin{table}[h!]
	\caption{The Weyl groups}
	\label{table:weyltable}
	\begin{tabular}{|c|c|c|c|c|c|c|lllllll}
		\cline{1-7}
		$d$ & 1 & 2 & 3 & 4 & 5 & 6 \\ \cline{1-7}
		$\Delta_r$ & $\E_8$ & $\E_7$ & $\E_6$ & $\D_5$ & $\A_4$ & $\A_1\times \A_2$  \\ \cline{1-7}
		$|\Weyl(\Delta_r)|$ & $2^{14}\cdot3^5\cdot5^2\cdot7$ & $2^{10}\cdot3^4\cdot5\cdot 7$ & $2^7\cdot3^4\cdot5$ & $2^7\cdot3\cdot 5$ & $2^3\cdot3\cdot 5$ & 12 \\ \cline{1-7} 
	\end{tabular}
\end{table}
Moreover, there are natural homomorphisms 
\[
\rho: \Aut(X_\CC) \rightarrow \Weyl({\Delta_r}),\ \ \eta: \Gamma={\rm Gal}(\CC/\RR)\rightarrow \Weyl({\Delta_r}),
\]
where $\rho$ is an injection for $d\leq 5$. We denote by $g^*$ the image of $g\in\Gamma\times G$ in the corresponding Weyl group.

Denote by $\mathbb{E}_r$ the sublattice of $\Pic(\XC)$ generated by the root system $\Delta_r$. For an element $g^*\in\Weyl(\Delta_r)$ denote by $\tr(g^*)$ its trace on $\mathbb{E}_r$. To determine whether a finite group $\Gamma\times G$ acts \textcolor{black}{strongly minimally} on $\XC$, we use the well-known formula from the character theory of finite groups
\begin{equation}\label{characterformula}
{\rm rk}\Pic(\XC)^{\Gamma\times G}=1+\frac{1}{|\Gamma\times G|}\sum_{g\in \Gamma\times G}\tr(g^*).
\end{equation}
Thus the group $\Gamma\times G$ acts \textcolor{black}{strongly minimally} on $\XC$ if and only if $\sum_{g\in \Gamma\times G}\tr(g^*)=0$. On the other hand, by the Lefschetz fixed point formula for any $h\in G$ we have,
\begin{equation}\label{LefschetzFixedPointFormula}
\Eu(\XC^h)=\tr(h^*)+3.
\end{equation}
\begin{rem}\label{rem: holom Lefs}
	Note that a cyclic group always has a fixed point on a complex rational variety. This follows from the holomorphic Lefschetz fixed-point formula.
\end{rem}

In this paper we shall use the known classification of conjugacy classes in the Weyl groups. These classes are indexed by {\it Carter graphs}, named e.g. $A_1$, $A_1^2$, etc. Here we follow the terminology of \cite{weyl} (used in \cite{di}). Among other things, a Carter graph determines the characteristic polynomial of an element from a given class and its trace on $K_{\XC}^{\bot}$, see \cite[Table 2]{di}. Another useful source of information about involutions in Weyl groups and real structures on del Pezzo surfaces is \cite{Wall}. Note that Wall labels the conjugacy classes by Dynkin diagrams; in the situation where it can be confusing for the reader, we give the precise correspondence between these two different notations (e.g. in  Table~\ref{table: dP1 real forms}).

\subsection{Sarkisov links}

The main tool for exploring conjugacy in Cremona groups is the {\it Sarkisov program}. Here we very briefly recall how this tool looks like. For details see \cite{isk-1}, \cite{di} or \cite{pol} for the theory developed over $\RR$.

We work in the category of $G$-surfaces over a perfect field $\kk$. Similarly to the classical case of trivial~$G$, any birational $G$-map between two $G$-surfaces can be decomposed into a sequence of birational $G$-morphisms and their inverses. A birational $G$-morphism $X\to Y$ can be thought of as a blow-up of a closed $G$-invariant $0$-dimensional subscheme $\mathfrak{p}$ of $Y$. Recall that $\deg(\mathfrak{p})=h^0(\mathcal{O}_{\mathfrak{p}})$. When $\mathfrak{p}$ is reduced and consists of closed points $y_1,\ldots,y_n$ with residue fields $\kappa(y_i)$, one has $\deg\mathfrak{p}=\sum\deg y_i$ with $\deg y_i=[\kappa(y_i):\kk]$. If $\mathfrak{p}$ is $G$-invariant, then it is a union of $G$-orbits. So, over the field of reals one can blow up orbits of real points and pairs of complex conjugate points.

In this paper we shall work with $G$-minimal del Pezzo surfaces and conic bundles (in the sense defined above). From the Mori theory's point of view, these are rational Fano-Mori $G$-fibrations of dimension two (extremal contractions $\pi: X\to C$, where $C$ is a point in the del Pezzo case, and $C$ is a curve in the conic bundle case). A birational $G$-map $f$ between Mori fibrations is a diagram of $G$-equivariant maps
\[
\xymatrix{
	X\ar@{-->}[r]^{f}\ar[d]_{\pi} & X'\ar[d]^{\pi'}\\
	C & C'
}
\]
Now, according to Sarkisov program, every birational map $f: X\dashrightarrow X'$ of rational minimal $G$-surfaces is factorized into a composition of {\it elementary Sarkisov links} of four types. For complete description of all such possible links we refer to \cite{isk-1}.

\subsection{Topological bounds}

For a finite group $G\subset\Aut(X_\CC)$, the representation 
\[
\rho: G\to\Weyl(\Delta_r)
\]
obviously restricts the order of $G$ when $K_X^2<6$, which makes the classification of finite subgroups of $\Cr_2(\kk)$ possible. It seems curious to us, that for real del Pezzo surfaces one can get some bounds on $|G|$ independently of the Weyl groups. We shall not use the following result, but in our opinion it is worth mentioning. 

\begin{prop}
	Let $X$ be an $\RR$-rational del Pezzo surface of degree $d$ and $G\subset\Aut(X)$ be a finite group. Then one of the following holds.
	\begin{itemize}
		\item If $X\cong\PP_\RR^2(a,b)$, then
		\[
		|G|\leqslant 84(8-d)
		\]
		for $a\geqslant 2$. For $a=0$ the group $G$ is isomorphic to
		\begin{equation}\label{eq: sphere groups}
		\ZZ/n,\ \ \Dih_n,\ \ \Alt_4,\ \ \Sym_4\ \ \text{or}\ \ \Alt_5.
		\end{equation}
		For $a=1$ one has $G\cong(n\times m)_\bullet k$, where $k\in\{1,2,3,4,6\}$.
		\item If $X\cong\Quad_{3,1}(0,b)$, then $G\cong H_\bullet 2^r$, where $r\in\{0,1\}$ and $H$ belongs to the list (\ref{eq: sphere groups}).
		\item If $X\cong\Quad_{2,2}(0,b)$, then $G\cong((n\times m)_\bullet k)_\bullet 2^r$, where $r\in\{0,1\}$ and $k\in\{1,2,3,4,6\}$.
	\end{itemize}
\end{prop}
\begin{proof}
	We may assume that $G$ faithfully acts on $X(\RR)$ by diffeomorphisms. Let $X\cong\PP_\RR^2(a,b)$. Then $X(\RR)\approx\#_{a+1}\RP^2$. Denote its orientable double cover by $\Sigma_a$. By \cite[Corollary 9.4]{Bredon} we may assume that $G$ acts faithfully on $\Sigma_a$ by orientation-preserving diffeomorphisms. Take any Riemannian metric on $\Sigma_a$ and average it with respect to $G$ action. The resulting $G$-invariant metric gives a complex $G$-invariant structure on $\Sigma_a$, and $G$ can be regarded as a group of automorphisms of a Riemann surface of genus $a$.
	
	Therefore, for $a=0$ the group $G$ embeds into $\Aut(\Sigma_0)\cong\PSL_2(\CC)$. We recall its subgroups in Lemma \ref{lem: PGL subg} below. For $a=1$ the claim follows from a well-known classification of automorphisms of elliptic curves. Finally, for $a>1$ the Hurwitz theorem implies
	\[
	|G|\leqslant 84(a-1),
	\]
	so $a+2b=9-d$ gives the result. Let $X\cong\Quad_{3,1}(0,b)$ or $X\cong\Quad_{2,2}(0,b)$. Again, $G$ faithfully acts by diffeomorphisms of $X(\RR)$. Passing to an index 2 subgroup, we may assume that the action is orientation-preserving. Applying the same arguments as above, we finish the proof\footnote{See also Remark \ref{rem: smooth actions}.}.
\end{proof}

\subsection{Classical linear groups} The next result is classical and will be used throughout all the paper \textcolor{black}{(see e.g. \cite{Bli} or \cite{BeaPGL} for a modern treatment).}

\begin{lem}\label{lem: PGL subg} The following assertions hold.
	\begin{enumerate}\label{prop: PGL2 and PGL3}
		\item[(i)] Any finite subgroup of $\PGL_2(\CC)$ is one of the following:
		\begin{equation}\label{eq: Klein's groups}
		\ZZ/n,\ \ \Dih_n,\ n\geqslant 1,\ \ \Alt_4,\ \Sym_4,\ \Alt_5.
		\end{equation} 
		Any finite subgroup of $\GL_2(\RR)$ and $\PGL_2(\RR)$ is either cyclic or dihedral.
		\item[(ii)] One has $\PGL_3(\RR)\cong\SL_3(\RR)$. Any finite subgroup of $\PGL_3(\RR)$ is among the ones listed in (\ref{eq: Klein's groups}).
	\end{enumerate}
\end{lem}

Despite its simplicity, Lemma \ref{prop: PGL2 and PGL3} has important consequences for classification of finite subgroups of $\Cr_2(\RR)$ and, more generally, groups acting on real geometrically rational surfaces. For example, it ``kills'' almost all {\it simple} finite subgroups of $\Cr_2(\RR)$, see Appendix \ref{app: simple groups}.

\section{Del Pezzo surfaces of degree 8}\label{sec: dP8}

In this section $X$ denotes a real del Pezzo surface of degree 8. We shall assume that $\XC\cong\PP_\CC^1\times\PP_\CC^1$ (the other surface of degree 8, the blow up of $\PP_\RR^2$ at one point, is never $G$-minimal), so either $X\cong Q_{3,1}$ or $X\cong Q_{2,2}$ \cite[Lemma 1.16]{kol}. We treat these two cases separately. 

\vspace{0.3cm}

Let $X=\Quad_{3,1}$. Since $Q_{3,1}$ is $\RR$-minimal, any $G\subset\Aut(X)$ acts \textcolor{black}{strongly minimally} on $X$. By definition, $\Aut(X)=\PO(3,1)$, where
\[
\PO(3,1)=\OO(3,1)/\{\pm I\}.
\]
On the other hand,
\[
\OO(3,1)=\OO(3,1)^{\uparrow}\times\{\pm I\},
\]
where $\OO(3,1)^{\uparrow}$ is the subgroup preserving the future light cone. In particular, $\OO(3,1)^{\uparrow}\cong\PO(3,1)$ and we may identify subgroups of $\PO(3,1)$ with subgroups of the Lorentz group $\OO(3,1)$. Finite subgroups of $\OO(3,1)$ were classified in \cite{lorentz}. The authors also indicated the smallest of the five locally isomorphic Lorentz groups which contains each finite subgroup. The group $\OO(3,1)^{\uparrow}$ was denoted $\OO_1(3,1)$. To list the finite subgroups of $\OO(3,1)^{\uparrow}$ we then have to look at finite subgroups belonging to $\OO_1(3,1)$ and $\mathrm{DO}(3,1)$ in the notation of \cite{lorentz}. In turns out that all our subgroups belong to class (i) in the cited paper, i.e. we may assume that they consist of elements of the form $g\oplus 1$, where $g\in\OO_3(\RR)$ and 1 is the identity acting on the time coordinate. The classification of finite subgroups of $\OO_3(\RR)$ (or {\it point groups} in three dimensions) is a very classical topic and we do not give the whole list here (one can consult \cite[II]{Conway} or apply Goursat's lemma to $\OO_3(\RR)=\SO_3(\RR)\times\{\pm I\}$). For an explicit description of these groups by matrices we refer the reader to \cite{lorentz}. 

\begin{rem}\label{rem: smooth actions}
	One can give a topological explanation of the embedding $G\hookrightarrow\OO_3(\RR)$. Indeed, the group $G$ faithfully acts by diffeomorphisms of $\Quad_{3,1}(\RR)\approx\Sph^2$. By the classical theorem of Brouwer-Kerekjarto-Eilenberg, every such action is equivalent (i.e. conjugate) to a linear one, see e.g. \cite[\S 2]{Zimm}.
\end{rem}

Now let $X=\Quad_{2,2}$. Then $X\cong\PP_\RR^1\times\PP_\RR^1$ and
\[
\Aut(X)\cong\big (\PGL_2(\RR)\times\PGL_2(\RR)\big )\rtimes(\ZZ/2).
\]
\begin{prop}
	Let $G\subset\Aut(\Quad_{2,2})$ be a finite subgroup such that $\Pic(Q_{2,2})^G\cong\ZZ$. Then $G$ is isomorphic to one of the following groups (which are all strongly minimal): 
	\[(
	\ZZ/n\triangle_D \ZZ/n)_\bullet 2\cong (\ZZ/m\times\ZZ/k)_\bullet\ZZ/2,\ \ \ \big (\Dih_{n}\triangle_D \Dih_{n}\big )_\bullet 2.
	\]
\end{prop}
\begin{proof}
	The group $\widehat{G}=G\cap (\PGL_2(\RR)\times\PGL_2(\RR))$ naturally acts on the factors of $X=\PP_\RR^1\times\PP_\RR^1$ preserving them. Let $\widehat{G}_1$ and $\widehat{G}_2$ be the images of $\widehat{G}$ under the projections of $\PGL_2(\RR)\times\PGL_2(\RR)$ onto its factors. By Goursat's lemma, $\widehat{G}= \widehat{G}_1\triangle_D \widehat{G}_2$ for some $D$. As $\ZZ/2$-component of $\Aut(\Quad_{2,2})$ acts on $\PP_\RR^1\times\PP_\RR^1$ by switching the factors, the groups $\widehat{G}_1$ and $\widehat{G}_2$ must be isomorphic: otherwise $G=\widehat{G}$ and $\Pic(X)^G\cong\ZZ^2$, a contradiction. Thus $\widehat{G}\cong H\triangle_D H$, where $H$ is either cyclic, or dihedral. Note that a subgroup of a direct product of two cyclic groups is itself a direct product of at most two cyclic groups. Thus for $H$ cyclic one can also write $G\cong (\ZZ/m\times\ZZ/k)_\bullet\ZZ/2$, $m,k\geqslant 1$. For some isomorphic presentations of $\Dih_n\triangle_D\Dih_n$ see \cite[Theorem 4.9]{di}.
\end{proof}

\begin{rem}\label{rem: stereographic projection}
	Let $X$ be a real del Pezzo surface of degree 8 with $X_\CC\cong\PP_\CC^1\times\PP_\CC^1$, and $G\subset\Aut(X)$. If $G$ has a real fixed point $p$ on $X$, then $G$ is linearizable. Indeed, blowing up $p$ and contracting the strict transforms of the lines passing through $p$, we conjugate $G$ to a subgroup of $\Aut(\PP_\RR^2)$.
\end{rem}

\begin{prop}\label{prop: dP8 lineariz}
	\textcolor{black}{
		Let $X$ be a real del Pezzo surface of degree 8 and $G\subset\Aut(X)$ be a finite group with $\Pic(X)^G\simeq\ZZ$. Then $G$ is linearizable if and only if either $G$ has a real fixed point on $X$, or $G\simeq\Dih_5$ acts on $X=\Quad_{2,2}$. In particular, a linearizable group is either cyclic, or dihedral.}
\end{prop}
\begin{proof}
	\textcolor{black}{If $G$ has a real fixed point on $X$, then $G$ is linearizable by Remark \ref{rem: stereographic projection}. Assume there is a birational map $f: (X,G)\dashrightarrow (\PP^2_\RR,G)$ and run the Sarkisov program on $X$ to decompose $f$ into a product of Sarkisov links; in what follows we refer to \cite[Theorem 2.6]{isk-1} for description of these links (including group action in the picture is straightforward). The first link can connect $\DP_8$ either with some $\DP_*$ (link of type II) or with $\CB_2$ (link of type I; recall that here $2$ stands for the number of singular fibers). In the latter case we can continue making the links in the class $\CB$ (e.g. of type II or IV), without creating new singular fibers, but at some point we have to link a conic bundle with a del Pezzo surface $S$. Same theorem shows that $S\in\DP_8$. Since we do not want to return back to $\DP_8$, we may assume that the first link was actually of type II. In the diagram below we list all possibilities (regardless the base field or group action). We stop drawing arrows if we have to link our surface with some $\DP_*$ which already occurred in the diagram. The labels denote the degrees of points which we blow up.} 
	
	\vspace{0.3cm}
\[
\xymatrix@C+1pc{
						& \boxed{\DP_9\simeq\PP^2} &\\
						\boxed{\DP_8}\ar[r]^{\hspace{0.3cm}7,6,4}\ar[dr]^{5}\ar[ddr]_{3}\ar[ur]^{\hspace{0.7cm}1} & \DP_8 & \DP_5\\
						& \DP_5\ar[r]^{2}\ar[ur]^{4,3}\ar[dr]^{1} & \DP_8 \\
						& \DP_6\ar[d]_{5,4,3,2}\ar[dr]^{1} & \boxed{\DP_9\simeq\PP^2}\\
						& \DP_6 & \DP_8
				}
\]
	
	\textcolor{black}{So, we see that we have only two possibilities to connect $(X,G)$ with some $(\PP_\RR^2,G)$. The first one assumes that the link starts at a $G$-invariant point, which have to be real in our case. The second possibility is a combination of links of type II, namely $\DP_8\overset{5}{\longrightarrow}\DP_5 \overset{1}{\longrightarrow} \DP_9$. In particular, $G$ must have a real fixed point on $\DP_5$, and hence either $G\simeq\ZZ/5$ or $G\simeq\Dih_5$ (see Proposition \ref{prop: dP5}). In the first case $G$ must have a fixed point on $X$ and $X\simeq\Quad_{3,1}$ (see e.g. \cite[4.4]{Yas}). Let $G\simeq\Dih_5$, and the link $\DP_8\to\DP_5$ is as follows: $X\overset{f}{\longleftarrow} Z\overset{g}{\longrightarrow} Y$, where $X\in\DP_8$, $Y\in\DP_5$, $f$ is a blow up of the point $\eta$, $\deg\eta=5$, and $g$ is a contraction to a point $\xi$; note that $\deg\xi=2$ by \cite[Theorem 2.6]{isk-1}. We now use the linearization argument given in Section \ref{sec: dP5} below (or \cite[\S 4.6]{Yas}). If $X\simeq\Quad_{3,1}$, then $g$ contracts two conjugate $G$-orbits, so $\xi$ is a pair of conjugate $G$-fixed points, and we cannot proceed to $\PP_\RR^2$. If $X\simeq\Quad_{2,2}$, then $g$ contracts two real $G$-orbits, so $\xi$ is a pair of real $G$-fixed points. Such a group indeed can be further linearized.}
	
	Finally, if $G$ has a fixed point $p\in X(\RR)$ then there is a faithful linear representation $G\hookrightarrow\GL_2(T_pX)$, so $G$ is either cyclic or dihedral by Lemma \ref{lem: PGL subg}.
\end{proof}

\section{Del Pezzo surfaces of degree 6}\label{sec: dP6}

Let $X$ be a real del Pezzo surface of degree 6. Then $\XC$ can be obtained by blowing up $\PP_\CC^2$ in three noncollinear points $p_1,p_2,p_3$. The set of $(-1)$-curves on $\XC$ consists of six curves: the exceptional divisors of blow-up $e_i={\pi}^{-1}(p_i)$ and the strict transforms of the lines $d_{ij}$ passing through $p_i,\ p_j$. In the anticanonical embedding $\XC\hookrightarrow\PP_\CC^6$ these exceptional curves form a ``hexagon'' $\Sigma$. This yields a homomomorphism to the symmetry group of this hexagon
\[
\rho: \Aut(\XC)\to \Aut(\Sigma)\cong\Weyl(\A_1\times \A_2)\cong\D_{6},
\]

Since the set of all $(-1)$-curves on $X_\CC$ is defined over $\RR$, its complement $T$ is isomorphic to a torus over $\CC$. But $X(\RR)\ne\varnothing$, so $T$ is in fact an algebraic $\RR$-torus. One can view it as the connected component of the identity of $\Aut(X)$. There exist only 4 real forms of $\RR$-rational del Pezzo surfaces of degree 6: $\PP_\RR^2(3,0)$, $\PP_\RR^2(1,1)$, $\Quad_{3,1}(0,1)$, and $\Quad_{2,2}(0,1)$. They correspond to real forms of $T$ described by V.~E~.Voskresenkii \cite[10.1]{Voskr}.

\begin{center}
	\setlength{\extrarowheight}{5pt}
	\begin{table}[h!]
		\caption{Real forms of $\RR$-rational del Pezzo surfaces of degree 6}
		\label{table: dP6 real forms}
		\begin{tabular}{|c|c|c|c|c|c|c|lllllll}
			\cline{1-5}
			$\Gamma:\Sigma$ & id & Fig. \ref{fig:hexagon} (A) & Fig. \ref{fig:hexagon} (B) & Fig. \ref{fig:hexagon} (C) \\ \cline{1-5}
			$X$ & $\hspace{0.8cm}\PP_\RR^2(3,0)\hspace{1cm}$ & $\hspace{0.8cm}\Quad_{2,2}(0,1)\hspace{1cm}$ & $\hspace{0.8cm}\PP_\RR^2(1,1)\hspace{1cm}$ & $\hspace{0.8cm}\Quad_{3,1}(0,1)\hspace{1cm}$  \\ \cline{1-5}
			$X(\RR)$ & $\#4\RP^2$ & $\Torus^2$ & $\#2\RP^2$ & $\Sph^2$  \\ \cline{1-5} 
		\end{tabular}
	\end{table}
\end{center}

\begin{figure}[h!]
	\centering
	\begin{subfigure}{.35\textwidth}
		\centering
		\includegraphics[width=.6\linewidth]{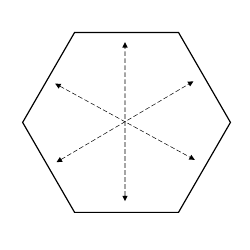}
		\caption{}
		\label{fig:a}
	\end{subfigure}%
	\begin{subfigure}{.35\textwidth}
		\centering
		\includegraphics[width=.6\linewidth]{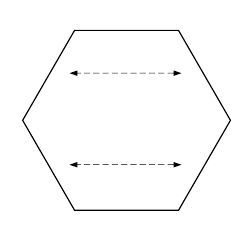}
		\caption{}
		\label{fig:b}
	\end{subfigure}%
	\begin{subfigure}{.35\textwidth}
		\centering
		\includegraphics[width=.6\linewidth]{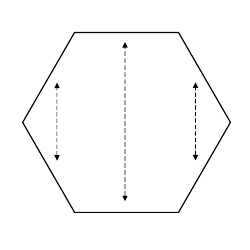}
		\caption{}
		\label{fig:c}
	\end{subfigure}
	\caption{Action of $\Gamma$ on $\Sigma$}
	\label{fig:hexagon}
\end{figure}

\begin{prop}\label{prop: dP6}
	Let $X$ be a real del Pezzo surface of degree 6 and $G\subset\Aut(X)$ be a finite group acting minimally on $X$. Then one of the following holds:
	\begin{enumerate}
		\item[(i)] The surface $X$ is isomorphic to $\Quad_{2,2}(2,0)\cong\PP_\RR^2(3,0)$ and can be given as
		\[
		\big\{([x_0:x_1:x_2],[y_0:y_1:y_2])\in\PP_\RR^2\times\PP_\RR^2:\ x_0y_0=x_1y_1=x_2y_2 \big\}
		\] 
		Its automorphism group fits into the short exact sequence 
		\[
		1\to\Ker\rho\to\Aut(X)\overset{\rho}{\to}\Dih_6\to 1.
		\]
		Here $\Ker\rho\cong(\RR^*)^2$ is the diagonal subgroup of $\PGL_3(\RR)$, and $\rho(\Aut(X))\cong\Dih_6$ is generated by the rotation $r=\rho(\alpha_1)$ and the reflection $s=\rho(\alpha_2)$, where 
		\begin{gather*}
		\alpha_1: ([x_0:x_1:x_2],[y_0:y_1:y_2])\mapsto ([y_2:y_0:y_1],[x_2:x_0:x_1]),\\
		\alpha_2: ([x_0:x_1:x_2],[y_0:y_1:y_2])\mapsto ([x_1:x_0:x_2],[y_1:y_0:y_2]).
		\end{gather*}
		The group $G$ is of the form
		\[
		(1a)\ \ H_\bullet\langle r\rangle\cong H_\bullet 6,\ \ \ \ \ \ (1b)\ \ H_\bullet\langle r^2,s\rangle\cong  H_\bullet\Sym_3,\ \ \ \text{or}\ \ \ (1c)\ \ H_\bullet\langle r,s\rangle\cong H_\bullet\Dih_6,
		\]
		where $H\subset\Ker\rho$ is isomorphic to a subgroup of $\ZZ/2\times\ZZ/2$.
		
		\item[(ii)] The surface $X$ is isomorphic to $\Quad_{2,2}(0,1)$ and can be given as
		\[
		\big\{([x_0:x_1],[y_0:y_1],[z_0,z_1])\in\PP_\RR^1\times\PP_\RR^1\times\PP_\RR^1:\ x_0y_0z_1+x_0y_1z_0+x_1y_0z_0-x_1y_1z_1=0 \big\}
		\] 
		Its automorphism group fits into the short exact sequence 
		\[
		1\to\Ker\rho\to\Aut(X)\overset{\rho}{\to}\Dih_6\to 1.
		\]
		Here $\Ker\rho\cong\SO_2(\RR)\times\SO_2(\RR)$, and $\rho(\Aut(X))\cong\Dih_6$ is generated by the rotation $r=\rho(\alpha_1)$ and the reflection $s=\rho(\alpha_2)$, where 
		\begin{gather*}
		\alpha_1: ([x_0:x_1],[y_0:y_1],[z_0:z_1])\mapsto ([z_1:z_0],[x_0:-x_1],[y_1:y_0]),\\
		\alpha_2: ([x_0:x_1],[y_0:y_1],[z_0:z_1])\mapsto ([y_0:y_1],[x_0:x_1],[z_0:z_1]).
		\end{gather*}
		The group $G$ is one of the following:
		\begin{gather*}
		(2a)\ \ H_\bullet\langle r\rangle\cong H_\bullet 6,\ \ 
		(2b)\ \ H_\bullet\langle r^2\rangle\cong H_\bullet 3,\ \ 
		(2c)\ \ H_\bullet\langle r^2,s\rangle\cong H_\bullet\Sym_3,\\ \ \ 
		(2d)\ \ H_\bullet\langle r^2,rs\rangle\cong H_\bullet\Sym_3,\ \
		(2e)\ \ H_\bullet\langle r,s\rangle\cong H_\bullet\Dih_6,
		\end{gather*}
		where $H\subset\Ker\rho$ is a direct product of at most 2 cyclic groups of an arbitrary large order.
	\end{enumerate}
\textcolor{black}{All listed groups do act minimally on the corresponding real surfaces.}
\end{prop}
\begin{proof}
All statements about automorphism groups of real del Pezzo surfaces of degree 6 and their equations can be found in \cite[Section 3]{RobZim}. Moreover, for $X=\Quad_{3,1}(0,1)$ or $\PP_\RR^2(1,1)$ the pair $(X,\Aut(X))$ is not minimal, so we may assume that $X=\Quad_{2,2}(0,1)$ or $X=\PP_\RR^2(3,0)$. 
Up to conjugacy, the group $\rho(G)\subset\Dih_6=\langle r,s:\ r^6=s^2=1,\ srs^{-1}=r^{-1}\rangle$ is one of the following:
\begin{itemize}
	\item cyclic: $\langle r^k\rangle$, $\langle s\rangle$, $\langle rs\rangle$, $k=0,1,2,3$;
	\item dihedral: $\langle r,s\rangle$, $\langle r^2,s\rangle$, $\langle r^2,rs\rangle$, $\langle r^3,s\rangle$.
\end{itemize}

{\bf Case $X=\PP_\RR^2(3,0)$}. All $(-1)$-curves on $X$ are real. Thus a cyclic group $\rho(G)\cong\langle r^k\rangle$ acts minimally on $X$ if and only if $k=1$ (otherwise one can $G$-equivariantly contract an orbit which consists of disjoint $(-1)$-curves and is defined over $\RR$). Following the same argument, it is easy to check that in the dihedral case only $\langle r^2,s\rangle$, and hence $\langle r,s\rangle$, act minimally on $X$. As any nontrivial finite subgroup of $\RR^*$ is isomorphic to $\ZZ/2$, we get the result.

\hspace{0.3cm}

{\bf Case $X=\Quad_{2,2}(0,1)$}. The action of $\Gamma$ on the hexagon is shown on Figure \ref{fig:hexagon}. Examining the action of $G$ on $\Sigma$, one easily gets that only the groups $\langle r\rangle$, $\langle r^2\rangle$, $\langle r^2,s\rangle$, $\langle r^2,rs\rangle$, $\langle r,s\rangle$ act minimally on $X$. 
\end{proof}

\begin{prop}\label{prop: dP6 linearization}
	Let $X$ be a real del Pezzo surface of degree 6, and $G\subset\Aut(X)$ be a finite group acting minimally on $X$. Assume that $G$ is linearizable. Then $G$ is one of the following groups (in the notation of Proposition \ref{prop: dP6}):
	\begin{itemize}
		\item isomorphic to $\Sym_4$: (1b) and (2c), where $H$ is a Klein 4-group;
		\item isomorphic to $\Alt_4$: (2b), where $H$ is a Klein 4-group;
		\item dihedral: 
		
		$\Dih_3\cong\Sym_3$: (1b), (2c), (2d); 
		
		$\Dih_6$: (1c), (1b), (2c), (2d), (2e); 
		
		$\Dih_{12}$: (1c), (2c), (2d), (2e); 
		
		$\Dih_{3k}$, $k\geqslant 2$: (2c), (2d); 
		
		$\Dih_{6k}$, $k\geqslant 2$: (2e).
		\item cyclic: 
		
		(1a): $\ZZ/6$ and $\ZZ/12$;
		
		(2a): $\ZZ/6k$;
		
		(2b): $\ZZ/3k$.
	\end{itemize}
\end{prop}
\begin{proof}
	This is an elementary group theory. As $\Alt_5$ is simple, none of the groups from Proposition \ref{prop: dP6} is isomorphic to $\Alt_5$. Let $G\cong\Sym_4$. Note that $\Sym_4$ has no normal subgroups $H$ with $G/H$ isomorphic to $\ZZ/3$, $\ZZ/6$ or $\Dih_{6}$. If $\Sym_4/H\cong\Sym_3$, then $H=\{e,(12)(34),(13)(24),(14)(23)\}$ is a Klein group. Let $G\cong\Alt_4$. Note that $\Alt_4$ has no normal subgroups $H$ with quotient isomorphic to $\ZZ/6$, $\Sym_3$ or $\Dih_6$. If $\Alt_4/H\cong\ZZ/3$, then $H$ is a Klein group.
	
	Let $G\cong\Dih_n$. We know that $G$ has a normal subgroup $H$ with $G/H$ isomorphic to $\ZZ/3$, $\ZZ/6$, $\Sym_3$ or $\Dih_6$. In particular $H$ is cyclic (otherwise $[G:H]\leqslant 2$). On the other hand, a quotient of a dihedral group is again dihedral. In the case (1) of Proposition \ref{prop: dP6} we get that for $H=\id$ the group $G$ is $\Dih_3$ (1b) or $\Dih_6$ (1c), while for $H\cong\ZZ/2$ the group $G$ is $\Dih_6$ (1b) or $\Dih_{12}$ (1c). In the case (2) the cyclic group $H$ can be of any order $k$, so either $G\cong\Dih_{3k}$ and is of type (2c), (2d), or $G\cong\Dih_{6k}$ and is of type (2e).
	
	Finally, let $G\cong\ZZ/n$. Then $H$ is cyclic. In the case (1) of Proposition \ref{prop: dP6} one has $|H|\leqslant 2$, and $G/H\cong\ZZ/6$. Thus $G\cong\ZZ/6$ or $\ZZ/12$. In the case (2) the order of $H$ can be arbitrary large, hence $G$ is isomorphic to $\ZZ/3k$ or $\ZZ/6k$.
\end{proof}

\begin{rem}
	As was shown in \cite[\S 4.5]{Yas} there exist infinitely many non-linearizable subgroups of type (2b) acting minimally on $\Quad_{2,2}(0,1)$. Moreover, we exhibited two non-conjugate embeddings of $G=(\ZZ/3)^2$ into $\Cr_2(\RR)$: the one is a trivial extension of type (2b), and the other comes from the fiberwise $G$-action on the conic bundle $X=\Quad_{2,2}\cong\PP_\RR^1\times\PP_\RR^1$ with $\rk\Pic(X)^G=2$.  
\end{rem}

\section{Del Pezzo surfaces of degree 5}\label{sec: dP5}

Each real del Pezzo surface $X$ of degree 5 is isomorphic to $\PP^2_\RR(a,b)$, where $(a,b)\in\{(4,0),(2,1),(0,2)\}$ \cite[Corollary 5.4]{kol}. There are 10, 4 or 2 real lines on $X$ respectively. It is clear from the blow-up model of $X$ that the configuration of $\Gamma$-orbits of exceptional curves is uniquely determined by the pair $(a,b)$. The incidence graph of such a configuration is the colored Petersen graph, where the lines in one $\Gamma$-orbit have the same color (and we additionally label by $*$ the real ones). We assume that $X$ is the blow-up of $\PP_\RR^2$ at four points $p_1,p_2,p_3,p_4$ in general position, $e_i$ is the exceptional divisor over the point $p_i$ and $d_{ij}$ is the proper transform of the line passing through the points $p_i$ and $p_j$, see Figure \ref{fig: petersen}.

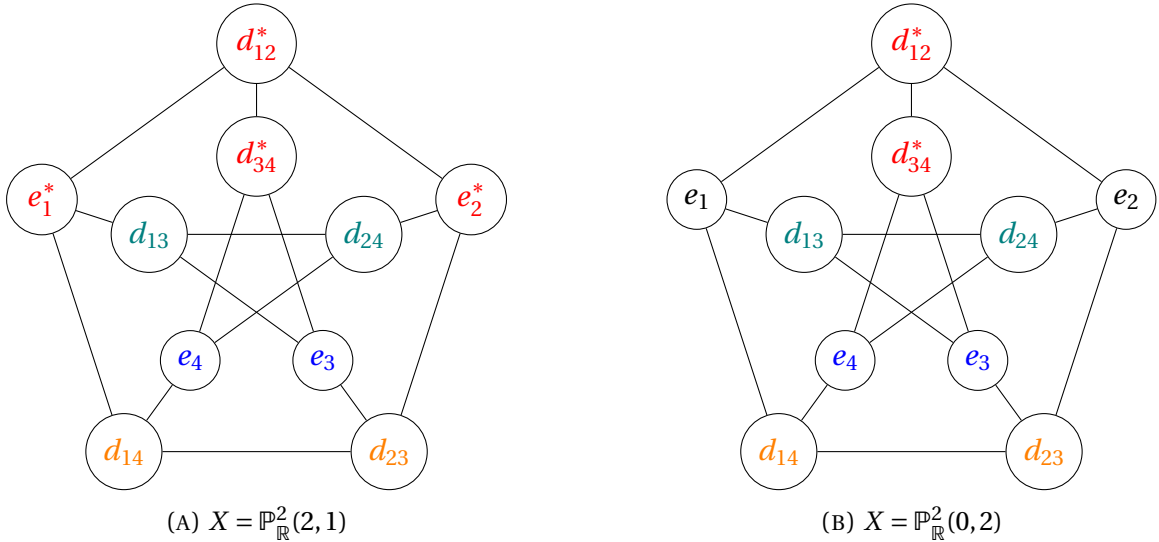
\begin{figure}[h!]
	\centering
	\begin{subfigure}{.49\textwidth}
		\centering
		\begin{tikzpicture}[every node/.style={draw,circle}]
		\graph [clockwise,math nodes] {     
			subgraph C [V={ \textcolor{red}{d_{12}^*}, \textcolor{red}{e_2^*}, \textcolor{orange}{d_{23}}, \textcolor{orange}{d_{14}}, \textcolor{red}{e_1^*} }, name=A, radius=3cm]; 
			subgraph N [V={ \textcolor{red}{d_{34}^*}, \textcolor{teal}{d_{24}}, \textcolor{blue}{e_3}, \textcolor{blue}{e_4}, \textcolor{teal}{d_{13}} }, name=B, radius=1.5cm];
			\foreach \i [evaluate={\j=int(mod(\i+1,5)+1);}] in {1,...,5}{
				A \i -- B \i;  
				B \i -- B \j;
			}
		}; 
		\end{tikzpicture}
		\caption{$X=\PP_\RR^2(2,1)$}
		\label{fig:dP5-4}
	\end{subfigure}
	\begin{subfigure}{.49\textwidth}
		\centering
		\begin{tikzpicture}[every node/.style={draw,circle}]
		\graph [clockwise,math nodes] {     
			subgraph C [V={ \textcolor{red}{d_{12}^*}, \textcolor{black}{e_2}, \textcolor{orange}{d_{23}}, \textcolor{orange}{d_{14}}, \textcolor{black}{e_1} }, name=A, radius=3cm]; 
			subgraph N [V={ \textcolor{red}{d_{34}^*}, \textcolor{teal}{d_{24}}, \textcolor{blue}{e_3}, \textcolor{blue}{e_4}, \textcolor{teal}{d_{13}} }, name=B, radius=1.5cm];
			\foreach \i [evaluate={\j=int(mod(\i+1,5)+1);}] in {1,...,5}{
				A \i -- B \i;  
				B \i -- B \j;
			}
		}; 
		\end{tikzpicture}
		\caption{$X=\PP_\RR^2(0,2)$}
		\label{fig:dP5-2}
	\end{subfigure}
	\caption{Graph of $(-1)$-curves on del Pezzo surface $X$ of degree 5}
	\label{fig: petersen}
\end{figure}

Let us do some extra work and find all possibilities for $\Aut(X)$.

\begin{prop}
	Let $X$ be a real del Pezzo surface of degree 5. Then
	\begin{itemize}
		\item $\Aut(X)\cong\Sym_5$ if $X\cong\PP_\RR^2(4,0)$,
		\item $\Aut(X)\cong\ZZ/2\times\ZZ/2$ if $X\cong\PP_\RR^2(2,1)$,
		\item $\Aut(X)\cong\Dih_4$ if $X\cong\PP_\RR^2(0,2)$.
	\end{itemize}  
\end{prop}
\begin{proof}
	The ``split'' case $X\cong\PP_\RR^2(4,0)$ is classical and can be found e.g. in \cite[Theorem 8.5.8]{cag}. Denote by $\Pi_{a,b}$ the colored incidence graph of $(-1)$-curves on $X_\CC=\PP_\RR^2(a,b)\otimes\CC$. As $\Aut(X)$ naturally acts on the	exceptional lines preserving incidence relations, we have a homomorphism $\psi: \Aut(X)\to\Aut(\Pi_{a,b})$. It is injective, as any automorphism of $X_\CC$ which fixes all $(-1)$-curves comes from an automorphism of $\PP_\RR^2$ that fixes 4 closed points $p_i$'s, so it must be trivial. 
	
	Note that for each $\varphi\in\Aut(\Pi_{a,b})$ and any two vertexes $v_1$ and $v_2$ we must have: (i) if $\{v_1,v_2\}$ is $\Gamma$-invariant then $\{\varphi(v_1),\varphi(v_2)\}$ is $\Gamma$-invariant; (ii) if $v_1$ and $v_2$ are incident then $\varphi(v_1)$ and $\varphi(v_2)$ are incident. Put
	\begin{gather*}
	\alpha:\ \ \ e_1\leftrightarrow e_2,\ \  d_{13}\leftrightarrow d_{24},\ \  d_{14}\leftrightarrow d_{23},\ \ e_3\leftrightarrow e_4,\\
	\beta:\ \ \ e_1\leftrightarrow e_2,\ \  d_{13}\leftrightarrow d_{23},\ \  d_{14}\leftrightarrow d_{24},\\
	\varsigma:\ \ e_3\leftrightarrow e_4,\ d_{14}\leftrightarrow d_{13},\ d_{24}\leftrightarrow d_{23},\\
	\varrho:\ \ d_{12}\leftrightarrow d_{34},\ e_1\mapsto e_3,\ e_2\mapsto e_4,\ e_3\mapsto e_2,\ e_4\mapsto e_1,\ d_{14}\mapsto d_{13},\ d_{23}\mapsto d_{24},\ d_{13}\mapsto d_{23},\ d_{24}\mapsto d_{14}.
	\end{gather*}
	(if a line is not indicated then it is stabilized). Note that $\alpha,\beta\in\Aut(\Pi_{2,1})$ and $\varsigma,\varrho\in\Aut(\Pi_{0,2})$. Then
	\begin{gather*}
	\Aut(\Pi_{2,1})=\langle\alpha\rangle_2\times\langle\beta\rangle_2\cong\ZZ/2\times\ZZ/2,\\
	\Aut(\Pi_{0,2})=\langle \varsigma,\varrho\ |\ \varrho^4=\varsigma^2=\id,\ \varsigma^{-1}\varrho\varsigma=\varrho^{-1}\rangle\cong\Dih_4.
	\end{gather*}
	Indeed, in the case of $\Pi_{0,2}$ one can use that $\Aut(\Pi_{0,2})$ acts on the set $\{e_1,e_2,e_3,e_4\}$ and the kernel of this action is obviously trivial. On the other hand, $\Aut(\Pi_{0,2})$ cannot be isomorphic to $\Sym_4$, as any automorphism of order 3 would fix $d_{12}$ (hence $e_1$ and $e_2$), $d_{34}$ (hence $e_3$ and $e_4$). Since $\Aut(\Pi_{0,2})$ contains $\Dih_4$, we get $\Aut(\Pi_{0,2})\cong\Dih_4$. The case of $\Pi_{2,1}$ is easy as well.
	
	To show that $\psi$ is surjective we explicitly construct the corresponding geometric actions. For this set
	\begin{gather*}
	\alpha': [x:y:z]\mapsto [x:y:-z],\ \ \ \beta': [x:y:z]\mapsto [-x:y:z],\ \ \ \varsigma'=\alpha',\ \ \
	\varrho': [x:y:z]\mapsto [z:y:-x].
	\end{gather*}
	We may also assume (after applying a suitable transformation from $\PGL_3(\RR)$) that the blown up points are
	\begin{gather*}
	p_1=[1:0:1],\ \ p_2=[1:0:-1],\ \ p_3=[0:1:i],\ \ p_4=[0:1:-i],\ \ \text{when}\ X=\PP_\RR^2(2,1),\\
	p_1=[1:i:0],\ \ p_2=[1:-i:0],\ \ p_3=[0:1:i],\ \ p_4=[0:1:-i],\ \ \text{when}\ X=\PP_\RR^2(0,2).
	\end{gather*}
	Then the lifts of $\alpha',\ \beta',\ \varsigma'$ and $\varrho'$ act as $\alpha,\ \beta,\ \varsigma$ and $\varrho$ respectively on the corresponding $\Pi_{a,b}$.
\end{proof}

\begin{prop}\label{prop: dP5}
	Let $X$ be a real del Pezzo surface of degree 5 and $G\subset\Aut(X)$ be a finite group acting minimally on $X$. Then $X$ is isomorphic to $\PP^2_\RR(4,0)$, and the group $G$ is one of the following:
		\[
		\Sym_5,\ \ \Alt_5,\ \ \ZZ/5\rtimes\ZZ/4=\langle a,b\ |\ a^5=b^4=1, bab^{-1}=a^2\rangle,\ \ \Dih_5,\ \ \ZZ/5.
		\]
	\textcolor{black}{All listed groups do act minimally on $X$.}
\end{prop}
\begin{proof}
	In the case $X=\PP_\RR^2(4,0)$ we argue exactly as if $\kk=\CC$, see \cite[Theorem 6.4]{di}. Assume that $X\cong\PP_\RR^2(2,1)$. Note that the curve $d_{12}$ is the only line on $X$ intersecting 3 real lines. Thus it is stabilized by $\Aut(X)$ and can be equivariantly contracted, implying that the pair $(X,\Aut(X))$ is not minimal. Now let $X\cong\PP_\RR^2(0,2)$. Then every automorphism in $\Aut(X)$ preserves the set $\{e_1,e_2,e_3,e_4\}$ consisting of 2 pairs of complex conjugate lines which are pairwise disjoint, and hence can be equivariantly contracted. So, $(X,\Aut(X))$ is not minimal. 
\end{proof}

Let $S=\PP_\RR^2(4,0)$. It follows from the classification of Sarkisov links that for $G=\Sym_5$ or $\Alt_5$ the pair $(S,G)$ is superrigid, see \cite[Proposition 7.12,7.13]{di}. Let $G=\langle a\rangle\cong\ZZ/5$. Then $S(\CC)^a$ consists of 2 points, whose blow-up is a del Pezzo surface $Y$ of degree 3 with two skew lines $\ell_1$ and $\ell_2$, either real or complex conjugate \cite[\S 4.6]{Yas}.  One can use the $G$-birational map
\[
\ell_1\times\ell_2\dashrightarrow Y,\ \ (p_1,p_2)\mapsto q,\ \ \ \text{where}\ Y\cap\overline{p_1p_2}=\{p_1,p_2,q\}
\]
to conjugate $G$ to a group acting on a quadric surface $Q$. If $\sigma(\ell_1)=\ell_2$, then $Q\cong\Quad_{3,1}$ and $G$ can be further linearized \cite[Proposition 4.18]{Yas}. Now let $G$ be $\ZZ/5\rtimes\ZZ/4=\langle a,b\ |\ a^5=b^4=1, bab^{-1}=a^2\rangle$ or $\ZZ/5\rtimes\ZZ/2=\langle a,b\ |\ a^5=b^2=1, bab^{-1}=a^{-1}\rangle\cong\Dih_5$. Then $S(\CC)^a$ consists of 2 points and this set is $b$-invariant. Therefore, we can again use the same birational map as above to conjugate $G$ to a group acting on a quadric surface (this is a Sarkisov link of type II).

\section{Del Pezzo surfaces of degree 4}\label{sec: dP4}

\subsection{Topology and equations} Throughout this section $X$ denotes a real del Pezzo surface of degree~4. It is well-known that the linear system $|-K_X|$ embeds $X$ into $\PP^4_\RR$ as a complete intersection of two quadrics, which we denote $Q_0$ and $Q_\infty$. If no confusion arises, we denote by the same letter a quadric, the corresponding quadratic form and its matrix. Let $\QQ$ be the pencil \[\lambda Q_0(x_0,\ldots,x_4)+\mu Q_\infty(x_0,\ldots,x_4).\] Its discriminant $\Delta(\mu,\lambda)\equiv\det(\lambda Q_0+\mu Q_\infty)$ is a binary form of degree 5. Since we assume $X$ smooth, the equation $\Delta=0$ has five distinct roots $[\lambda_i:\mu_i],\ i=1,\ldots,5$. Equivalently, the matrix $Q_0^{-1}Q_\infty$ (we may suppose $Q_0$ nonsingular) has five distinct eigenvalues $-\lambda_i/\mu_i\in\CC$ . They correspond to the singular members of $\QQ$, which we denote by $Q_i$, $i=1,\ldots,5$. 

Note that eigenspaces corresponding to different eigenvalues are orthogonal with respect to both $Q_0$ and $Q_\infty$. Over $\CC$, we can find a basis of eigenvectors, making both $Q_0$ and $Q_\infty$ diagonal, so the pencil takes the form
\[
\sum_{i=0}^{4}(\lambda a_i+\mu b_i)x_i^2
\]
with $b_i/a_i=-\lambda_i/\mu_i$. 

The complex conjugation permutes the eigenspaces. In a $\Gamma$-invariant one, we can pick a real vector for our basis, so the corresponding part of the pencil's equation has real coefficients $a_i$ and $b_i$. For two complex conjugate eigenspaces, we get a two-dimensional real subspace $W$ orthogonal to the other eigenspaces. If we pick an orthogonal basis $\{w,\overline{w}\}$ in $W\otimes\CC$, where $w$ is an eigenvector with eigenvalue $-b/a$, then
\[
Q(z_1w+z_2\overline{w})=(\lambda a+\mu b)z_1^2+(\lambda \overline{a}+\mu \overline{b})z_2^2.
\]
Clearly, $z_1w+z_2\overline{w}\in W$ if and only if $z_1=\overline{z}_2$. Put
\[
z_1=u+iv,\ \ \ \ a=a_1+ia_2,\ \ \ \ b=b_1+ib_2.
\]
to get
\[
2(\lambda a_1+\mu b_1)(u^2-v^2)-4(\lambda a_2+\mu b_2)uv.
\]
Set
\[
-b/a=\alpha+i\beta,\ \ \ a=i/2
\]
to obtain the normal form
\[
\mu\beta(u^2-v^2)+2(\lambda-\alpha\mu)uv.
\]
Let us summarize this discussion by stating the following classification result.
\begin{prop}\label{prop: dP4 equations}
	A real del Pezzo surface of degree 4 can be reduced to one of the following normal forms
	\begin{description}\setlength\itemsep{1em}
		\item[(I)]\ \
		$
		\begin{cases} 
			a_0x_0^2+a_1x_1^2+a_2x_2^2+a_3x_3^2+a_4x_4^2=0  \\ 
			b_0x_0^2+b_1x_1^2+b_2x_2^2+b_3x_3^2+b_4x_4^2=0 
		\end{cases}
		$ \hspace{3.6cm} in $\ \Proj\RR[x_0,x_1,x_2,x_3,x_4]$;
		\item[(II)]\ 
		$
		\begin{cases} 
			a_0x_0^2+a_1x_1^2+a_2x_2^2+2u_1v_1=0  \\ 
			b_0x_0^2+b_1x_1^2+b_2x_2^2+\beta_1(u_1^2-v_1^2)-2\alpha_1 u_1v_1=0 
		\end{cases}
		$ \hspace{1.8cm} in $\ \Proj\RR[x_0,x_1,x_2,u_1,v_1]$;
		\item[(III)]\
		$
		\begin{cases}
		a_0x_0^2+2u_1v_1+2u_2v_2=0\\ 
		b_0x_0^2+\beta_1(u^2_1-v^2_1)-2\alpha_1 u_1v_1+\beta_2(u^2_2-v^2_2)-2\alpha_2u_2v_2=0 
		\end{cases}
		$ in $\ \Proj\RR[x_0,u_1,v_1,u_2,v_2]$,
		\vspace{0.3cm}
	\end{description}
	where $-b_i/a_i$ and $-(\alpha_i\pm i\beta_i)$ are eigenvalues of $\QQ$.
\end{prop}

Now let us describe how the topology of $X(\RR)$ depends on the equation of $X$. Nonsingular real pencils of quadrics were classified by C. T. C. Wall in \cite{WallPolytopes} by an invariant called characteristic. In the notation of Proposition \ref{prop: dP4 equations} set
\[
a_k=r_k\cos\theta_k,\ \ \ \ b_k=r_k\sin\theta_k,\ \ \ \ r_k>0.
\]
and define points on the circle
\[
P_k=(\cos\theta_k,\sin\theta_k),\ \ \ Q_k=-P_k.
\]
These points can be grouped in blocks: as we proceed anticlockwise around the circle we meet a block of $m_1$ points $P_t$, then a block of $n_1$ points $Q_t$, then a block of $m_2$ points $P_t$ and so on. When we are half way round, we meet an opposite block of $m_1$ points $Q_t$, so one has $m_1=n_{g+1}$ for some $g\geqslant 0$. This $g$ is called the {\it genus} and the sequence $(m_1,\ldots,m_{2g+1})$ in cyclic order the {\it characteristic} $\Xi(\QQ)$ of our pencil. Below we list some information about the topology and real lines on $X$, following \cite{WallPolytopes}, \cite{Wall} and \cite{kol} (we list only those surfaces which are rational over $\RR$). Using Proposition \ref{prop: dP4 equations}, for each real form we also indicate the type of equation of $X$.

\begingroup
\renewcommand*{\arraystretch}{1.4}
\begin{longtable}{|c|c|c|c|c|c|c|c|c|}
	\caption{Real forms of $\RR$-rational del Pezzo surfaces of degree 4}
	\label{table: dP4 real forms} \\
	\hline
	Class of $\sigma^*\in\Weyl(\D_5)$ & Eigenvalues of $\sigma^*$ & $\Xi(\QQ)$ & Equation type & $X$ & $X(\RR)$ & $\#$ real lines \\ \hline
	$\id$ & $1^5$ & (1,1,1,1,1) & I & $\PP_\RR^2(5,0)$ & $\#6\RP^2$ & 16  \\ \hline
	$A_1$ & $-1,1^4$ & (1,1,1) & II & $\PP_\RR^2(3,1)$ & $\#4\RP^2$ & 8  \\ \hline
	$A_1^2$ & $-1^2,1^3$ & (1) & III & $\PP_\RR^2(1,2)$ & $\#2\RP^2$ & 4  \\ \hline
	${A_1^2}'$ & $-1^2,1^3$ & (2,2,1) & I & $\Quad_{2,2}(0,2)$ & $\Torus^2$ & 0  \\ \hline
	$A_1^3$ & $-1^3,1^2$ & (3) & II & $\Quad_{3,1}(0,2)$ & $\Sph^2$ & 0  \\ \hline
\end{longtable}
\endgroup

\begin{rem}
	Note that the sum of entries in $\Xi(\QQ)$ equals to the number of real eigenvalues of $\QQ$. In particular, there is {\it no} one-to-one correspondence between the numbers of real eigenvalues of $\QQ$ and the real structures on $X$. 
\end{rem}

\subsection{Automorphisms}

Let $v_i$ and $Q_i^b$ denote the vertex and the base of the singular quadric $Q_i$ respectively. Since $\Gamma$ acts on the set $\{v_1,v_2,v_3,v_4,v_5\}$, there can be 1, 3 or 5 real $v_i$'s. As $Q_i^b\otimes\CC$ has two pencils of lines, each $Q_i$ has two pencils of planes, whose intersections with $X_\CC$ give two complementary pencils of conics $\CCC_i$ and $\CCC_i'$ on $X_\CC$. These pencils satisfy the conditions $\CCC_i\cdot\CCC_i'=2$, $\CCC_i\cdot\CCC_j=\CCC_i\cdot\CCC_j'=1$ for $i\ne j$, and $\CCC_i+\CCC_i'\sim -K_X$. Two complementary pencils define a double cover $\pi_i: X_\CC\to\PP_\CC^1\times\PP_\CC^1$, which coincides with the projection of $X$ from $v_i$. Depending on the type of real locus $Q_i^b(\RR)$ (i.e. on the realness of two pencils of lines on $Q_i^b$) one has either $\sigma(\CCC_i)=\CCC_i$, $\sigma(\CCC_i')=\CCC_i'$ (if $Q_i^b\cong\Quad_{2,2}$), or $\sigma(\CCC_i)=\CCC_i'$ (if $Q_i^b\cong\Quad_{3,1}$). 

The Galois involution of the double cover $\pi_i$ induces an automorphism $\tau_i\in\Aut(X_\CC)$. For real $v_i$ both $\pi_i$ and $\tau_i$ are defined over $\RR$. As was explained in the beginning of this section, in a suitable system of complex coordinates both $Q_0$ and $Q_\infty$ can be brought to diagonal form, so the equations of $X$ can be written in the form
\[
\sum_{i=0}^{4}x_i^2=\sum_{i=0}^{4}\theta_i x_i^2=0,
\]
and then $\tau_i$ are given by $x_i\mapsto -x_i$. These five commuting involutions generate a normal abelian subgroup $A\subset\Aut(X_\CC)$ with a unique relation $\tau_1\tau_2\tau_3\tau_4\tau_5=\id$, hence
\[
A=\{1,\tau_k,\tau_i\tau_j:\ 1\leqslant k\leqslant 5,1\leqslant i<j\leqslant 5 \},\ \ \ \ A\cong(\ZZ/2)^4.
\]
In what follows it will be convenient for us to use the following description of this group, see \cite[Lemma 9.11]{blancabelian}:
\[
A=\left\{a=(a_1,a_2,a_3,a_4,a_5)\in(\ZZ/2)^5:\ \sum_{i=1}^{5}a_i=0\right \},
\]
where an element $(a_1,\ldots,a_5)$ exchanges the two conic bundles $\CCC_i$ and $\CCC_i'$ if $a_i=1$ and preserves each one if $a_i=0$. In this terminology, the automorphism $a=(a_1,a_2,a_3,a_4,a_5)$ corresponds to the projective transformation 
\begin{equation}\label{eq: a in coordinates}
[x_1:x_2:x_3:x_4:x_5]\mapsto [(-1)^{a_1}x_1:(-1)^{a_2}x_2:(-1)^{a_3}x_3:(-1)^{a_4}x_4:(-1)^{a_5}x_5],
\end{equation}
so $\tau_1$ corresponds to $(0,1,1,1,1)$, $\tau_2$ corresponds to $(1,0,1,1,1)$ etc.

Further, the groups $\Aut(X)$ and $\Aut(X_\CC)$ act on the pencil $\QQ$ preserving the set of five degenerate quadrics or, equivalently, the set of pairs $\RRR_i=\{\CCC_i,\CCC_i'\}$. Thus we have two homomorphisms
\begin{equation}\label{eq: dP4 hom}
\rho_1:\ \Aut(X_\CC)\to \PGL_2(\CC),\ \ \ \rho_2: \Aut(X_\CC)\to\Sym_5
\end{equation}
with $\ker\rho_1=\ker\rho_2=A$. In fact, the exact sequence 
\[
\id\to A\to\Aut(X_\CC)\to\Imag\rho_2\to\id
\]
splits, and $\Aut(X_\CC)\cong A\rtimes\Imag\rho_2$. One can easily see \cite[Section 6]{di} that $\Aut(X_\CC)/A$ is one of the following groups:
\begin{equation}\label{eq: dP4: image}
\id,\ \ \ZZ/2,\ \ \ZZ/3,\ \ \ZZ/4,\ \ \ZZ/5,\ \ \Sym_3,\ \ \Dih_5.
\end{equation}
Denote by $\rho$ the restriction of $\rho_2$ on the {\it real} automorphism group $\Aut(X)$. Set
\[
A_o=\Ker\rho=A\cap\Aut(X),\ \ \ A'=\Imag(\rho).
\]
\begin{convention*}
	In this paragraph every permutation $\tau\in\Sym_5$ should be understood as a permutation of the set $\{\RRR_i:\ i=1,\ldots,5\}$. For an automorphism $(a,\tau)\in\Aut(X_\CC)$ we denote it simply by $a$ if $\tau=\id$, and by $\tau$ if $a=0$.
\end{convention*}

\subsection{Groups acting minimally on real del Pezzo quartics}

We now start to enumerate the groups acting minimally on real del Pezzo surfaces of degree 4. For each real form listed in Table \ref{table: dP4 real forms}, we first get some restrictions on the groups $A_o$ and $A'$, and then list possible strongly minimal groups $G\subset\Aut(X)$. We describe a way to get an explicit group structure of $G$ (e.g. to list all its elements $(a,\tau)$), and do this job ourselves for subgroups of $A_o$. To shorten the exposition, we do not treat systematically minimal groups of {\it mixed type}. i.e. those which are not contained in $A_o$ or $A'$, but demonstrate how one can get such description in Proposition \ref{prop: dP4 sphere case} (case $G\simeq\ZZ/4$).

It is straightforward to write down all these automorphisms in coordinates using (\ref{eq: a in coordinates}). To write down the equation of $X$, one may use Proposition \ref{prop: dP4 equations} choosing coefficients in accordance with characteristic $\Xi(\QQ)$, see Table \ref{table: dP4 real forms}.  

The split case $X\cong\PP_\RR^2(5,0)$ immediately follows from the work of Dolgachev and Iskovskikh \cite{di}, as $\sigma^*=\id$ and the whole groups $(\ZZ/2)^4$ and $\Sym_5$ act by real transformations of $\PP_\RR^4$.

\begin{prop}[{\cite[Theorem 6.9]{di}}]
	Let $X=\PP_\RR^2(5,0)$ be a real del Pezzo surface of degree 4, and $G\subset\Aut(X)$ be a group acting \textcolor{black}{strongly minimally} on $X$. Then $G$ is isomorphic to one of the following:
	\begin{gather*}
	(\ZZ/2)^k,\ \ \ k=2,3,4;\ \ \ \ZZ/2\times\ZZ/4,\ \ \Dih_4,\ \ L_{16},\ \ (\ZZ/2)^4\rtimes\ZZ/2,\ \
	\ZZ/8,\ \ M_{16},\\ \ (\ZZ/2)^4\rtimes\ZZ/4,\ \ (\ZZ/2)^2\rtimes\ZZ/3,\
	\ZZ/2\times\Alt_4,\ \ (\ZZ/2)^4\rtimes\ZZ/3,\ \ \ZZ/3\rtimes\ZZ/4,\ \ (\ZZ/2)^k\rtimes\Sym_3,\ \ k=2,3,4;\\ L_{16}\rtimes\ZZ/3,\ \
	(\ZZ/2)^4\rtimes\Dih_5,\ \ (\ZZ/2)^4\rtimes\ZZ/5,
	\end{gather*}
	where $L_{16}=\langle a,b,c\ |\ a^4=b^2=c^2=[c,a]b=[a,b]=[c,b]=1\rangle$ and $M_{16}=\langle a,b,c\ |\ a^8=b^2=[a,b]a^4=1\rangle$ are non-abelian groups of order 16. \textcolor{black}{Moreover, all listed groups act strongly minimally on $X$.}
\end{prop}

We now proceed with non-trivial real forms of del Pezzo quartics. Let $X=\Quad_{3,1}(0,2)$ be the blow up of $\Quad_{3,1}$ at four points $p,\ \overline{p},\ q,\ \overline{q}$. Denote by $E_x$ the exceptional divisor over a point $x\in\Quad_{3,1}$, and by $F$ the strict transform of a fiber. Then, in the notation as above, one has
\begin{gather*}
\{\CCC_1,\CCC_1'\}=\{F+\overline{F}-E_p-E_{\overline{p}},\  F+\overline{F}-E_q-E_{\overline{q}}\},\\
\{\CCC_2,\CCC_2'\}=\{F+\overline{F}-E_p-E_q,\  F+\overline{F}-E_{\overline{p}}-E_{\overline{q}}\},\\
\{\CCC_3,\CCC_3'\}=\{F+\overline{F}-E_p-E_{\overline{q}},\  F+\overline{F}-E_q-E_{\overline{p}}\},\\
\{\CCC_4,\CCC_4'\}=\{F,\  F+2\overline{F}-E_p-E_{\overline{p}}-E_q-E_{\overline{q}}\},\ \ \ \ \ \ \ \ \ 
\{\CCC_5,\CCC_5'\}=\{\overline{F},\  2F+\overline{F}-E_p-E_{\overline{p}}-E_q-E_{\overline{q}}\}.
\end{gather*}
Note that in each pair one has $\CCC_i+\CCC_i'= -K_X=2F+2\overline{F}-E_p-E_{\overline{p}}-E_q-E_{\overline{q}}$. In what follows we shall depict the action of $\sigma$ on five pairs of conic bundles like this: 
\[\xymatrix@C+1pc{
	\CCC_1\bullet & \CCC_2\bullet\ar@{<->}[d] & \CCC_3\bullet\ar@{<->}[d] & \CCC_4\bullet\ar@{<->}[r] & \bullet\CCC_5 \\
	\CCC_1'\bullet & \CCC_2'\bullet & \CCC_3'\bullet & \CCC_4'\bullet\ar@{<->}[r] & \bullet\CCC_5' \\	
}
\]
This example corresponds to description of pairs $\RRR_i=\{\CCC_i,\CCC_i' \}$ for $X=\Quad_{3,1}(0,2)$ given above. No arrow means that the corresponding conic bundle is $\sigma$-invariant. We shall omit the bullets' labels in the future. Now it is easy to see that $A'\subset\{\id, (23), (45), (23)(45)\}$ (however this inclusion is strict, as one can see from the list (\ref{eq: dP4: image})), and any element $(a_1,a_2,a_3,a_4,a_5)\in\Ker\rho$ has $a_4=a_5$, so $A_o$ embeds into $(\ZZ/2)^3$ (here and below we often use the fact that $\Gamma$ commutes with automorphism). 

\begin{prop}\label{prop: dP4 sphere case}
	Let $X\cong\Quad_{3,1}(0,2)$ be a real del Pezzo surface of degree 4, and $G\subset\Aut(X)$ be a group acting \textcolor{black}{strongly minimally} on $X$. Then
the kernel $A_o$ of $\rho: \Aut(X)\to\Sym_5$ is isomorphic to $(\ZZ/2)^3$, and generated by elements $\gamma_1=(0,1,1,0,0)$, $\gamma_2=(1,0,1,0,0)$ and $\gamma_3=(0,0,0,1,1)$. Further, the image $A'$ of $\rho$ is either $\langle (23)(45)\rangle$ or trivial. Finally, the group $G$ is one of the following:
		\[
		\ZZ/2,\ \ \ (\ZZ/2)^2,\ \ \ (\ZZ/2)^3,\ \ \ZZ/4,\ \ \ (\ZZ/2)^2\rtimes\ZZ/2,\ \ \ (\ZZ/2)^3\rtimes\ZZ/2.
		\]
		More precisely, the first group $\ZZ/2$ is generated by either $(\gamma_2+\gamma_3,\id)$, or $(\gamma_1+\gamma_2+\gamma_3,\id)$. All other groups, except $\ZZ/4$, contain at least one of these elements. The group $\ZZ/4$ is generated by $\big ((1,0,1,1,1), (23)(45)\big)$. The first three groups lie in $\Ker\rho$. \textcolor{black}{Finally, all listed groups indeed act strongly minimally on $X$.}
\end{prop}
More information about the structure of the last three groups is given in the proof.
\begin{rem}
	The case when $X\cong\Quad_{3,1}$ and $G$ is a group of prime order was investigated in \cite[\S 4.3]{rob}. It was shown that 
	\begin{enumerate}
		\item $X$ can be given by the equations
		\[
		\begin{cases}
		(\mu-\mu\overline{\mu}+\overline{\mu})x_1^2-2x_1x_2+x_2^2+(1-\overline{\mu}+\mu\overline{\mu}-\mu)x_3^2+x_4^2=0,\\
		\mu\overline{\mu}x_1^2-2\mu\overline{\mu}x_1x_2+(\mu-1+\overline{\mu})x_2^2+\mu\overline{\mu}x_4^2+(1-\overline{\mu}+\mu\overline{\mu}-\mu)x_5^2.
		\end{cases}
		\]
		\item $A_o$ is isomorphic to $(\ZZ/2)^3$, and generated by elements $\gamma_1$, $\gamma_2$ and $\gamma_3$. 
		\item $A'$ is either $\langle (23)(45)\rangle$ or trivial. Moreover, the former happens if and only if $|\mu|=1$.
	\end{enumerate}
To save some space, we shall use these results below referring to \cite{rob} for their proofs . 
\end{rem}
\begin{proof}[Proof of Proposition \ref{prop: dP4 sphere case}]
	In the light of the previous Remark, we may proceed with determining minimal groups. In what follows we denote the elements of $\Ker\rho$ as $a\equiv(a,\id)$, where $a=(a_1,\ldots,a_5)\in~(\ZZ/2)^5$, and the elements of $\Imag\rho$ are denoted as $\tau\equiv(0,\tau)$, $\tau\in\Sym_5$. Below we shall use the following trivial observation several times. Assume that elements of $G$
	\begin{itemize}
		\item either all have $a_1=0$;
		\item or all have $a_4=a_5=0$.
	\end{itemize}
	Then $G$ is not \textcolor{black}{strongly minimal}. Indeed, in the first case $G$ fixes $\sigma$-invariant $\CCC_1$ and $\CCC_1'$, hence we have $\rk\Pic(X)^G>1$. In the second case $G$ fixes $F+\overline{F}$, which is not a multiple of $K_X$; hence $\rk\Pic(X)^G>1$. For brevity, we will call any of the two conditions above a {\it $\star$-condition} (it will be always clear from the context which one we actually mean).
	\vspace{0.3cm} 
	
	{\sc Case $G\subset\Ker\rho$} 
	
	\vspace{0.3cm} 
	
	Assume $G=\langle g=(a,\id)\rangle$ acts \textcolor{black}{strongly minimally}. By the previous remark $a_1=1$ and $a_4=a_5=1$. We see that $g$ is one of the following elements: $\alpha_1=(1,0,1,1,1)$ or $\alpha_2=(1,1,0,1,1)$. One can easily write down these automorphisms in homogeneous coordinates of $\PP^4_\RR$, see \cite[Proposition 4.11]{rob}. In particular, the whole group $G=\Ker\rho$ acts \textcolor{black}{strongly minimally} on $X$.  
	
	Now assume $G\cong(\ZZ/2)^2$. We need to consider only those $G$ which do not contain $\alpha_1$ or $\alpha_2$. Denote by $\beta_1=(0,1,1,1,1)$ and $\beta_2=(1,1,0,0,0)$ remaining non-trivial elements of $\Ker\rho$. Then $G$ is one of the following:
	\[
	G_1=\langle \beta_1,\gamma_1\rangle=\{\id,\beta_1,\gamma_1,\gamma_3\},\ \ \ G_2=\langle \beta_2,\gamma_1\rangle=\{\id,\gamma_1,\gamma_2,\beta_2\}.
	\]
	Each element of $G_1$ has $a_1=0$, so $G_1$ is not \textcolor{black}{strongly minimal}. On the other hand, each element of $G_2$ has $a_4=a_5=0$, $G_2$ is not \textcolor{black}{strongly minimal} either. 
		
	\vspace{0.3cm} 
	
	{\sc Case $G\nsubseteq\Ker\rho$} 
	
	\vspace{0.3cm} 
	
	We start with the case of a cyclic group. $G=\langle g\rangle$. Set $\tau=(23)(45)$. Clearly, $g=(0,\tau)$ does not act \textcolor{black}{strongly minimally}, so we may suppose
	$
	g=(a,\tau),
	$
	$a\ne 0$. Again we must have $a_1=a_4=a_5=1$, so $g$ has the form
	\begin{gather*}
	\big ((1,0,1,1,1),\tau\big ),\ \ \text{or}\ \ \big ((1,1,0,1,1),\tau\big )=\big ((1,0,1,1,1),\tau\big )^3.
	\end{gather*}
	In particular, $g$ is of order 4, and we may assume that we are in the first case. A simple calculation shows that the action of $g$ and $\sigma$ on $\Pic(X)=\langle F,\overline{F}, E_p,\overline{E}_p, E_q,\overline{E}_q\rangle\cong\ZZ^6$ is
	\[
	g^*=
	\begin{pmatrix}
	2 & 1 & 1 & 1 & 1 & 1\\
	1 & 2 & 1 & 1 & 1 & 1\\
	-1 & -1 & -1 & -1 & -1 & 0\\
	-1 & -1 & -1 & -1 & 0 & -1\\
	-1 & -1 & 0 & -1 & -1 & -1\\
	-1 & -1 & -1 & 0 & -1 & -1
	\end{pmatrix}\ \ \ \text{and}\ \ 
	\sigma^*=\begin{pmatrix}
	0 & 1 & 0 & 0 & 0 & 0\\
	1 & 0 & 0 & 0 & 0 & 0\\
	0 & 0 & 0 & 1 & 0 & 0\\
	0 & 0 & 1 & 0 & 0 & 0\\
	0 & 0 & 0 & 0 & 0 & 1\\
	0 & 0 & 0 & 0 & 1 & 0
	\end{pmatrix}
	\]
	From this one can easily get that $\rk\Pic(X_\CC)^{\Gamma\times G}=(\sum_{h\in\Gamma\times G}\tr(h^*))/|\Gamma\times G|=1$, so $G$ is \textcolor{black}{strongly minimal}.
	
	Now suppose that $G$ is not cyclic and set $G_0=G\cap\Ker\rho$. We may assume that $|G_0|$ equals 2 or 4, as otherwise $G$ contains $\Ker\rho$ and hence is \textcolor{black}{strongly minimal}. First consider the case $|G_0|=2$. Since we already considered the cyclic case, we may assume that $G=\langle g\rangle\times\langle h\rangle$, where $h=(0,\tau)$ and $g=(a,\id)$. The condition $gh=hg$ gives $(a,\tau)=(\tau\cdot a,\tau)$, which means $a_2=a_3$, $a_4=a_5$. From $a_1+a_2+a_3=0$ we obtain $a_1=0$ which implies the same for all elements of $G$, showing that we have a $\star$-condition.
	
	Finally, suppose $G_0=\langle (a,\id)\rangle\times\langle (b,\id)\rangle\cong\ZZ/2\times\ZZ/2$ and put $H=\langle(0,\tau)\rangle$. The condition $G_0H=HG_0$ implies that the set $\big\{(a,\tau),(b,\tau),(a+b,\tau)\big\}$ coincides with $\big \{(\tau\cdot a,\tau),(\tau\cdot b,\tau),(\tau\cdot (a+b),\tau)\big\}$. Let us examine possible cases.
	\begin{description}
		\item[\sc Case $\tau\cdot a=a$] Then $a_4=a_5$, and $a_2=a_3$ implies $a_1=0$. We have the following possibilities for~$a$:
		\[
		\text{(i)}\ \ (0,0,0,1,1),\ \ \ \ \ \text{(ii)}\ \ (0,1,1,0,0),\ \ \ \ \ \text{(iii)}\ \ (0,1,1,1,1).
		\]
		\begin{description}
			\item[\sc Subcase $\tau\cdot b=b$] As above, this implies $b_1=0$, and we are in the $\star$-condition, contradicting minimality. 
			\item[\sc Subcase $\tau\cdot b=a+b$] According to each possibility for $a$, we have
			\begin{enumerate}
				\item[(i)] $(b_1,b_3,b_2,b_5,b_4)=(b_1,b_2,b_3,b_4+1,b_5+1)$, so $b_4+b_5=2b_4+1=1$, a contradiction.
				\item[(ii)] $(b_1,b_3,b_2,b_5,b_4)=(b_1,b_2+1,b_3+1,b_4,b_5)$, so $0=b_1+b_2+b_3=b_1+2b_2+1$ implies $b_1=1$. As we may assume $b_4=b_5=1$ (otherwise the $\star$-condition is satisfied), there are only two possibilities for $b$:
				\[
				\alpha_1=(1,0,1,1,1)\ \ \ \ \text{or}\ \ \ \ \alpha_2=(1,1,0,1,1).
				\]
				Both these elements indeed give minimal automorphisms, as was noticed in the very beginning of the proof. 
				\item[(iii)] $(b_1,b_3,b_2,b_5,b_4)=(b_1,b_2+1,b_3+1,b_4+1,b_5+1)$. Then $b_4+b_5=2b_4+1=1$, a contradiction. 
			\end{enumerate}
		\end{description}
	\item[\sc Case $\tau\cdot a=b$] We have $a_1=b_1$, $a_2=b_3$, $a_3=b_2$, $a_4=b_5$, $a_5=b_4$. The $\star$-condition implies that one may assume $a_1=b_1=1$ and $a_4=b_4=a_5=b_5=1$. We conclude that $\{a,b \}=\{\alpha_1,\alpha_2\}$, so $G$ is minimal. 
	\item[\sc Case $\tau\cdot a=a+b$] is obtained from the first one by switching the roles of $a$ and $b$.
	\end{description}
\textcolor{black}{Finally, let us stress once again that all groups of Proposition \ref{prop: dP4 sphere case} are defined over $\RR$ and strongly minimal.}
\end{proof}

We next pass to the case when $X\cong\PP_\RR^2(a,b)$. 

\begin{prop}
	Let $X\cong\PP_\RR^2(1,2)$ be a real del Pezzo surface of degree 4. Then $A_o\cong(\ZZ/2)^2$ and $A'$ lies in the group
	\[
	\langle s=(23), r=(2435)\ |\ s^2=r^4=1, srs=r^{-1} \rangle\cong\Dih_4.
	\]
	There are no strongly minimal groups $G$ acting on $X$. 
\end{prop}

\begin{proof}
	One has
	\begin{gather*}
	\{\CCC_1,\CCC_1'\}=\{L-E_1,-K_X-L+E_1\};\\
	\{\CCC_2,\CCC_2'\}=\{L-E_2,-K_X-L+E_2\};\ \ \
	\{\CCC_3,\CCC_3'\}=\{L-\overline{E}_2,-K_X-L+\overline{E}_2\};\\
	\{\CCC_4,\CCC_4'\}=\{L-E_3,-K_X-L+E_3\};\ \ \
	\{\CCC_5,\CCC_5'\}=\{L-\overline{E}_3,-K_X-L+\overline{E}_3\}.
	\end{gather*}
	The complex involution acts as
	\[\xymatrix@C+1pc{
	\bullet & \bullet\ar@{<->}[r] & \bullet & \bullet\ar@{<->}[r] & \bullet \\
	\bullet & \bullet\ar@{<->}[r] & \bullet & \bullet\ar@{<->}[r] & \bullet \\	
	}
	\]
	This immediately gives the statement about $A'$. Moreover, for any element $(a_1,a_2,a_3,a_4,a_5)\in A_o$ one has $a_2=a_3$ and $a_4=a_5$. Thus $A_o$ is a subgroup of
	\[
	\{(0,0,0,0,0),\ (0,0,0,1,1),\ (0,1,1,0,0),\ (0,1,1,1,1)\}\cong(\ZZ/2)^2,
	\]
	and $a_1$ always equals to $0$. This implies that both $\CCC_1$ and $\CCC_1'$ are (real and) $G$-invariant for any group $G$, hence $\rk\Pic(X)^G>1$.
\end{proof}

\begin{prop}\label{prop: dP4 case 3,1}
	Let $X\cong\PP_\RR^2(3,1)$ be a real del Pezzo surface of degree 4. Then $A_o$ lies inside the group
	\[
	\{(a_1,a_2,a_3,a_4,a_5)\in(\ZZ/2)^5:\ a_4+a_5=a_1+a_2+a_3=0\}\cong(\ZZ/2)^3,
	\] 
	and $A'$ is a subgroup of
	\[
	\Symm\{\RRR_1,\RRR_2,\RRR_3\}\times\Symm\{\RRR_4,\RRR_5 \}\cong\Sym_3\times\ZZ/2\cong\Dih_6
	\]
	isomorphic to $\id$, $\ZZ/2$, $\ZZ/3$ or $\Sym_3$. Moreover, each group acting \textcolor{black}{strongly minimally} on $X$ is either contained in $A_o$ and isomorphic to
	\[
	(\ZZ/2)^2,\ \ \ (\ZZ/2)^3,
	\]
	or is an extension of such a group by a subgroup of $A'$ isomorphic to $\ZZ/2$, $\ZZ/3$ or $\Sym_3$ (\textcolor{black}{and every such group actually occurs as a strongly minimal group}), or is a group of mixed type of order $>2$.
\end{prop}
\begin{proof}
	We have
	\begin{gather*}
	\{\CCC_1,\CCC_1'\}=\{L-E_1,-K_X-L+E_1\};\ \ \
	\{\CCC_2,\CCC_2'\}=\{L-E_2,-K_X-L+E_2\};\ \ \
	\{\CCC_3,\CCC_3'\}=\{L-E_3,-K_X-L+E_3\};\\
	\{\CCC_4,\CCC_4'\}=\{L-E_4,-K_X-L+E_4\};\ \ \
	\{\CCC_5,\CCC_5'\}=\{L-\overline{E}_4,-K_X-L+\overline{E}_4\}.
	\end{gather*}
	The complex involution acts as
	\[\xymatrix@C+1pc{
		\bullet & \bullet\ar@{}[r] & \bullet & \bullet\ar@{<->}[r] & \bullet \\
		\bullet & \bullet\ar@{}[r] & \bullet & \bullet\ar@{<->}[r] & \bullet \\	
	}
	\]
	Thus, for any element $(a_1,a_2,a_3,a_4,a_5)\in A_o$ one has $a_4=a_5$ and $A_o\subset(\ZZ/2)^3$ is the group described in the statement. Embedding $A'\hookrightarrow\Dih_6$ indicated therein is clear as well; to exclude some possibilities for $A'$ one consults (\ref{eq: dP4: image}).
	
	Now let $G\subset\Aut(X)$ be a strongly minimal group. Note that $G\nsubseteq A'$. Indeed, otherwise $\CCC_1+\CCC_2+\CCC_3=3L-E_1-E_2-E_3$ is defined over $\RR$, $G$-invariant and not a multiple of $-K_X$. 
	
	Assume that $G\subset A_o$. Note that $G$ cannot be of order 2, as every element $g=((a_1,a_2,a_3,a_4,a_5),\id)\in A_o$ has $a_i=0$ for some $i\in\{1,2,3 \}$, so that the corresponding $\CCC_i$ is $\langle g\rangle$-invariant. By the same principle, every strongly minimal group $G\subset A_o$ of order 4 must consist of elements which do not share $0$ on the same $i$-th place with $i\in\{1,2,3 \}$. This leaves the following possibilities for $G$:
	\begin{gather*}
	G_o^1=\{(0,0,0,0,0),(0,1,1,0,0),(1,0,1,1,1),(1,1,0,1,1)\},\\
	G_o^2=\{(0,0,0,0,0),(1,1,0,0,0),(0,1,1,1,1),(1,0,1,1,1)\},\\
	G_o^3=\{(0,0,0,0,0),(1,0,1,0,0),(0,1,1,1,1),(1,1,0,1,1)\}.
	\end{gather*}
	All these groups are in fact \textcolor{black}{strongly minimal}. To check this, it is convenient to assume that $\Pic(X_\CC)\otimes\RR$ is spanned by $e_0=-K_X$, $e_i=\CCC_i$, $i=1,\ldots, 5$. In this basis, the actions of $a=(a_1,a_2,a_3,a_4,a_5)$ and $\sigma\circ a$ are given by
	\begin{equation*}
	a^*=
	\begingroup 
	\setlength\arraycolsep{2pt}
	\begin{pmatrix}
	1 & a_1 & a_2 & a_3 & a_4 & a_5\\
	0 & (-1)^{a_1} & 0 & 0 & 0 & 0\\
	0 & 0 & (-1)^{a_2} & 0 & 0 & 0\\
	0 & 0 & 0 & (-1)^{a_3} & 0 & 0\\
	0 & 0 & 0 & 0 & (-1)^{a_4} & 0\\
	0 & 0 & 0 & 0 & 0 & (-1)^{a_5}
	\end{pmatrix}\ \ \ \ \
	\sigma^* a^*=\begin{pmatrix}
	1 & a_1 & a_2 & a_3 & a_4 & a_5\\
	0 & (-1)^{a_1} & 0 & 0 & 0 & 0\\
	0 & 0 & (-1)^{a_2} & 0 & 0 & 0\\
	0 & 0 & 0 & (-1)^{a_3} & 0 & 0\\
	0 & 0 & 0 & 0 & 0 & (-1)^{a_5}\\
	0 & 0 & 0 & 0 & (-1)^{a_4} & 0
	\end{pmatrix}
	\endgroup
	\end{equation*}
	Thus for a group $G\subseteq A_o$ one has
	\begin{gather*}
	\rk\Pic(X_\CC)^{\Gamma\times G}=1+\frac{1}{2|G|}\sum_{h\in\Gamma\times G}\tr h^*=1+\frac{1}{2|G|}\sum_{a\in G} \big (2(-1)^{a_1}+2(-1)^{a_2}+2(-1)^{a_3}+(-1)^{a_4}+(-1)^{a_5}\big )=\\
	=1+\frac{1}{2|G|}(2(\delta_0-\delta_1)+(\varepsilon_0-\varepsilon_1)),
	\end{gather*}
	where $\delta_i$ is the total number of $i$'s occurring at the first three positions of all $a\in G$, and $\varepsilon_i$ is the total number of $i$'s occurring at the last two positions of $a$. Hence $G$ is strongly minimal if and only if $2(\delta_0-\delta_1)+(\varepsilon_0-\varepsilon_1)=0$. It is now straightforward to check the latter condition for $G_o^1$, $G_o^2$ and $G_o^3$, so all listed groups do act strongly minimally on $X$. Arguments, similar to the ones in Proposition \ref{prop: dP4 sphere case}, show that there are no involutions of mixed type acting strongly minimally on $X$.
\end{proof}

\begin{prop}
	Let $X\cong\Quad_{2,2}(0,2)$ be a real del Pezzo surface of degree 4, and $G\subset\Aut(X)$ be a \textcolor{black}{strongly minimal} group. Then $A'$ is a subgroup of
	\[
	\Symm\{\RRR_1,\RRR_2,\RRR_3\}\times\Symm\{\RRR_4,\RRR_5 \}\cong\Sym_3\times\ZZ/2\cong\Dih_6
	\]
	isomorphic to $\id$, $\ZZ/2$, $\ZZ/3$ or $\Sym_3$. Moreover, each \textcolor{black}{strongly minimal} group is either contained in $A_o$ and isomorphic to
	$
	(\ZZ/2)^k,
	$
	$k=1,2,3,4$,
	or is an extension of such a group by a subgroup of $A'$ isomorphic to $\ZZ/2$, $\ZZ/3$ or $\Sym_3$ (and all listed groups indeed occur as strongly minimal groups), or is a group of mixed type. More information about the structure of $G$ is given in the proof. 
\end{prop}
\begin{proof}
	As above, denote by $E_x$ the exceptional divisor over a point $x\in\Quad_{2,2}$, and by $F_1$ and $F_2$ the strict transforms of fibers. Then one has
	\begin{gather*}
	\{\CCC_1,\CCC_1'\}=\{F_1+F_2-E_p-E_{\overline{p}},\  F_1+F_2-E_q-E_{\overline{q}}\},\\
	\{\CCC_2,\CCC_2'\}=\{F_1,\  F_1+2F_2-E_p-E_{\overline{p}}-E_q-E_{\overline{q}}\},\ \ \ \
	\{\CCC_3,\CCC_3'\}=\{F_2,\  2F_1+F_2-E_p-E_{\overline{p}}-E_q-E_{\overline{q}}\},\\
	\{\CCC_4,\CCC_4'\}=\{F_1+F_2-E_p-E_q,\  F_1+F_2-E_{\overline{p}}-E_{\overline{q}}\},\ \ \ \
	\{\CCC_5,\CCC_5'\}=\{F_1+F_2-E_p-E_{\overline{q}},\  F_1+F_2-E_q-E_{\overline{p}}\}.	
	\end{gather*}
	The complex involution acts as
	\[\xymatrix@C+1pc{
		\bullet & \bullet & \bullet & \bullet\ar@{<->}[d] & \bullet\ar@{<->}[d] \\
		\bullet & \bullet & \bullet & \bullet & \bullet \\	
	}
	\]
	and the same reasoning as in the proof of Proposition \ref{prop: dP4 case 3,1} applies to $A'$ to get restrictions on this group. We proceed with enumerating minimal subgroups of $A_o$. One can use the same basis for $\Pic(X_\CC)\otimes\RR$ as in the proof of Proposition \ref{prop: dP4 case 3,1} and see that $a^*$ remains unchanged, while $\sigma^*a^*$ sends $e_i$ to $(1-a_i)e_0+(-1)^{a_i+1}e_i$ for $i=4,5$. Thus for $G\subset A_o$ we have
	\begin{gather*}
	\rk\Pic(X_\CC)^{\Gamma\times G}=1+\frac{1}{2|G|}\sum_{a\in G} \big (2(-1)^{a_1}+2(-1)^{a_2}+2(-1)^{a_3}+(-1)^{a_4}+(-1)^{a_5}+(-1)^{a_4+1}+(-1)^{a_5+1}\big )=\\
	=1+\frac{1}{|G|}\sum_{a\in G} \big ((-1)^{a_1}+(-1)^{a_2}+(-1)^{a_3}\big )=1+\frac{1}{|G|}(\delta_0-\delta_1),
	\end{gather*}
	implying that $G$ is strongly minimal if and only if $\delta_0=\delta_1$. We leave to the interested reader to write down all the possibilities for such $G$.
	 
\end{proof}

\addtocontents{toc}{\protect\setcounter{tocdepth}{1}}
\subsection{Conjugacy problem}

It follows from the classification of Sarkisov links \cite[Theorem 2.6]{isk-1} that if $G$ has no real fixed points on $X$, then every $G$-link starting on $X$ is of type II and leads to the same (isomorphic) surface. Therefore, the conjugacy class of $G$ in $\Cr_2(\RR)$ is determined by the conjugacy class of $G$ in $\Aut(X)$; we do not carry out the detailed classification of the latter ones here. For each real form, some information about $A_o$ and $A'$  can be found in the previous section. 

If $G$ has a real fixed point on $X$, then its blow up induces a link of type I and gives a $G$-minimal conic bundle of degree 3. Note that according to Proposition \ref{prop: PGL2 and PGL3} and the description of possible groups from the previous section, $G$ can be one of the following:
\[
\ZZ/2,\ \ (\ZZ/2)^2,\ \ \ZZ/4,\ \ \ZZ/8,\ \ \Dih_4.
\]

\addtocontents{toc}{\protect\setcounter{tocdepth}{2}}
\section{Del Pezzo surfaces of degree 3: cubic surfaces}\label{sec: dP3}

Before stating our main result, let us briefly recall some facts about real cubic surfaces. It is classically known that any smooth complex cubic surface contains exactly 27 lines. It is probably less known that in \cite{Segre} Segre divided real lines on smooth real cubic into two species called {\it  hyperbolic} and {\it elliptic}. Consider a real line on the cubic surface. Any plane passing through this line intersects the surface in the line itself and a further residual conic. This conic intersects the line in
two points. Define an involution on the line by exchanging these two points of intersection. The fixed points of this involution are called {\it parabolic points}. The real line is called {\it hyperbolic} if the involution has two real parabolic
points. The real line is called {\it elliptic} if it has a pair of complex conjugate parabolic points. In the following table we collect the information about possible types of $\sigma^*\in\Weyl(\E_6)$ (including Segre's notation $F_i$), real lines, tritangent planes\footnote{Recall that a tritangent plane is a plane intersecting a smooth cubic surface in three lines.} and topology of $X(\RR)$.

\begingroup
\renewcommand*{\arraystretch}{1.2}
\begin{longtable}{|c|c|c|c|c|}
	\caption{Real lines and real structures on cubic surfaces}
	\label{table: dP3 real lines} \\
	\hline
	$\sigma^*$ & \cite{Segre} & $\#$ of real lines/tritangent planes & $\#$ of elliptic/hyperbolic lines & Topology of $X(\RR)$ \\ \hline
	
	$\id$ & $F_1$ & (27,45) & 12/15 & $\#7\RP^2$ \\ \hline
	
	$A_1$ & $F_2$ & (15,15) & 6/9 & $\#5\RP^2$ \\ \hline
	
	$A_1^2$ & $F_3$ & (7,5) & 2/5 & $\#3\RP^2$ \\ \hline
	
	$A_1^3$ & $F_4$ & (3,7) & 0/3 & $\RP^2$ \\ \hline
	
	$A_1^4$ & $F_5$ & (3,13) & 0/3 & $\RP^2\sqcup\Sph^2$ \\ \hline
\end{longtable}
\endgroup

This information can be used to get some restrictions on group actions on $X$.

\begin{lem}\label{lem: p-groups cubic surface}
	Let $G$ be a $p$-group, where $p$ is any prime number, and $X$ be a real $\RR$-rational cubic surface. Then $X$ is not $G$-minimal.
\end{lem}
\begin{proof}
	It was shown in \cite{Yas} that a group of odd order cannot act minimally on a $\RR$-rational cubic surface. On the other hand, a 2-group cannot act minimally on a real cubic surface, as there exists an invariant real line. 
\end{proof}

\begin{lem}\label{lem: dP3 7 lines}
	There are no finite groups acting minimally on a real cubic surface with $\sigma^*$ of type $A_1^2$.
\end{lem}
\begin{proof}
	According to Table \ref{table: dP3 real lines} such a surface contains 2 elliptic lines. Assume that $G$ acts minimally on $X$. Then it permutes elliptic lines which must intersect at a point. The plane passing through these lines intersects $X$ in the third real line, which must be $G$-invariant, a contradiction.
\end{proof}

\begin{lem}\label{lem: dP3 3 lines lemma}
	Let $X$ be a real cubic surface with $\sigma^*$ of type $A_1^3$ or $A_1^4$, and $G\subset\Aut(X)$ be a finite group acting minimally on $X$. Then $G$ is one of the following groups:
	\begingroup
	\renewcommand*{\arraystretch}{1.4}
	\begin{longtable}{llcc}
		$k=0:$ & $\Sym_3$, \\ 
		$k=1:$ & $\ZZ/6$,\ \ $\Dih_6$,\ \ $\BDih_6\cong\ZZ/3\rtimes\ZZ/4$,\\
		$k=2:$ & $\Alt_4$,\ \ $\ZZ/6\times\ZZ/2$,\ \ $\Dih_6\times\ZZ/2$,\ \ $\Sym_4$,\ \ $\Alt_4\times\ZZ/2$,\ \ $\BDih_6\times\ZZ/2$,
	\end{longtable}
	\endgroup
	(see the meaning of $k$ in the proof).
\end{lem}
\begin{proof}
	If $\sigma^*$ is of type $A_1^3$ or $A_1^4$, then $X$ has exactly 3 real lines. In both cases these lines form a triangle, possibly a degenerate one (i.e. all lines meet at one Eckardt point). Indeed, in the first case this is obvious, as $X$ dominates $\PP_\RR^2$. In the second case $X$ is non-rational over $\RR$, so it cannot contain two disjoint real lines.
	
	So, the group $G$ acts on the set of three real lines, say $\ell_1$, $\ell_2$, $\ell_3$, and one has a homomorphism $\delta: G\to\Sym_3$. The minimality condition implies that $\Imag\delta$ contains an element of order 3. The kernel $\Ker\delta$ consists of automorphisms that preserve each $\ell_i$, and in particular either fix three real points $p=\ell_1\cap\ell_2$, $\ell_1\cap\ell_3$, $\ell_2\cap\ell_3$, or preserve the unique point $p=\ell_1\cap\ell_2\cap\ell_3$. In both cases $\Ker\delta$ embeds into $\GL(T_pX)=\GL_2(\RR)$ and acts on $T_pX$ with two real eigenvectors, hence must be isomorphic to $(\ZZ/2)^k$, $k=0,1,2$. Now a simple exercise in group theory and Lemma \ref{lem: p-groups cubic surface} give the list of groups in the statement.
\end{proof}

We are ready to state the main result of this section. Note that in the following theorem we do not classify all possible automorphism groups of real cubic surfaces. It is more convenient for us to go through classification of possible $\Aut(X_\CC)$ instead. The latter can be found in  \cite[9.5]{cag} and \cite{di}; see also\footnote{Note that Segre's classification is known to be incorrect in some places. For example, the class VII is missing in his classification.} \cite{Segre} and \cite{hosoh}. For reader's convenience, we collect this description in the table below.

\begingroup
\renewcommand*{\arraystretch}{1.2}
\begin{longtable}{|c|c|c|c|}
	\caption{Automorphism groups of complex cubic surfaces}\label{table: dP3 aut} \\
	\hline
	Type \cite{di}  & $\Aut(X_\CC)$ & Equation \\ \hline
	
	
	I  & $(\ZZ/3)^3\rtimes\Sym_4$ &  $x_0^3+x_1^3+x_2^3+x_3^3$ \\ \hline
	II  & $\Sym_5$ & $x_0^2x_1+x_1^2x_2+x_2^2x_3+x_3^2x_0$ \\ \hline
	III  & $\Heis_3(3)\rtimes\ZZ/4$ &   $x_0^3+x_1^3+x_2^3+x_3^3+6ax_1x_2x_3$ \\ \hline
	IV  & $\Heis_3(3)\rtimes\ZZ/2$ &   $x_0^3+x_1^3+x_2^3+x_3^3+6ax_1x_2x_3$ \\ \hline
	V  & $\Sym_4$ & $x_0^3+x_0(x_1^2+x_2^2+x_3^2)+ax_1x_2x_3$ \\ \hline
	VI  & $\Sym_3\times\ZZ/2$ & $x_2^3+x_3^3+ax_2x_3(x_0+x_1)+x_0^3+x_1^3$ \\ \hline
	VII  & $\ZZ/8$ &  $x_3^2x_2+x_2^2x_1+x_0^3+x_0x_1^2$ \\ \hline
	VIII  & $\Sym_3$ & $x_2^3+x_3^3+x_2x_3(ax_0+bx_1)+x_0^3+x_1^3$ \\ \hline
	IX  & $\ZZ/4$ &  $x_3^2x_2+x_2^2x_1+x_0^3+x_0x_1^2+ax_1^3$ \\ \hline
	X  & $(\ZZ/2)^2$ & $x_0^2(x_1+x_2+ax_3)+x_1^3+x_2^3+x_3^3+6bx_1x_2x_3$ \\ \hline
	XI & $\ZZ/2$ & $x_1^3+x_2^3+x_3^3+6ax_1x_2x_3+x_0^2(x_1+bx_2+cx_3)$ \\ \hline
	
\end{longtable}
\endgroup

\begin{thm}\label{thm: cubic main}
	Let $X$ be a smooth real $\RR$-rational cubic surface, and $G\subset\Aut(X)$ be a group acting minimally on $X$. Then, according to the type of $X_\CC$, one of the following cases holds:
	\begin{description}
		\item[Type I]  $X$ is a real form of the Fermat cubic surface
		\begin{equation}\label{eq: Fermat}
		x_0^3+x_1^3+x_2^3+x_3^3=0.
		\end{equation}
		\textcolor{black}{There are 3 real forms of the Fermat cubic, denoted $F_{\id}$, $F_{(12)}$ and $F_{(12)(34)}$ (see Section \ref{subsec: Fermat}). In the first two cases, $\sigma^*$ is of type $A_1^3$, and for $F_{(12)(34)}$ it is of type $A_1$. For the real form $F_{\id}$, the group $G$ is $\Sym_3$, $\Alt_4$ or $\Sym_4$ (acting by permutation of coordinates in (\ref{eq: Fermat})); the groups $\Sym_4$ and $\Sym_3$ do occur as minimal\footnote{Meaning that we do not know about $\Alt_4$.}. For the real form $F_{(12)}$ the group $G$ is $\Sym_3$, $\ZZ/6$ or $\Dih_6$, while for the real form $F_{(12)(34)}$ it embeds into $(\Sym_3\times\Sym_3)\rtimes\ZZ/2$ (in the last two cases we do not claim that all such subgroups are minimal).}
		\item[Type II] $X$ is the real Clebsch cubic surface
		\[
		x_0+x_1+x_2+x_3+x_4=x_0^3+x_1^3+x_2^3+x_3^3+x_4^3=0.
		\]
		Moreover, $\sigma^*=\id$, $\Aut(X)\cong\Sym_5$ and $G$ is either $\Sym_4$, or $\Sym_5$ (both cases occur).
		\item[Types III and IV] Then $X$ is a real form of the cyclic cubic surface
		\[
		x_0^3+x_1^3+x_2^3+x_3^3+ax_1x_2x_3=0,
		\]
		and $G=\Aut(X)\cong\Sym_3$ acts minimally by permuting the coordinates $x_1$, $x_2$ and $x_3$. The real structure $\sigma^*$ is of type $A_1^3$.
		\item[Type V] $X$ is the real cubic surface
		\[
		\alpha x_0^3+x_1^3+x_2^3+x_3^3+x_4^3=0,\ \ \ x_0+x_1+x_2+x_3+x_4=0.
		\]
		and depending on the parameter $\alpha$, we have one of the following cases (all groups, except possibly $\Alt_4$, are indeed minimal):
		\begingroup
		\renewcommand*{\arraystretch}{1.4}
		\begin{longtable}{llccc}
			$\id$ & ($\alpha>1/4$) & $\Sym_4$, $\Sym_3$ \\ 	
			$A_1^3$ & ($\alpha<1/16$) &\ \ \ \ \ \ \ \ \ \ \ \ \ \ $\Sym_3,\ \ \ \Alt_4,\ \ \ \Sym_4.$ 
		\end{longtable}
		\endgroup
		\item[Type VI] Then $\Aut(X)\cong\Sym_3\times\ZZ/2$, and $X$ is a real form of the surface given by
		\[
		x_0^3+x_1^3+x_2^3+x_3^3+ax_2x_3(x_0+x_1)=0
		\]
		The group $G$ is one of the following:
		\[
		\Sym_3,\ \ \ZZ/6\simeq\ZZ/3\times\ZZ/2,\ \ \Sym_3\times\ZZ/2
		\]
		(all groups act minimally). Possible types of $\sigma^*$ are: $\id$, $A_1$, and $A_1^3$ (more information is given in the proof).
		\item[Type VIII] Then $\Aut(X)\cong\Sym_3$, $G=\Aut(X)$, and $X$ is a real form of the surface given by
		\[
		x_0^3+x_1^3+x_2^3+x_3^3+x_2x_3(ax_0+bx_1)=0.
		\]
		($\Aut(X)$ acts minimally). Possible types of $\sigma^*$ are $\id$, $A_1$ and $A_1^3$ (more information is given in the proof).
 	\end{description}
\end{thm}

\begin{proof}
	Here we give an overview of the proof, referring the reader to subsequent paragraphs for details. First we notice that Lemma \ref{lem: p-groups cubic surface} implies that Types VII, IX, X and XI are not relevant for us, as $G$ would be a $p$-group. 
	
	Next we look at surfaces with comparatively ``large'' automorphism groups. Cubics of Types II and V are studied in paragraphs \ref{sec: Clebsch} and \ref{subsec: dP3 and Sym4} respectively. Type I is discussed in paragraph \ref{subsec: Fermat}. Types III and IV are discussed in paragraph \ref{subsec: non-equianharmonic}.
	
	Now let us consider the case when $X$ is a surface of type VIII, i.e. $\Aut(X_\CC)\cong\Sym_3$. Then we must have $G=\Aut(X)\cong\Sym_3$ and $G$ is minimal by \cite[Theorem 6.14 (7)]{di}. Let us find possible real structures. Note that $X$ has 3 real Eckardt points $p_1$, $p_2$ and $p_3$ (recall that there is a bijective correspondence between the set of Eckardt points on a smooth cubic surface, and the set of its involutions whose fixed loci consist of a hyperplane section and an isolated point \cite[Proposition 9.1.23]{cag}). Note that $p_i$ are collinear and do not lie on a line contained in $X$ \cite[Proposition 9.1.26]{cag}. Let us say that a real Eckardt point is of the {\it first type} if all three lines passing through this point are real, and of the {\it second type} if there are one real $R_i$ and two complex conjugate. We may assume that one of the following cases hold:
	\begin{enumerate}
		\item $p_1$, $p_2$ are of 1st type and $p_3$ is of 2nd type. Then clearly $G$ preserves $R_3$, hence is not minimal.
		\item $p_1$ and $p_2$ are of 2nd type and $p_3$ is of 1st type. We may assume that $G$ permutes $R_1$ and $R_2$. If $R_1\cap R_2=\varnothing$, then $G$ is not minimal. Otherwise the plane $\langle R_1,R_2\rangle$ intersects $X$ at some real line which is $G$-invariant.
		\item All points are of 2nd type. Then $G$ acts on $R_i$. Note that these lines cannot intersect in one point, as it would be another Eckardt point. Therefore, $R_i$ form a triangle. Now a surface containing a point of 2nd type can be of Segre types $F_3$, $F_4$ or $F_5$ \cite[p. 153]{Segre}. In our situation, $\RR$-rationality assumption and Lemma \ref{lem: dP3 7 lines} imply that $\sigma$ is of type $A_1^3$.
		\item Finally, if all points are of 1st type, then we have at least 9 real lines on $X$ and hence $\sigma^*$ is of type $\id$ or $A_1$.  
	\end{enumerate}
	Finally, assume that $X$ is a real cubic of type VI, i.e. $\Aut(X_\CC)\cong\Sym_3\times\ZZ/2$. Then $X_\CC$ has 4 Eckardt points: 3 collinear points $p_1$, $p_2$, $p_3$, and the fourth point $q$. The lines $\overline{qp_i}$, $i=1,2,3$, lie on $X_\CC$ (otherwise, $X_\CC\cap\overline{qp_i}=\{q,p_i,r_i\}$ with $r_i$ being an Eckardt point by \cite[Proposition 9.1.26]{cag}). Since both $\Gamma$ and $G$ preserve collinearity, $q$ is real and $G$-fixed. Assume that not all $p_i$'s are real, and let $p_1$ be the only real point among them. Then the real line $\overline{qp_1}$ is $G$-invariant for any $G\subset\Aut(X)$. Thus we may assume that all Eckardt points on $X$ are real. In particular, $\Aut(X)\cong\Sym_3\times\ZZ/2$ (otherwise we do not have enough real involutions), and then \cite[Theorem 6.14 (6)]{di} shows that minimal subgroups are $\Sym_3$, $\ZZ/6$ and the whole $\Aut(X)$. Same considerations as in the previous case show that $p_i$ have the same type. One can show that $\sigma$ can be of types $\id$, $A_1$ and $A_1^3$, see \cite[\S 106]{Segre}.
\end{proof}

\begin{rem}
The classification given in Theorem \ref{thm: cubic main} can be also formulated in terms of elements of the Weyl group $\Weyl(\E_6)$. Let $X$ be a real $\RR$-rational cubic surface, and $G\subset\Aut(X)$ be a group acting minimally on $X$. Recall that $|\Weyl(\E_6)|=2^73^45$. By Lemma \ref{lem: p-groups cubic surface} we may assume that $G$ contains an element of order 3 or 5. We have the following cases:
\begin{itemize}
\item {\it $G$ has an element of order 5}. Then $X$ is isomorphic to the Clebsch diagonal cubic over $\RR$, see Proposition \ref{prop: Clebsch forms}.
\item {\it $G$ has an element of order 3}. Let $g\in G$ be an element of order 3. Then $g^*$ is of type\footnote{3C, 3D and 3A respectively in ATLAS notation.} $A_2$, $A_2^2$ or $A_2^3$ in $\Weyl(\E_6)$. On the other hand, if $g^*$ is of type $A_2^3$, then $\tr g^*=-3$, hence $\Eu(X(\CC)^g)=0$. This is possible if and only if $X(\CC)^g$ consists of an elliptic curve, a section by a fixed hyperplane of $g$ in $\PP_\CC^3$. But $g\in\PGL_4(\RR)$, so it cannot have such a hyperplane (as it would correspond to an eigenvalue of multiplicity 3). So, we may assume that $g^*$ is either of type $A_2$, or of type $A_2^2$ in $\Weyl(\E_6)$. 
	\begin{itemize}
			\item {\it $G$ has an element of type $A_2$}. As was shown in \cite[9.5.1]{cag}, in this case $X_\CC$ is isomorphic to the Fermat cubic surface
			\[
			x_0^3+x_1^3+x_2^3+x_3^3=0,
			\]
			whose real forms are studied in paragraph \ref{subsec: Fermat}.
			\item {\it $G$ has an element of type $A_2^2$}. Then $X_\CC$ is isomorphic to the surface
			\begin{equation}\label{eq: 3D cubic}
			x_0^3+x_1^3+x_2^3+x_3^3+x_0x_1(ax_2+bx_3)=0.
			\end{equation}
			whose complex automorphism groups is $\Sym_3$ for general values of parameters $a$ and $b$. For special values we get more automorphisms, which can be illustrated as follows (the arrows denote specialization, and the numbers denote the type of surface according to \cite{di}):
			\[\xymatrix@C+1pc{
				& \boxed{$V$}\ \Sym_4\ar[r]\ar[dr] & \boxed{$II$}\ \Sym_5\\
				\boxed{$VIII$}\ \Sym_3\ar[ur]\ar[r]\ar[dr] & \boxed{$VI$}\ \Sym_3\times\ZZ/2\ar[r]\ar[ur] &\boxed{$I$}\ (\ZZ/3)^3\rtimes\Sym_4\\
				&\boxed{$IV$}\ \Heis_3(3)\rtimes\ZZ/2\ar[r]\ar[ur] &\boxed{$III$}\ \Heis_3(3)\rtimes\ZZ/4
			}\]
			Note that the types III, IV, and I correspond to the situation when the surface (\ref{eq: 3D cubic}) specializes to a cyclic cubic surface (defined later). Such surfaces are logically divided in three distinct types: harmonic, equianharmonic and the rest, see below. The equianharmonic case, namely the Fermat cubic (I), is discussed in paragraph \ref{subsec: Fermat}. The types III and IV correspond to harmonic and neither harmonic nor equianharmonic cubics respectively and are discussed in paragraph \ref{subsec: non-equianharmonic}. 
	\end{itemize}	
\end{itemize}	
\end{rem}
	
In the next few paragraphs we discuss cubic surfaces of types I-V. For this we first need to recall Sylvester's classical approach to cubic forms.
	
\section*{Sylvester non-degenerate cubic surfaces}
	
Recall that by the classical result of J. Sylvester (see \cite[Corollary 9.4.2]{cag} for modern exposition) a general homogeneous cubic form $F(x_0,x_1,x_2,x_3)$ can be written as a sum
\begin{equation}\label{eq: Sylvester form}
F=L_0^3+L_1^3+L_2^3+L_3^3+L_4^3,
\end{equation}
over $\CC$, where $L_i(x_0,x_1,x_2,x_3)$ are linear forms, no four are linearly dependent. Moreover, these forms are defined uniquely, up to scaling by a cubic root of unity. The corresponding planes $L_i=0$ cut out so-called {\it Sylvester pentahedron}. Let $\alpha_0L_0+\ldots+\alpha_4L_4=0$ be a unique, up to proportionality, linear relation. Consider the embedding
\[
\PP^3\hookrightarrow\PP^4,\ \ \ \iota:\ [x_0:x_1:x_2:x_3]\mapsto[y_0:y_1:y_2:y_3:y_4]\overset{\text{def}}{=}[L_0:L_1:L_2:L_3:L_4].
\]
If $S=\{F=0\}\subset\PP^3$ is the cubic surface given by $F$, one has
\[
\iota(S)=\Big\{\sum_{i=0}^{4}y_i^3=\sum_{i=0}^{4}\alpha_i y_i=0\Big\}.
\]
Let us further make the change of coordinates $z_i=\alpha_i y_i$ and assume that our surface is given by
\begin{equation}\label{eq: Sylvester form 2}
\sum_{i=0}^{4}\lambda_i z_i^3=\sum_{i=0}^{4}z_i=0,
\end{equation}
where $\lambda_i=1/\alpha_i^3$. These parameters are uniquely determined up to permutation and common scaling by the isomorphism class of the surface. 

Representation (\ref{eq: Sylvester form 2}) is called the {\it Sylvester form} of a cubic surface. So, a {\it general} cubic surface admits a unique\footnote{In the sense mentioned above.} Sylvester form. We call such surfaces {\it Sylvester nondegenerate} (and {\it Sylvester degenerate} otherwise).

One can show that the automorphism group of any surface given by (\ref{eq: Sylvester form 2}) is a subgroup of the group $\Sym_5$ which acts by permuting coordinates (or, equivalently, the sides of the Sylvester pentahedron) \cite[Theorem 6.1]{DolDun}. Moreover, the equation
\[
\sum_{i=0}^{4}\lambda_i z_i^3=0
\]
must be transformed into itself under any such permutation $\tau$, i.e. the constant $\zeta\in\CC$ by which this equation is multiplied equals to 1. Indeed, otherwise it is easy to see that $\zeta$ must be a 5th primitive root of unity, and $\tau$ is a cycle of length 5. The equation then necessarily reduces to
\[
z_0^3+\zeta z_1^3+\zeta^2 z_2^3+\zeta^3 z_3^3+\zeta^4 z_4^3=0,
\]
which defines a Sylvester degenerate cubic surface. So, in order to have some nontrivial permutation among the $z_i$'s transforming (\ref{eq: Sylvester form 2}) into itself, the parameters $\lambda_i$'s must be not all distinct. As was noticed already in \cite{Segre}, the corresponding automorphism groups are generated by permutations of $z_i's$ with the same values of $\lambda_i$'s (e.g. if $\lambda_0=\lambda_1=\lambda_2$, then $\Aut(X_\CC)\cong\Sym_3$ is the group of permutations of $z_0$, $z_1$ and $z_2$):

\begingroup
\renewcommand*{\arraystretch}{1.2}
\begin{longtable}{|c|c|c|c|}
	\caption{Automorphism groups of Sylvester non-degenerate cubic surfaces}
	\label{table: dP3 general} \\
	\hline
	Name in \cite{di} & $\Aut(X_\CC)$ & Relations between $\lambda_i$'s \\ \hline
	
	II & $\Sym_5$ & $\lambda_0=\lambda_1=\lambda_2=\lambda_3=\lambda_4$ \\ \hline
	
	V & $\Sym_4$ & $\lambda_1=\lambda_2=\lambda_3=\lambda_4$ \\ \hline
	
	VI & $\Sym_3\times\ZZ/2$ & $\lambda_0=\lambda_1=\lambda_2,\ \lambda_3=\lambda_4$ \\ \hline
	
	VIII & $\Sym_3$ & $\lambda_0=\lambda_1=\lambda_2$ \\ \hline
	
	X & $(\ZZ/2)^2$ & $\lambda_1=\lambda_2,\ \lambda_3=\lambda_4$ \\ \hline
	
	XI & $\ZZ/2$ & $\lambda_3=\lambda_4$ \\ \hline
\end{longtable}
\endgroup

\subsection{Clebsch diagonal cubic}\label{sec: Clebsch}

(see also Segre's account \cite[\S 102]{Segre})

\begin{prop}\label{prop: Clebsch forms}
	Let $X$ be a real $\RR$-rational cubic surface with $\Aut(X_\CC)\cong\Sym_5$, and $G\subset\Aut(X)$ be a group acting minimally on $X$. Then $X$ is ismomorphic to the Clebsch diagonal cubic
	\[
	x_1+x_2+x_3+x_4+x_5=0,\ \ x_1^3+x_2^3+x_3^3+x_4^3+x_5^3=0,
	\]
	in $\PP^4_\RR$, $\Aut(X)\cong\Sym_5$ and $G$ is either $\Sym_5$ or $\Sym_4$ (both groups occur).
\end{prop}
\begin{proof}
	It is well known that $X_\CC$ is $\CC$-isomorphic to the Clebsch cubic surface \cite[Theorem 9.5.8]{cag}. Note that $\Sym_5$ acts on $X_\CC$ acts by permuting coordinates $x_1,\ldots,x_5$, and $\Gamma=\Gal(\CC/\RR)$ acts on $\Sym_5$ trivially. Thus $H^1(\Gamma,\Sym_5)=\Hom(\Gamma,\Sym_5)/\sim$, where $\sim$ denotes conjugation by elements of $\Sym_5$. Since $\Sym_5$ has exactly 2 conjugacy classes of involutions with representatives (12) and (12)(34), we see that the Clebsch cubic has 3 real forms, which we denote by $X_{\id}$, $X_{(12)}$ and $X_{(12)(34)}$. We are now going to calculate the number of real lines on each nontrivial real form of the Clebsch cubic. First start with their description on $X_\CC$. Put
	\[
		L_{ijk}=
		\{
		x_i=
		x_j+x_k=0
		\},
		\]
		where either $i=1$ and $(jk)\in\{(23),(24),(25)\}$, or $i\in\{2,3,4,5\}$, $j=1$ and $k\ne i$, $k\ne 1$ (clearly, some permutation of indexes give same lines). Further, define
		\[
		L_{ijkl}=\{
		x_i+\zeta x_j+x_k=
		x_j+\zeta x_i+x_l=
		\zeta x_i+\zeta x_j-x_5=0
		\}
		\]
		where $i,j,k,l\in\{1,2,3,4\}$, $i<j$ and $\zeta=(1+\sqrt{5})/2$.
	An easy calculation shows that we have 15 different lines $L_{ijk}$ and 12 different lines $L_{ijkl}$.
	
	Now we are interested in those lines $L\subset X_\CC$ for which $\sigma(L)=L$. It means that in the case of the real form corresponding to the cycle (12) we are looking for those lines among $L_{ijk}$ and $L_{ijkl}$  which are $(12)$-invariant. So, we get just $L_{312}$, $L_{412}$ and $L_{512}$. An easy calculation shows that $X_{(12)}$ has 13 real tritangent planes (these correspond to $(12)$-invariant pairs of lines). In particular, $X_{(12)}$ is not rational over $\RR$ (see Table \ref{table: dP3 real lines}), so $X$ is not isomorphic to $X_{(12)}$. Similarly, for the cycle (12)(34) the real lines on the corresponding real form are $L_{512}$, $L_{513}$ $L_{514}$, $L_{1234}$, $L_{1243}$, $L_{3412}$ and $L_{3421}$. Our calculations are summed up in Table \ref{table: Clebsch real lines}.
	\begingroup
	\renewcommand*{\arraystretch}{1.2}
	\begin{longtable}{|c|c|c|c|}
		\caption{Real lines on Clebsch cubic surface}
		\label{table: Clebsch real lines} \\
		\hline
		Form & $X_{\id}$ & $X_{(12)}$ & $X_{(12)(34)}$ \\ \hline
		Number of real lines & 27 & 3 (not $\RR$-rational) & 7\\
		\hline		
	\end{longtable}
	\endgroup
	Finally, if $X$ has exactly 7 real lines, then it cannot be $G$-minimal for any group $G$ by Lemma \ref{lem: dP3 7 lines}. If all 27 lines are real, then $G$ is either $\Sym_4$ or $\Sym_5$ by \cite[Theorem 6.14]{di}. 
\end{proof}

\subsection{Cubic surfaces with automorphism group $\Sym_4$}\label{subsec: dP3 and Sym4} (see also \cite[\S 107]{Segre}) Let $X$ be a real $\RR$-rational cubic surface with $\Aut(X_\CC)\cong\Sym_4$. The Sylvester form of $X_\CC$ is
\begin{equation}\label{eq: dP3 G2}
\alpha z_0^3+z_1^3+z_2^3+z_3^3+z_4^3=0,\ \ \ z_0+z_1+z_2+z_3+z_4=0,
\end{equation}
where we put $\alpha=\lambda_0/\lambda_1$. Both $\Aut(X)$ and $\Gamma$ act on the faces of the Sylvester pentahedron, which we denote by $\pi_0,\ldots,\pi_4$. In particular, we have three distinct cases: (1) all $\pi_i$'s are real; (2) $\pi_0,\ \pi_1,\ \pi_2$ are real, and $\pi_3=\sigma(\pi_4)$; (3) $\pi_0$ is real, and $\pi_1=\sigma(\pi_2)$, $\pi_3=\sigma(\pi_4)$. Note that in the last two cases $\Aut(X)$ faithfully acts on the set of pairs $\{ \pi_1,\pi_2\}$ and $\{\pi_3,\pi_4\}$, so $\Aut(X)$ will be a 2-group, and $X$ is then never $G$-minimal, see \cite[Table on p. 161]{Segre}. 

Thus we may assume that the Sylvester presentation is real, i.e. all $\pi_i$ are defined over $\RR$ and $X$ varies in the real pencil of cubic surfaces $X_\alpha$, $\alpha\in\RP^1$. A simple calculation shows that $X_\alpha$ is a smooth cubic surface for all $\alpha\in\CC$, except $\alpha=1/4$ and $\alpha=1/16$. The surface $X_{1/16}$ has the unique singular point $[-4:1:1:1:1]$, and the surface $X_{1/4}$ has exactly four singular points $[2:1:-1:-1:-1]$, $[2:-1:-1:1:-1]$, $[2:-1:1:-1:-1]$, and $[2:-1:-1:-1:1]$. By Ehresmann's fibration theorem, the surfaces $X_\alpha$ arising from $\alpha$'s lying between a consecutive pair of the values $-\infty$, $1/16$, $1/4$, $+\infty$ are homeomorphic to each other. The special cases $\alpha=1$ and $\alpha=0$ yield Clebsch and Fermat cubic surfaces respectively. Their real forms are studied in paragraphs \ref{sec: Clebsch} and \ref{subsec: Fermat}. Finally, it can be shown that for $1/16<\alpha<1/4$ the real loci $X_\alpha(\RR)$ are disconnected, so $\sigma^*$ is of type $A_1^4$ for such surfaces. We can illustrate the situation as follows:
\[
\alpha:\ \ \ \-\infty\ \overset{A_1^3}{\rule{2cm}{0.8pt}}1/16\overset{A_1^4}{\rule{2cm}{0.8pt}}1/4\overset{\id}{\rule{2cm}{0.8pt}}\ +\infty
\]
By \cite[Theorem 6.14 (5)]{di} the groups $\Aut(X)\simeq\Sym_4$ and $\Sym_3$ act minimally on (\ref{eq: dP3 G2}) by permuting coordinates. For the real structure of type $A_1^3$ we might have a new group $\Alt_4$ (however, we do not address this question here).

\section*{Sylvester degenerate cubic surfaces}

We are now going to study those real cubic surfaces which either do not admit the Sylvester form at all, or this form is not unique. The latter ones are called {\it cyclic} surfaces. These are the surfaces for which four of the five $L_i$'s are linearly dependent, and after a suitable change of variables the equation becomes
\begin{equation}\label{eq: cyclic cubic}
F=x_0^3+G_3(x_1,x_2,x_3),
\end{equation}
where $G_3$ is a ternary cubic form (so, our surface is  a Galois triple cover of $\PP^2$). Consider the cubic curve $E$ given by $G_3=0$. Following \cite[Definition 3.1.2, Theorem 3.1.3]{cag}, we call $E$ and the corresponding cyclic surface {\it harmonic} if the absolute invariant $j(E)=1728$, and {\it equianharmonic} if $j(E)=0$. 

\subsection{Equianharmonic case: Fermat cubic}\label{subsec: Fermat}

(compare \cite[\S 103]{Segre}) Any equianharmonic cubic is projectively isomorphic to the Fermat cubic over $\CC$. In this subsection $X$ denotes the Fermat cubic surface
\[
x_1^3+x_2^3+x_3^3+x_4^3=0.
\]
Recall that $\Aut(X_\CC)\cong(\ZZ/3)^3\rtimes\Sym_4$, where one can view $(\ZZ/3)^3$ as the group
\[
\big \{\omega=(\omega_1,\omega_2,\omega_3,\omega_4)\in\CC^4:\ \omega_i^3=1\ \text{for all}\ i,\ \text{and}\ \omega_1\omega_2\omega_3\omega_4=1 \big \}
\]
with an obvious action $\psi$ of $\Sym_4$ on $(\ZZ/3)^3$. The group $\Gamma$ acts on $\Aut(X_\CC)$ as
\[
\sigma\cdot \big((\omega_1,\omega_2,\omega_3,\omega_4),\tau\big)=\big((\overline{\omega}_1,\overline{\omega}_2,\overline{\omega}_3,\overline{\omega}_4),\tau\big ).
\]
Any 1-cocycle $c: \Gamma\to\Aut(X_\CC)$ is given by $c(\sigma)=(\omega,\tau)$ such that $c(\sigma)\cdot\sigma(c(\sigma))=1$, i.e.
\begin{equation}\label{eq: cocycle}
(\omega,\tau)\cdot (\overline{\omega},\tau)=\big ( (\omega_1\overline{\omega}_{\tau^{-1}(1)},\omega_2\overline{\omega}_{\tau^{-1}(2)},\omega_3\overline{\omega}_{\tau^{-1}(3)},\omega_4\overline{\omega}_{\tau^{-1}(4)}),\tau^2 \big)=1.
\end{equation}
In particular, $\tau$ is either trivial, or of order 2. If $c\sim c'$, then $\tau$ and $\tau'$ (corresponding to $c(\sigma)$ and $c'(\sigma)$) are conjugate in $\Sym_4$, thus we may assume that $\tau$ is one of the following: $\id$, $(12)$ or $(12)(34)$. A slightly tedious computation shows that this indeed corresponds to partition of the set of 1-cocylces into 3 conjugacy classes with representatives $(1,\id)$, $(1,(12))$ and $(1,(12)(34))$, so the Fermat cubic surface has 3 real forms. We refer to these cases as $F_{\id}$, $F_{(12)}$ and $F_{(12)(34)}$ respectively. The 27 lines on $X_\CC$ are given by
\begin{gather*}
\alpha_{kj}:\ \ x_1+\omega^kx_4=x_2+\omega^j x_3=0,\\
\beta_{kj}:\ \ x_1+\omega^kx_3=x_4+\omega^j x_2=0,\\
\gamma_{kj}:\ \ x_1+\omega^kx_2=x_4+\omega^j x_3=0,
\end{gather*}
where $0\leqslant j,k\leqslant 2$, and $\omega$ is a primitive 3rd root of unity. One can easily check that $(\sigma,g)$-invariant lines (i.e. real ones) are
\begin{gather*}
\alpha_{00},\ \ \beta_{00},\ \ \gamma_{00}\ \ \ \text{for}\ g=(1,\id),\\
\gamma_{00},\ \ \gamma_{10},\ \ \gamma_{20}\ \ \ \text{for}\ g=(1,(12)),\\
\text{all}\ \gamma_{kj},\ \alpha_{00}\ \ \beta_{00},\ \ \alpha_{12},\ \ \beta_{11},\ \ \alpha_{21},\ \ \beta_{22}\ \ \ \text{for}\ \ g=(1,(12)(34)),
\end{gather*}
We see that there are 3 real lines on $F_{\id}$ and $F_{(12)}$, and 15 real lines on $F_{(12)(34)}$. Note that three real lines on $F_{\id}$ form a triangle, while on $F_{(12)}$ they intersect at an Eckardt point. A real cubic surface with 15 real lines is always rational over $\RR$. To determine which of the forms $F_{\id}$ and $F_{(12)}$ are $\RR$-rational, one can compute the number of real tritangent planes. These are given by
\begin{gather*}
x_1+\omega^i x_2+\omega^j x_3+\omega^k x_4=0,\\
x_s+\omega^l x_p=0,
\end{gather*}
where $s<p$, and $i,j,k,l\in\ZZ/3$. So, in each of two cases the number of real planes is 7, which means that all real forms of the Fermat cubic are rational over $\RR$ (see Table \ref{table: dP3 real lines}). 

Finally, let us determine which groups can act minimally on a real Fermat cubic. The surface (\ref{eq: Fermat}) corresponds to the real form $F_{\id}$. Thus, a minimal group $G$ embeds into $\Sym_4$ (acting by permutation of coordinates). The groups $\Sym_4$ and $\Sym_3$ do act minimally by \cite[Theorem 6.14 (1)]{di}. Since we have a non-trivial real structure of type $A_1^3$, it could be possible that $\Alt_4$ (i.e. the only remaining non-$p$-group) acts minimally on such $X$; we do not investigate this question here. For the real form $F_{(12)}$, our group $G$ must embed into the tangent space of a real Eckardt point; it is easy to see, using Lemma \ref{lem: dP3 3 lines lemma} or directly \cite[\S 103, case II]{Segre}, that $G$ is $\Sym_3$, $\ZZ/6$ or $\Dih_6$. We leave it to the interested reader to write down an explicit equation of $F_{(12)}$ and to find which of these groups are actually minimal. Consider now the form $F_{(12)(34)}$. By \cite[\S 43, \S 103]{Segre} the group $\Aut(X)$ is a group of order 72, having $\Sym_3\times\Sym_3$ as an index 2 subgroup. In fact, it is straightforward to give a real cubic surfaces acted by $(\Sym_3\times\Sym_3)\rtimes\ZZ/2$, just by considering the surface
\[
S:\ g_3(x,y)+g_3(z,t)=0,
\]
where $g_3(x,y)=x^3-3xy^2$ is the absolute invariant of $\Dih_3\cong\Sym_3$. Note that $S$ is automatically the real form of the Fermat cubic, since only the automorphism group of the latter one can contain a copy of $\Sym_3\times\Sym_3$ (see Table \ref{table: dP3 aut}; the case of $\Heis_3(3)\rtimes\ZZ/4$ is easily excluded).

\subsection{Non-equianharmonic case}\label{subsec: non-equianharmonic} A cyclic non-singular and non-equianharmonic cubic surface has the canonical equation \cite[\S 88]{Segre} over $\CC$
\begin{equation}\label{eq: non-equianh}
x_0^3+(x_1^3+x_2^3+x_3^3-3\lambda x_1x_2x_3)=0
\end{equation}
with $\lambda(\lambda^3+8)(\lambda^3-1)\ne 0$. It corresponds to Segre types (viii) and (ix) \cite[\S 100]{Segre} and Types III-IV of \cite{di}. So, the equation (\ref{eq: non-equianh}) describes a cyclic cubic surface varying in a pencil whose real members correspond to $\lambda\in\RR$. There are only two real singular surfaces in this pencil, arising from $\lambda=\infty$ and $\lambda=1$. It can be checked\footnote{One can pick a specific value of $\lambda>1$ and $\lambda<1$ and calculate the number of real lines and tritangent planes, and then use Table \ref{table: dP3 real lines}.} that $\sigma^*$ is always of type $A_1^3$ (see also \cite[\S 104]{Segre}) 

Let $f$ be a homogeneous polynomial defining a hypersurface $Z$ in $\PP^n$. Recall that the hypersurface $\He(Z)=\big \{\det\He(f)=0\big \}$ is called the Hessian hypersurface of $Z$. The Hessian of a cyclic cubic surface is the union of a {\it fundamental plane} $\Pi=\{x_0=0\}$ and the cone over a cubic curve. Thus each automorphism of $X$ is a linear map operating separately on $x_0$ and $x_1,\ x_2,\ x_3$. One can show that $\Aut(X)$ is isomorphic to a subgroup of $\Sym_3$ \cite[\S 104]{Segre}. So, a minimal group $G$ must be isomorphic to $\Sym_3$; note that such a group indeed acts minimally on $X$ (since it is already minimal over $\CC$, \cite[Theorem 6.14]{di}). 

\begin{rem}
	Recall that the intersection $C=\Pi\cap X_\CC$ is a cubic curve, whose 9 inflection points correspond to 9 Eckardt points of $X_\CC$. Obviously, in our case $C$ is defined over $\RR$ (as $\Pi$ is $\Gamma$-invariant, being the only plane component of the Hessian). It is well known that a real cubic curve has exactly 3 real inflection points, and these points are collinear. In terminology of the proof of Theorem \ref{thm: cubic main}, the corresponding Eckardt points on $X$ are of type 2 (these automatically follows from the type of $\sigma^*$, or can be easily seen from the explicit description of lines on $X$, see \cite[Example 9.1.24]{cag}). 
\end{rem}

\subsection{Non-cyclic Sylvester degenerate surfaces}

A detailed description of the automorphism groups of such surfaces can be found in \cite[\S 100]{Segre} (cases x--xvii). After excluding $2$-groups, we are left just with two types (xi) and (xiv), having (complex) automorphism groups $\Sym_3$ and $\Sym_3\times\ZZ/2$ respectively. Such groups were already discussed in the proof of Theorem \ref{thm: cubic main}.

\addtocontents{toc}{\protect\setcounter{tocdepth}{1}}
\subsection{Conjugacy classes}

Classification of links in \cite{isk-1} shows that del Pezzo cubic surfaces are rigid, and hence the conjugacy class of $G$ in $\Cr_2(\RR)$ is determined by the conjugacy class of $G$ in $\Aut(X)$.

\vspace{0.5cm}

\section{Del Pezzo surfaces of degree 2}\label{section: dP2}

Throughout this section $X$ (or $X_B^{\sgn}$, see below) denotes a real del Pezzo surface of degree 2. The anticanonical map 
$
\varphi_{|-K_X|}: X\rightarrow\PP_\RR^2
$
is a double cover branched over a smooth quartic $B\subset\PP_\RR^2$. The Galois involution $\gamma$ of the double cover is called the {\it Geiser involution}. Note that $B(\RR)$ divides $\RP^2$ into connected open sets and only one of these is non-orientable. Choose an equation $F(x,y,z)=0$ of $B$ such that $F$ is negative on that non-orientable set. One can associate two different degree 2 del Pezzo surfaces to $B$, namely
\begin{gather*}
X^{\sgn}_B=\big\{[x:y:z:w]\in\PP_\RR(1,1,1,2):\ w^2=\sgn\cdot F(x,y,z)\big\},\ \ \ \text{where}\ \sgn\in\{1,-1\}.
\end{gather*}
It is classically known that there are 6 topological types of degree 4 smooth real plane curves. Correspondingly there are 12 topological types of degree 2 real del Pezzo surfaces. The following table lists only those $X=X_B^{\sgn}$ which are rational over $\RR$ (see \cite{Wall} or \cite{kol} for details):

\begingroup
\renewcommand*{\arraystretch}{1.4}
\begin{longtable}{|c|c|c|c|c|c|c|c|}
	\caption{Involutions in $\Weyl(\E_7)$ and real forms of $\RR$-rational del Pezzo surfaces of degree 2}
	\label{table: dP2 real forms} \\
	\hline
	Conjugacy class of $\sigma^*\in\Weyl(\E_7)$ & Eigenvalues of $\sigma^*$ & $\tr\sigma^*$ & $\sgn$ & $X^{\sgn}_B(\RR)$ & $B(\RR)$ & $\#$ real lines \\ \hline
	$\id$ & $1^7$ & 7 & $-$ & $\#8\RP^2$ & $\bigcirc\bigcirc\bigcirc\bigcirc$ & 56  \\ \hline
	$A_1$ & $-1,1^6$ & 5 & $-$ & $\#6\RP^2$ & $\bigcirc\bigcirc\bigcirc$ & 32  \\ \hline
	$A_1^2$ & $-1^2,1^5$ & 3 & $-$ & $\#4\RP^2$ & $\bigcirc\bigcirc$ & 16  \\ \hline
	${A_1^3}''$ & $-1^3,1^4$ & 1 & $-$ & $\#2\RP^2$ & $\bigcirc$ & 8  \\ \hline
	${A_1^3}'$ & $-1^3,1^4$ & 1 & $+$ & $\Torus^2$ & $\circledcirc$ & 0  \\ \hline
	${A_1^4}'$ & $-1^4,1^3$ & $-1$ & $+$ & $\Sph^2$ & $\bigcirc$ & 0  \\ \hline
\end{longtable}
\endgroup

The Geiser involution is contained in the center of $\Aut(X)$ and fits into the short exact sequence
\[
1\longrightarrow\langle\gamma\rangle\longrightarrow\Aut(X)\longrightarrow\Aut(B)\longrightarrow 1,
\]
It is well known that this exact sequence splits, i.e. $\Aut(X)\cong\Aut(B)\times\langle\gamma\rangle$. In particular, we have the following possibilities for the group $G$:

\begin{itemize}
	\item $\gamma\notin G$. Then $G$ is isomorphic to a subgroup $G_B\subset\Aut(B)\subset\PGL_3(\RR)$. Possible automorphism groups of real algebraic curves of genus 3 (considered as Klein surfaces) were described\footnote{In fact, for each automorphism group the authors even provide some restrictions on the number of real ovals.} in \cite{real genus 3}. Excluding those which do not embed into $\PGL_3(\RR)$ we get the following list:
	\[
	\ZZ/2,\ \ \ZZ/2\times\ZZ/2,\ \ \Dih_3,\ \ \Dih_4,\ \ \Dih_6,\ \ \Sym_4.
	\]
	Since our quartic lies in $\PP_\RR^2$ it is not  difficult to obtain this classification using invariant theory, see Appendix \ref{appendix} for description of some invariants. This also shows that our curve cannot admit an automorphism of order 6: otherwise the equation of $B$ reduces to the form 
	\[
	z^4+Az^2(x^2+y^2)+B(x^2+y^2)^2=0,
	\]
	which is singular. Further, by \cite[Theorem 1.2]{Yas} there is no $H=\ZZ/3$-action on an {\it $\RR$-rational} del Pezzo surface of degree 2 with $\Pic(X)^H\simeq\ZZ$. Therefore, if $G$ does not contain $\gamma$ and acts \textcolor{black}{strongly minimally} on an $\RR$-rational del Pezzo surface of degree 2, then it is isomorphic to one of the following groups:
	\begin{equation}\label{eq: dP2 list}
		\ZZ/2,\ \ \ZZ/2\times\ZZ/2,\ \ \ZZ/4,\ \ \Dih_3,\ \ \Dih_4,\ \ \Alt_4,\ \ \Sym_4.
	\end{equation} 
	\item $\gamma\in G$. Then $G$ is of the form $\langle \gamma\rangle\times G_B$, where $G_B$ is one of those listed in (\ref{eq: dP2 list}) (if not trivial), and the group $\ZZ/6$ containing $\gamma$. Recall that for every real del Pezzo surface of degree 2 we have $\Pic(X)^\gamma\simeq\ZZ$. Therefore, any group $G\subset\Aut(X)$ containing $\gamma$ is {\it automatically} strongly minimal.
\end{itemize}

The main result of this section is the following.

\begin{prop}
	Let $(X,\sigma)$ be a real del Pezzo surface of degree 2 and $G\subset\Aut(X)$ be a group acting \textcolor{black}{strongly minimally} on $X$ and not containing the Geiser involution. Then one of the following possibilities holds.
	\begin{enumerate}
		\item $G$ is a cyclic group $\langle g\rangle_n$
		\item[]
		\begin{itemize}
			\item $n=2$:
				\begin{itemize} 
					\item[${\bf (2^+)}$] $g: [x:y:z:w]\mapsto [x:y:-z:w]$, $g^*$ has type ${A_1^4}'$, and $\sigma^*$ is of type ${A_1^4}'$, ${A_1^3}'$ or ${A_1^3}''$. The equation of $X$ has the form
					\[
					\pm w^2=z^4+f_2(x,y)z^2+f_4(x,y),
					\]
					where $f_2$ and $f_4$ are some binary forms of degrees 2 and 4 respectively which are chosen\footnote{Unfortunately, we do not know if there is any characterization of $B(\RR)$ in terms of the coefficients of $F$. However, it should not be difficult to do determine the topology of $B(\RR)$ for a {\it given} equation.} in accordance with Table \ref{table: dP2 real forms} (as well as the sign of $w$).
				\end{itemize}
			\item $n=4$:
				\begin{itemize} 
					\item[${\bf (4^+)}$] $g: [x:y:z:w]\mapsto [-y:x:z:w]$, $g^*$ has type $A_3^2$, and $X$ is of the form
					\begin{equation}
					\pm w^2=z^4+Az^2(x^2+y^2)+B(x^4+y^4)+Cx^2y^2+D(x^3y-xy^3)
					\end{equation}
					for some $A,B,C,D\in\RR$.
					\item[${\bf (4^-)}$] $g: [x:y:z:w]\mapsto [-y:x:z:-w]$, $g^*$ has type $\Dih_4(a_1)\times A_1$, and $X$ is of the form
					\begin{equation}\label{eq: 4-}
					\pm w^2=z^4+Az^2(x^2+y^2)+B(x^4+y^4)+Cx^2y^2
					\end{equation}
					for some $A,B,C\in\RR$.
				\end{itemize}			
		\end{itemize}
		\textcolor{black}{In each case, except possibly ${\bf 2}^{+}$ and $\sigma$ of type ${A_1^3}''$, it is indeed possible to choose the coefficients in the equation of $X$, such that $G$ is strongly minimal.}
		\item $G$ is isomorphic to one of the groups $(\ZZ/2)^2$, $\Sym_3$, $\Dih_4$, $\Alt_4$ or $\Sym_4$ and contains at least one of the elements described in (1). In particular, all groups occur (but we do not find all possibilities for compatible real structures).
	\end{enumerate}
\end{prop}

In what follows we assume that $\gamma\notin G$.

\addtocontents{toc}{\protect\setcounter{tocdepth}{1}}
\subsection{\bf Case $G\cong\ZZ/2$}\label{dP2: Z2 without Geiser} 
Let $g$ be an involution generating $G$. Without loss of generality we may assume that $g$ acts on $\PP_\RR^2$ as $[x:y:z]\mapsto [x:y:-z]$ and then the equation of $X_\CC$ has the form
\[
\pm w^2=z^4+2f_2(x,y)z^2+f_4(x,y),
\]
where $f_4$ has no multiple factors (since $B$ is smooth). 

Assume that $g\ne \gamma$ and $X$ is strongly $\langle g\rangle$-minimal. Then $\sigma^*\ne\id$. Otherwise, $\rk\Pic(X_\CC)^{\langle g\rangle}=1$ implies $\tr{\id^*}+\tr g^*=7+\tr g^*=0$, so $g$ acts as $-\id$ on $\E_7$, i.e. coincides with $\gamma$. Thus we may assume that $\tr\sigma^*\in\{-1,1,3,5\}$. 
	
For the action on $X$ we have two possibilities
\[
{\bf (2^+)}\ \ [x:y:z:w]\mapsto[x:y:-z:w],\ \ \text{or}\ \ \ {\bf (2^-)}\ \  [x:y:z:w]\mapsto[x:y:-z:-w].
\]
We consider these two cases separately.
\newline
\newline	
$\boxed{\bf 2^-}$\hspace{0.3cm} The fixed locus $X_\CC^g$ consists of 4 points $[x:y:0:0]$ where $f_4(x,y)=0$. Thus $\tr g^*=1$, $g^*$ is of type ${A_1^3}$ in $\Weyl(\E_7)$ and $\tr\ \id^*+\tr g^*+\tr\sigma^*=8+\tr\sigma^*\in\{7,9,11,13\}$. Since $X_\CC$ is assumed to be $\Gamma\times G$-minimal, we must have $\tr(\sigma^*g^*)\in\{-7,-9,-11,-13\}$. The last three values are impossible in $\Weyl(\E_7)$, so we may assume that $\sigma^*$ is of type ${A_1^4}'$ and $\sigma^*=\gamma^*\circ g^*$. We are going to show that this case does not occur.

Indeed, run two $H$-equivariant minimal model programs on $X_\CC$, one with $H=\langle \sigma\rangle$ and the other with $H=\langle g\circ\gamma\rangle$. Their common result will be some del Pezzo surface $Z$. Since $g\circ\gamma$ is of type ($\bf 2^{+}$) it fixes an elliptic curve on $X$ (see below), so we have $K_Z^2\leqslant 4$ (it is easy to check that a del Pezzo surface $Z$ with $K_Z^2>4$ cannot contain a fixed-point elliptic curve). On the other hand, $Z$ is minimal over $\RR$, hence is not $\RR$-rational. But then $X$ is non-rational over $\RR$ too, a contradiction.
\newline
\newline
$\boxed{\bf 2^+}$\hspace{0.3cm} Then $X_\CC^g$ consists of 2 real points $[0:0:1:\pm 1]$ and a smooth genus 1 curve $\mathcal{E}=\{w^2=f_4(x,y)\}$. Thus $\tr g^*=-1$, $g^*$ is of type ${A_1^4}'$ and $\tr\ \id^*+\tr g^*+\tr\sigma^*\in\{5,7,9,11\}$. Again, the last two options are not possible for $\Weyl(\E_7)$. Thus either $\sigma^*$ is of type ${A_1^4}'$ (with $X(\RR)\approx\Sph^2$), or ${A_1^3}'$ ($X(\RR)\approx\Torus^2$), or ${A_1^3}''$ ($X(\RR)$ is a Klein bottle). 

The first two possibilities do occur. The first one is considered in Example \ref{ex: Robayo counterexample}. The second one is obtained by applying the same construction to $\Quad_{2,2}$.
	
\begin{ex}\label{ex: Robayo counterexample}
	Consider a quadric surface $Q=\{t_0^2+t_1^2+t_2^2=t_3^2\}\subset\Proj\RR[t_0,t_1,t_2,t_3]$ with $Q(\RR)\approx\Sph^2$ and three pairs of complex conjugate points
	\[
	p_{\pm}=[\pm i:\pm i\sqrt{2}:2:1],\ \ s_{\pm}=[\pm i:0:1:0],\ \ r_{\pm}=[0:\pm i:\sqrt{2}:1],\\
	\]
	lying on $Q$. Let $\widehat{g}\in\Aut(Q)$ be the automorphism given by
	\[
	[t_0:t_1:t_2:t_3]\mapsto [-t_0:-t_1:t_2:t_3]
	\] 
	(a ``rotation'' of $\Sph^2$ by $180^\circ$). Then $\widehat{g}(p_{+})=p_{-},\ \widehat{g}(s_{+})=s_{-},\ \widehat{g}(r_{+})=r_{-}$. Denote by $\pi: X\to Q$ the blow up of $Q$ at our six points and by $\widetilde{g}$ the lift of $\widehat{g}$ on $X$. We claim that
	\begin{itemize}
		\item[(1)] $X$ is a smooth real del Pezzo surface of degree 2,
		\item[(2)] The involution $\sigma^*$ on $X$ is of type $A_1^4$ (in particular, $X(\RR)\approx\Sph^2$),
		\item[(3)] $X$ is minimal with respect to $g=\gamma\circ\widetilde{g}$.
	\end{itemize}
	Let us assume that (1) holds for a moment. Note that $\Pic(X_\CC)$ is generated by three pairs of complex conjugate exceptional divisors $E_{p_{\pm}},\ E_{s_{\pm}},\ E_{r_{\pm}}$ and a pair of complex conjugate divisors $F,\ \overline{F}$, where $F=\pi^*(\ell)$, $\overline{F}=\pi^*(\overline{\ell})$, $\Pic(Q_\CC)=\ZZ[\ell]\oplus\ZZ[\overline{\ell}]$. Note that $\sigma$ permutes the members in each pair (which implies (2)), while  $\widetilde{g}$ permutes the members in each pair of exceptional divisors and preserves $F$ and $\overline{F}$. So, $\widetilde{g}^*\circ\sigma^*$ acts with trace equal to $6$ in $\Pic(X_\CC)\otimes\RR$, hence with trace equal to $5$ in $K_X^\bot\otimes\RR$. Put $g=\widetilde{g}\circ\gamma$. Since $\gamma^*$ acts as $-\id$ in $K_X^\bot\otimes\RR$ one has $\tr\big ((\gamma\circ g)^*\circ\sigma^*\big )=-5$, so $X_\CC$ is $\langle g\rangle\times\Gamma$-minimal.
	
	Finally, let us prove (1). For convenience of calculation, let us make the linear change of coordinates
	\[
	T_0=t_3-t_2,\ \ T_1=t_0-it_1,\ \ T_2=t_0+it_1,\ \ T_3=t_3+t_2.
	\]
	Then $Q=\{t_3^2-t_2^2=t_0^2+t_1^2\}$ is given by $T_0T_3=T_1T_2$ in $\Proj\CC[T_0,T_1,T_2,T_3]$ and the blown up points are
	\begin{gather*}
		p=[-1:i+\sqrt{2}:i-\sqrt{2}: 3],\ \ \widetilde{p}=[-1:-i-\sqrt{2}:-i+\sqrt{2}: 3],\ \ \ \ \ \ \
		s=[-1:i:i:1],\ \ \widetilde{s}=[-1:-i:-i:1],\\
		r=[1-\sqrt{2}:1:-1:1+\sqrt{2}],\ \ \widetilde{r}=[1-\sqrt{2}:-1:1:1+\sqrt{2}].
	\end{gather*}
	Divisors of bidegree $(1,0)$ and $(0,1)$ on $Q$ are the lines $T_1=tT_0$, $T_3=tT_2$ and $T_2=tT_0$, $T_3=tT_1$. It can be easily checked that no two points from above lie on such lines. Further, divisors of bidegree $(1,1)$ are hyperplane sections of $Q$, but our points do not simultaneously satisfy the equation $\alpha T_0+\beta T_1+\gamma T_2+\delta T_3$. Next, assume that our six points lie on the curve $C$ of bidegree $(1,2)$ (note that $C$ is smooth). Then $g(C)$ is a curve of bidegree $(2,1)$ still containing all six points. But $C\cdot g(C)=5$, a contradiction. Finally, assume that the six points lie on a curve $E$ of bidegree $(2,2)$. Note that $E$ has at most one ordinary double point and $E^2=8$. So, the self-intersection of a strict transform of $E$ after the blow-up is at least $-1$. 
\end{ex}

\subsection{\bf Case $G\cong\ZZ/4$} 

Let $g$ be a generator of $G_B\subset \PGL_3(\RR)\cong\SL_3(\RR)$. We may assume that $g$ acts as $[x:y:z]\mapsto [-y:x:z]$. The equation of $B$ then has the form
\begin{equation}\label{eq: quartic order 4}
z^4+Az^2(x^2+y^2)+B(x^4+y^4)+Cx^2y^2+D(x^3y-xy^3)=0,
\end{equation}
	There are two possible lifts of $g$ to an automorphism of $X$, namely
	\[
	{\bf (4^+)}\ \ [x:y:z:w]\mapsto [-y:x:z:w]\ \ \ \text{or}\ \ {\bf (4^-)}\ \ [x:y:z:w]\mapsto [-y:x:z:-w].
	\]

	We treat these two cases separately.
	\newline
	\newline
	$\boxed{\bf 4^+}$\hspace{0.3cm} Let $p=[x:y:z:w]\in\PP(1,1,1,2)$ be a point fixed by $g$. Then either $p=[0:0:z:w]$ or $p=[1:\pm i:0:0]$. The condition $p\in X$ implies $w/z^2=\pm 1$ in the former case, and $2B-C\pm 2iD=0$ in the latter. It follows that in the second case we have $C=2B$, $D=0$, so $F$ reduces to the form
	\begin{equation}\label{eq: singular quartic}
	z^4+Az^2(x^2+y^2)+B(x^2+y^2)^2=0,
	\end{equation}
	which is singular. Therefore $\Fix(g,X_\CC)=\{[0:0:1:\pm 1]\}$ and $\tr g^*=-1$. It follows that the conjugacy class of $g^*$ in $\Weyl(\E_7)$ is $A_3\times A_1^2$ or $A_3^2$. 
	
	In the first case $\Sp(g^*)=\{\pm i,-1^3,1^2\}$, so $\Sp(g^2)^*=\{-1^2,1^5\}$ and $(g^2)^*$ cannot be of type $A_1^4$, a contradiction. Thus only the case of $A_3^2$ remains. This case does occur, see Example \ref{ex: ord 4 dP2}
	\newline
	\newline
	$\boxed{\bf 4^-}$\hspace{0.3cm} As above, let $p=[x:y:z:w]\in\PP(1,1,1,2)$ be a point fixed by $g$. Then $p=[1:\pm i: 0:\beta]$, where $\beta=w/x^2$. By Remark \ref{rem: holom Lefs} the set $X(\CC)^g$ is not empty. It is easy to check then that $X(\CC)^g$ consists of 4 points $[1:\pm i:0:\beta]$, where $\beta^2=2B-C+2iD$ and $\beta^2=2B-C-2iD$, so $D=0$ ($2B-C\ne 0$: otherwise $2B=C$, $D=0$, which gives a singular quartic).
	
	So, $\tr g^*=1$ and $g^*$ belongs to the class $(A_3\times A_1)'$, $(A_3\times A_1)''$ or $\D_4(a_1)\times A_1$ in $\Weyl(\E_7)$. The same arguments as in the $(\bf 4^+)$-case exclude the first two possibilities. The remaining case does occur, see Example \ref{ex: ord 4 dP2}.
	
	\begin{ex}\label{ex: ord 4 dP2}
		Consider a smooth real del Pezzo surface
		\[
		X=\big \{[x:y:z:w]:\ x^4+6x^2y^2+y^4-2z^4=w^2\big \}\subset\Proj\PP_\RR(1,1,1,2).
		\]
		The curve $B=\big \{[x:y:z]:\ x^4+6x^2y^2+y^4-2z^4=0\big \}$ is a smooth plane quartic with one oval. One can easily find all 28 bitangents of $B$ and 56 lines on $X_\CC$:
		\begin{gather*}
		\big\{ w=\pm \sqrt{2}iz^2,\ \ x=\alpha_1 y\big \},\ \ \ \alpha_1=i(\pm 1\pm\sqrt{2}),\\
		\big \{ w=\pm (x^2+3y^2),\ \ z=\alpha_2 y\big \}\ \ \text{and}\ \ \big \{w=\pm(3x^2+y^2),\ \ z=\alpha_2 x \big \},\ \ \ \alpha_2^2=\pm 2i,\\
		\big \{ w=\pm \frac{1}{\sqrt{2}}(x-y)^2,\ \ z=\alpha_3 (x+y)\big \}\ \ \text{and}\ \ \big \{w=\pm\frac{1}{\sqrt{2}}(x+y)^2,\ \ z=\alpha_3 (x-y) \big \},\ \ \ \alpha_3^2=\pm \frac{1}{2},\\
		\big \{ w=\pm i(x^2+4ixy-y^2),\ \ z=\alpha_4 (x+iy)\big \}\ \ \text{and}\ \ \big \{w=\pm i(x^2-4ixy-y^2),\ \ z=\alpha_4 (x-iy) \big \},\ \ \ \alpha_4^2=\pm 1.
		\end{gather*}
		(in \cite{Trepalin} these sets of lines are called $\theta$-, $\eta$-, $\sigma$-, and $\tau$-lines respectively). Consider the automorphism $g: [x:y:z:w]\mapsto[-y:x:z:\pm w]$. It is easy to check that there are no disjoint $\langle g\rangle$-orbits defined over $\RR$, so $X$ is $\langle g\rangle$-minimal.
	\end{ex}

\subsection{Other groups} 

We may assume that the groups listed in (\ref{eq: dP2 list}) contain an element of order 2 or 4, described above. This shows that all these groups can act strongly minimally on a real del Pezzo surface of degree 2. We however do not find all compatible real structures here.

\subsection{$G$-links}

To classify isomorphism classes of $(X,G)$, we use classification of Sarkisov links \cite[Corollary 7.11]{di} or \cite[Theorem 2.6]{isk-1} which says that if a del Pezzo surface of degree $d$ has no orbits of length $<d$, then $X$ is superrigid. In particular a Del Pezzo surface of degree 2 is superrigid unless $G$ has a fixed point, which must be real in our case. The only possible link is a birational Bertini involution (see the next paragraph). From Lemma \ref{lem: PGL subg} we conclude that $X$ is $G$-superrigid for groups $\Alt_4$, $\Sym_4$ (not containing $\gamma$), or the groups $H\times (\ZZ/2)^2$, $H\times\ZZ/4$, $H\times\Sym_3$, $H\times\Dih_4$, $H\times\Alt_4$, $H\times\Sym_4$, where $H$ is generated by the Geiser involution.

\section{Del Pezzo surfaces of degree 1}\label{section: dP1} Let $X$ be a del Pezzo surface of degree 1 over a field $\kk$. Its anticanonical model \[\Proj\bigoplus_{n\geqslant 0} H^0(X,-nK_X)\] is a smooth sextic hypersurface $f(w,x,y,z)=0$ in $\PP_{\kk}(3,1,1,2)$. Write $f(w,x,y,z)=w^2-g_3(x,y,z)w-g_6(x,y,z)$, where $g_i\in\kk[x,y,z]$ is a polynomial of (graded) degree $i$. If $\Char\kk\ne 2$ one can make the change of variables $w\mapsto w+g_3/2$ and reduce the equation to the form
\[
w^2=Az^3+z^2h_2(x,y)+zh_4(x,y)+h_6(x,y).
\]
When $\kk=\RR$ one can make the change of variables $z\mapsto z/\sqrt[3]{A}-B/3\sqrt[3]{A^2}$ and reduce the equation of $X$ to the form
\begin{equation}\label{eq: dP1}
w^2=z^3+f_4(x,y)z+f_6(x,y).
\end{equation}
The linear system $|-2K_X|$ has no base points and exhibits $X$ as a double cover of a quadratic cone $Q\subset\Proj\RR[x,y,z]$. The corresponding Galois involution $\beta$ is called the {\it Bertini involution} and acts as
\[
[w:x:y:z]\mapsto[-w:x:y:z]=[w:-x:-y:z].
\]
Its fixed point locus $X^\beta$ is the union of a curve $R\subset Q$ of genus 4 and a single point $q$. This point is the unique base point of  the elliptic pencil $|-K_X|$, so in particular $q\in X(\RR)$. 

\begin{rem}\label{rem: dP1 Weyl}
	In Table \ref{table: dP1 real forms} below we collect some information about real structures on del Pezzo surfaces of degree 1. This time we do not restrict ourselves to $\RR$-rational surfaces only, because --- as will become clear in \S \ref{subsec: dP1 S3} --- we should have a closer look at involutions' conjugacy classes in $\Weyl(\E_8)$, and deal with the fact that sometimes the Carter graph does not determine an involution up to conjugacy. 
	
	For an irreducible reflection group $\Weyl$ acting on a vector space $V$, and involution $\sigma^*\in\Weyl$, define $i(\sigma^*)=\dim V^{-}$, where $V=V^{+}\oplus V^{-}$ is the decomposition into eigenspaces. In the notation of Table~\ref{table: dP1 real forms}, $i(\sigma^*)$ is the sum of lower indices. Note that there is a central involution $-\id$ in $\Weyl(\E_8)$, which induces a correspondence of each $\sigma^*$ with $\sigma_{t}^*$ (called the {\it Bertini twist} of $\sigma^*$ in \S \ref{subsec: dP1 S3}), where $i(\sigma_{t}^*)=8-i(\sigma^*)$. It will be important for us that two classes with $i(\sigma^*)=4$ are both self-corresponding under this, see \cite[\S 2]{Wall} for details. 
\end{rem}

\begingroup
\renewcommand*{\arraystretch}{1.4}
\begin{longtable}{|c|c|c|c|c|c|c|c|}
	\caption{Involution in $\Weyl(\E_8)$ and real forms of del Pezzo surfaces of degree 1}
	\label{table: dP1 real forms} \\
	\hline
	\cite{Wall} & \cite{weyl} & Eigenvalues of $\sigma^*$ & $\tr\sigma^*$ & $X(\RR)$ &  number of real lines on $X$ \\ \hline
	
	$1$ & $\varnothing$ & $1^8$ & 8 & $\#9\RR\PP^2$ & 240  \\ \hline
	$A_1$ & $A_1$ & $-1,1^7$ & 6 & $\#7\RP^2$  & 126  \\ \hline
	$A_1^2$ & $A_1^2$ & $-1^2,1^6$ & 4 & $\#5\RP^2$  & 60  \\ \hline
	$A_1^3$ & $A_1^3$ & $-1^3,1^5$ & 2 & $\#3\RP^2$  & 26  \\ \hline
	$A_1^4$ & ${A_1^4}''$ & $-1^4,1^4$ & 0 & $\RP^2$  & 8  \\ \hline
	$D_4$ & ${A_1^4}'$ & $-1^4,1^4$ & 0 & $\Sph^2\sqcup\#3\RP^2$  & 24  \\ \hline
	$A_1\times D_4$ & ${A_1^5}$ & $-1^5,1^3$ & $-2$ & $\Sph^2\sqcup\RP^2$  & 6  \\ \hline
	$D_6$ & ${A_1^6}$ & $-1^6,1^2$ & $-4$ & $2\Sph^2\sqcup\RP^2$  & 4  \\ \hline 
	$E_7$ & ${A_1^7}$ & $-1^7,1$ & $-6$ & $3\Sph^2\sqcup\RP^2$  & 2  \\ \hline 
	$E_8$ & ${A_1^8}$ & $-1^8$ & $-8$ & $4\Sph^2\sqcup\RP^2$  & 0  \\ \hline 
\end{longtable}
\endgroup

Since $\Aut(X)$ fixes $q$, we have the natural faithful representation
\[
\Aut(X)\hookrightarrow\GL(T_qX)\cong\GL_2(\RR),
\]
so either $\Aut(X)\cong\ZZ/n$ or $\Aut(X)\cong\D_n$. 

\vspace{0.3cm}

Let $G\subset\Aut(X)$. The action of $G$ on the pencil $|-K_X|$ induces the action on $C=\Proj\RR[x,y]\cong\PP_\RR^1$ (recall that by construction $\{x,y\}$ is a basis in $H^0(X,-K_X)$). 
This gives us the natural homomorphism \[\upsilon: G\to\Aut(C)=\PGL_2(\RR).\] Put $G_0=\Ker\upsilon$. Every element of $G_0$ acts on $\PP_\RR(3,1,1,2)$ as $\diag\{\alpha,\varepsilon,\varepsilon,\beta\}$, where $\varepsilon=\pm 1$, and $\varepsilon^6=\alpha^2=\beta^3$. Thus $\beta=1$, $\alpha=\pm 1$, and $|G_0|\leqslant 2$. Moreover, $G_0\cong\ZZ/2$ means that $\beta\in G$.

\vspace{0.3cm}

Since each $g\in G\subset\Aut(X)$ leaves the equation (\ref{eq: dP1}) invariant, $g$ must have the form
\[
[w:x:y:z]\mapsto [w:ax+by:cx+dy:z], \ \ \ a,b,c,d\in\RR,
\]
(unlike the case $\kk=\CC$). In particular, $f_4(x,y)$ and $f_6(x,y)$ are absolute invariants of $G$. From the list of basis invariants in Appendix \ref{appendix} we get that for $n>4$, $n\ne 6$, one has $f_4=a(x^2+y^2)^2$, $f_6=b(x^2+y^2)^3$, so $27f_6^2+4f_4^3=(27b^2+4a^3)(x^2+y^2)^6$, and hence $X$ is singular. Moreover, $X$ does not admit $H=\ZZ/3$-action with $\Pic(X)^H\simeq\ZZ$ by \cite[Theorem 1.2]{Yas}. Therefore, a \textcolor{black}{strongly minimal} $G$ can be isomorphic to one of the following groups:
\begin{equation}\label{eq: dP1 list}
	\ZZ/2,\ \ \ZZ/4,\ \ \ZZ/2\times\ZZ/2,\ \ \ZZ/6,\ \ \Dih_3,\ \ \Dih_4,\ \ \Dih_6.
\end{equation}

Note that $\rk(X_\CC)^\beta=1$, hence to classify groups acting \textcolor{black}{strongly minimally} on $X$, we can focus only on those that do not contain the Bertini involution. 

\begin{prop}\label{prop: dP1 classification}
	Let $X$ be a real $\RR$-rational del Pezzo surface of degree 1 and $G$ be a finite group acting \textcolor{black}{strongly minimally} on $X$. Then $G$ contains the Bertini involution and we are in one of the following cases (\textcolor{black}{all groups indeed act on $X$ strongly minimally}):
	
	\begingroup
	\renewcommand*{\arraystretch}{1.2}
	\begin{longtable}{|c|c|c|c|}
		\caption{Minimal automorphism groups of del Pezzo surfaces of degree 1.}\label{table: dP 1}
		\label{table: dP1} \\
		\hline
		$G$ & Generators & $f_4(x,y)$ & $f_6(x,y)$ \\ \hline
		
		\multicolumn{4}{|c|}{\it Cyclic groups $G=\langle r\rangle$} \\ \hline
		
		$\ZZ/2$ & $\beta$ & $f_4(x,y)$ & $f_6(x,y)$ \\ \hline
		$\ZZ/4$ & $R_4$ & $ax^4+bx^2y^2+ay^4+cxy(x^2-y^2)$ & $(x^2+y^2)(a'x^4+d'x^3y+b'x^2y^2-d'xy^3+a'y^4)$ \\ \hline
		$\ZZ/6$ & $R_6$ & $a(x^2+y^2)^2$ & \begin{tabular}{@{}c@{}}$b(x^2+y^2)^3+c(x^6-15x^4y^2+15x^2y^4-y^6)$, \\  $+d(6x^5y-20x^3y^3+6xy^5)$.\end{tabular} \\ \hline
		
		\multicolumn{4}{|c|}{\it Dihedral groups $G=\langle R_n,S\ |\ R_n^n=S^2=1, SR_nS^{-1}=R_n^{-1}\rangle$} \\ \hline
		
		$(\ZZ/2)^2$ & $R_2,\ S$ & $ax^4+bx^2y^2+cy^4$ & $a'x^6+b'x^4y^2+c'x^2y^4+d'y^6$ \\ \hline
		
		$\Dih_4$ & $R_4,\ S$ & $ax^4+bx^2y^2+cy^4$ & $(x^2+y^2)(a'x^4+b'x^2y^2+a'y^4)$ \\ \hline
		
		$\Dih_6$ & $R_6,\ S$ & $a(x^2+y^2)^2$ & $b(x^2+y^2)^3+c(x^6-15x^4y^2+15x^2y^4-y^6)$ \\ \hline
	\end{longtable}
	\endgroup
\end{prop}

\begin{proof}
	All the groups listed in Table \ref{table: dP 1} do contain the Bertini involution, so they act strongly minimally on $X$. To write down the corresponding equation one should consult Appendix \ref{appendix}. It remains to exclude the groups which do not contain the Bertini involution. Below we assume that $\beta\notin G$ and $G\ncong\Sym_3$. The case $G\cong\Sym_3$ requires more thorough analysis and is excluded in paragraph \ref{subsec: dP1 S3}.
	
	\subsection*{{\bf Case} $G=\ZZ/2$} Denote by $g$ an involution which generates $G$. We may assume that $g$ acts on $T_qX\cong\RR^2$ as $\diag\{-1,1\}$. The set $X_\CC^g$ is the disjoint union of the elliptic curve $x=0$ and 2 or 3 points $w=y=0$, so $\tr g^*\in\{-1,0\}$. However, there are no involutions in $\Weyl(\E_8)$ whose trace equals $-1$, so we assume $\tr g^*=0$. If $X_\CC$ is strongly $\Gamma\times G$-minimal, then $\tr\id^*+\tr g^*+\tr\sigma^*+\tr(\sigma\circ g)^*=8+\tr\sigma^*+\tr(\sigma\circ g)^*=0$. Hence $\tr\sigma^*=0,\ \tr(\sigma\circ g)^*=-8$ (see Table \ref{table: dP1 real forms}). The latter equality implies that $\sigma^*=g^*\circ\beta^*$. 
	
	Now run two $H$-equivariant minimal model programs on $X_\CC$, one with $H=\langle \sigma\rangle$ and the other with $H=\langle g\circ\beta\rangle$. Their common result will be some del Pezzo surface $Z$. Since $g\circ\beta$ fixes an elliptic curve on $X$ (as well as $g$), we have $K_Z^2\leqslant 4$ (it is easy to check that a del Pezzo surface $Z$ with $K_Z^2>4$ cannot contain a fixed-point elliptic curve). On the other hand, $Z$ is minimal over $\RR$, hence is not $\RR$-rational. But then $X$ is non-rational over $\RR$ too, a contradiction.
	
	\subsection*{\bf Case $G=\ZZ/2n$, $n\geqslant 2$}\label{subsection: dP1 Z4 Z8} Let $g$ generate $G$. As $G$ does not contain the Bertini involution, we may assume that $g^n$ acts as $\diag\{-1,1\}$ on $T_qX$, so $\det g^n=-1$. But each $h\in\GL_2(\RR)$ with $2<\ord h<\infty$ has determinant equal to 1, a contradiction.
	
	\subsection*{\bf Case $G=\Dih_n$, $n\geqslant 2$} It is easy to see that $G=\Dih_2\cong(\ZZ/2)^2$ always contains the Bertini involution. So, we assume that $n>2$ and $n$ is even. Then $G$ contains an element of order $n$ whose $(n/2)^{\rm th}$-power is not the Bertini involution. The same argument as before shows that this is impossible.
\end{proof}

\addtocontents{toc}{\protect\setcounter{tocdepth}{2}}
\subsection{Geometry of $\davidsstar$-configurations and $\Sym_3$-actions}\label{subsec: dP1 S3}

	We now apply the techniques of \cite{Trepalin19} to analyze $\Sym_3$-actions on real del Pezzo surfaces of degree 1. More precisely, we now show that if $\beta\notin G\cong\Sym_3$, then $G$ cannot act on any $\RR$-rational del Pezzo surface of degree 1 with invariant Picard number equal to one (the $\RR$-rationality assumption is crucial). 
	
	So, assume $G=\langle g,h\ |\ g^3=h^2=1, gh=hg^{-1}\rangle$ and $\rk\Pic(X_\CC)^{\Gamma\times G}=1$. Since we suppose $\beta\notin G$, all involutions in $G$ have zero traces on $K_X^\bot$ (i.e. of types ${A_1^4}'$ or ${A_1^4}''$). All elements of order 3 in $G$ are of type $A_2^2$, with trace equal 2, as was shown in \cite[\S 5.4]{Yas}. The formula (\ref{characterformula}) implies
	\begin{equation}\label{eq: dP1 S3 invariant Pic}
	\rk\Pic(X_\CC)^G=3.
	\end{equation}
	Following the terminology of \cite{Trepalin19}, we say that six $(-1)$-curves $H_1,\ldots,H_6$ on a del Pezzo surface of degree 1 form a {\it $\davidsstar$-configuration}, if
	\[
	H_i\cdot H_{i+1}=0,\ \ \ H_i\cdot H_{i+2}=2,\ \ \ H_i\cdot H_{i+3}=3
	\]
	(all subscripts are modulo 6). In fact,
	\begin{equation}\label{eq: start config}
	H_i+H_{i+3}\sim -2K_X,\ \ \ \ \ H_i+H_{i+2}+H_{i+4}\sim -3K_X,\ \ \ \ \ \sum_{i=1}^6 H_i\sim -6K_X,
	\end{equation}
	and $H_{i+3}=\beta(H_i)$. The graph of a $\davidsstar$-configuration looks like $\davidsstar$, where the vertices denote $(-1)$-curves, the edges denote intersections with multiplicity $2$, and the curves corresponding to opposite vertices intersect with multiplicity $3$. Two $\davidsstar$-configurations $\davidsstar=\{H_1,\ldots,H_6\}$ and $\davidsstar'=\{H_1',\ldots,H_6'\}$ are called {\it asynchronized} if $H_i\cdot H_j'=1$ for any $i,j$. 

	By \cite[Lemma 4.12]{Trepalin}, for every element $g^*$ of type $A_2^2$ there are twelve $g$-invariant $(-1)$-curves on $X_\CC$ forming two $\davidsstar$-configurations, and two $g$-invariant $\davidsstar$-configurations on which $g$ acts faithfully. These four configurations are pairwisely asynchronized. Denote the first two configurations by $\davidsstar_A=\{A_1,\ldots,A_6\}$ and $\davidsstar_B=\{B_1,\ldots,B_6\}$, and the last two (where $g$ acts faithfully) by $\davidsstar_C=\{C_1,\ldots,C_6\}$ and $\davidsstar_D=\{D_1,\ldots,D_6\}$. Let us choose the numbering in every 6-tuple so that the first two entries are disjoint (i.e. neighbors in $\davidsstar$ graph). By the proof of \cite[Lemma 4.15]{Trepalin}, the classes $a_i=A_i+K_X,\ b_i=B_i+K_X,\ c_i=C_i+K_X$ and $d_i=D_i+K_X$, $i=1,2$, form a basis of the vector space $V=\Pic(X_\CC)\otimes\RR\cap K_X^\bot$. We may assume that $g$ acts on $\davidsstar_C$ and $\davidsstar_D$ by rotating them (more precisely, the ``triangles'' $\bigtriangleup$ and $\bigtriangledown$) counterclockwise. Using relations (\ref{eq: start config}), one easily finds the matrix of $g^*$ in our basis:
	\[
	g^*=I_4\oplus \begin{pmatrix}
	-1 & -1\\
	1 & 0
	\end{pmatrix}\oplus
	\begin{pmatrix}
	-1 & -1\\
	1 & 0
	\end{pmatrix},\ \ \ 
	(g^{-1})^*=I_4\oplus \begin{pmatrix}
	0 & 1\\
	-1 & -1
	\end{pmatrix}\oplus
	\begin{pmatrix}
	0 & 1\\
	-1 & -1
	\end{pmatrix},
	\]
	Now the involution $h\in G$ acts on the $\davidsstar$-configurations, and it is easy to see that condition (\ref{eq: dP1 S3 invariant Pic}) implies that there is a $G$-invariant $\davidsstar$-configuration among our four (the incidence relation in $\davidsstar$ together with $gh=hg^{-1}$ show that $h$ acts either trivially, or as a central symmetry). We call this invariant configuration $\davidsstar_0$. 
	
	Similarly, the $\Gamma\times G$-minimality of $X_\CC$ implies that $\Gamma$ acts by central symmetry on $\davidsstar_0$. Denote by $\sigma^*$ the image of the complex involution on $X$ in the Weyl group $\Weyl(\E_8)$, and assume that $X$ is given by equation (\ref{eq: dP1}). Changing the sign of $w^2$ in that equation gives another del Pezzo surface of degree 1, which we denote $X[\beta]$ and call the {\it Bertini twist} of $X$. If $\sigma_{t}$ is the complex involution on $X[\beta]$, then its image in $\Weyl(\E_8)$ equals $\sigma_{t}^*=\beta^*\circ\sigma^*$. Note that $\beta^*$ acts as $-\id$ on $K_X^\bot$, and therefore $\tr(\sigma_{t}^*)=-\tr\sigma^*$. In particular,
\[
\rk\Pic X[\beta]_\CC^{\Gamma_{t}\times G}=3,
\]
where $\Gamma_{t}=\langle\sigma_{t}\rangle$. The output of $G$-Minimal Model Program on $X[\beta]$ is a real $G$-minimal del Pezzo cubic surface $Y$. Note that now $\Gamma_t$ stabilizes the vertices of $\davidsstar_0$, so $\sigma_{t}^*$ has the same type in $\Weyl(\E_8)$ as the type of the complex involution on $Y$ --- i.e. $\id$, $A_1$, $A_1^2$, $A_1^3$, or some lift of $A_1^4$ (see Section \ref{sec: dP3}). Therefore the original involution $\sigma^*$ is of type $A_1^8$, $A_1^7$, $A_1^6$, $A_1^5$, ${A_1^4}'$ or ${A_1^4}''$. The first four correspond to non-$\RR$-rational del Pezzo surfaces. So, we may assume that the complex involution on $Y$ is of type $A_1^4$, and hence both $Y$ and $X[
\beta]$ are not $\RR$-rational. As was noticed in Remark \ref{rem: dP1 Weyl}, both classes ${A_1^4}'$ and ${A_1^4}''$ are self-dual under the Bertini twist, hence $X$ is not rational over $\RR$ either.

\subsection{Embedding into $\Cr_2(\RR)$ and conjugacy classes} A priori it is not clear that one can choose the coefficients of $f_4$ and $f_6$ in Table \ref{table: dP1} in such a way that the corresponding surfaces are $\RR$-rational. Here is one of the possible approaches to this problem.

Let $\widetilde{X}$ denote the blow-up of $X$ at $q$. By Proposition \ref{prop: dP criterion of R-rationality}, the surface $X$ is $\RR$-rational if and only if $\widetilde{X}(\RR)$ is connected. The surface $\widetilde{X}$ is an elliptic fibration over $\PP_\RR^1$ with a real section (coming from the exceptional divisor of the blow-up). As shown in \cite[\S 5]{Wall} the set $\widetilde{X}(\RR)$ is connected if\footnote{This condition is sufficient but not necessary. In fact $\widetilde{X}(\RR)$ is non-connected if $\Eu(\widetilde{X}(\RR))>0$, but the case $\Eu(\widetilde{X}(\RR))=0$ is more subtle and we will not discuss it here.} $\Eu(\widetilde{X}(\RR))<0$. Now $\Eu(\widetilde{X}(\RR))$ is the sum of Euler characteristics of singular fibers.

Recall that every geometrically singular member of $|-K_X|$ is an irreducible curve of arithmetic genus 1. Therefore, each singular fiber of the fibration $\widetilde{X}\to\PP_\RR^1$ is a rational curve with a unique singularity which is either a node or a simple cusp. From the point of view of Euler characteristic only the nodes do matter: we have contributions $+1$ from each acnode (a singularity which is equivalent to the singularity $y^2=x^3-x^2$ over $\RR$) and $-1$ for each crunode (those which are equivalent to $y^2=x^3+x^2$), see Figure \ref{pic:acnode crunode}.

\begin{figure}[h!]
	\centering
	\begin{subfigure}{.49\textwidth}
		\centering
		\includegraphics[width=.85\linewidth]{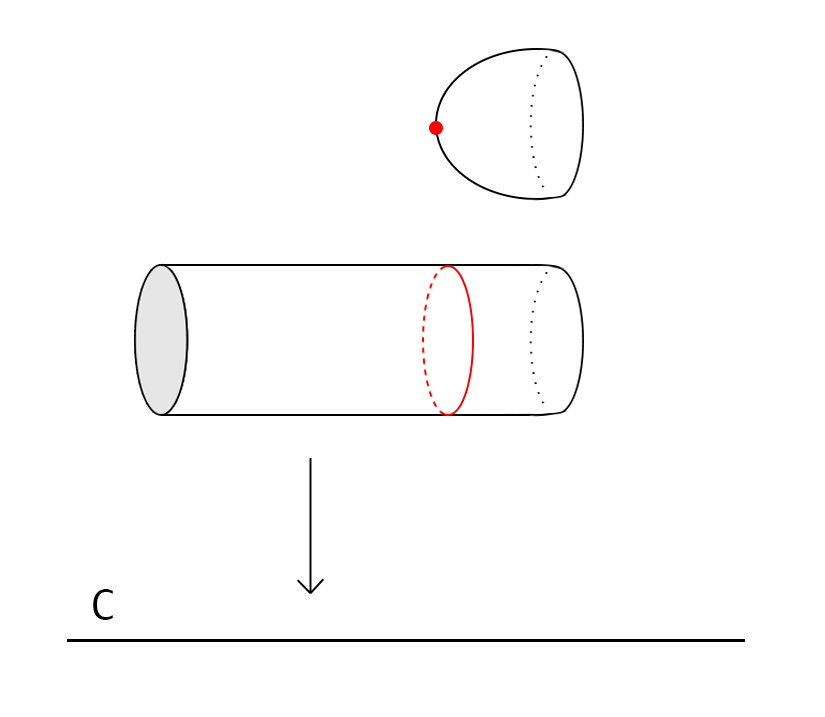}
	\end{subfigure}
	\begin{subfigure}{.49\textwidth}
		\centering
		\includegraphics[width=.85\linewidth]{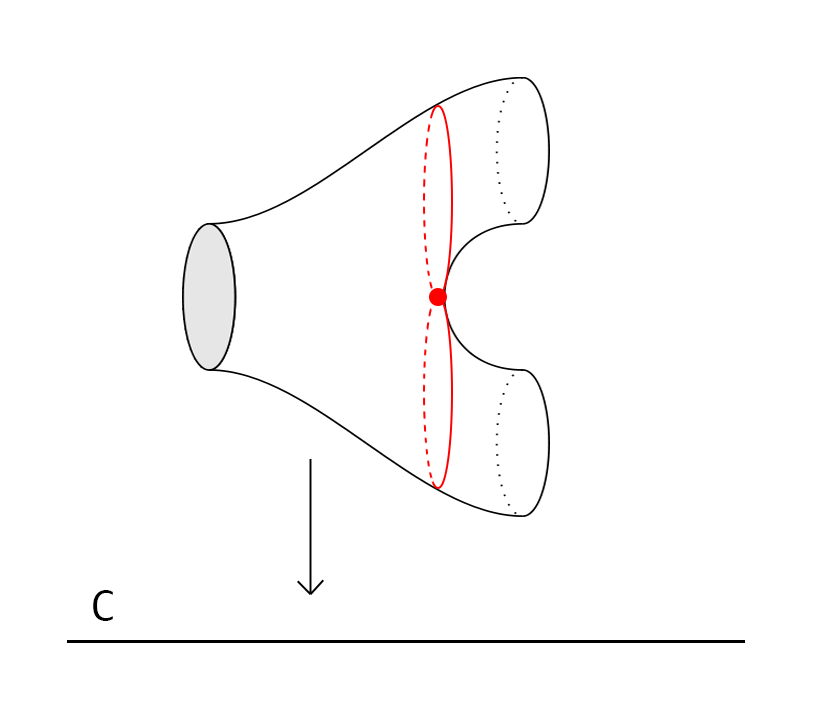}
	\end{subfigure}
	\caption{Transition between fibers with 1 component and those with 2 components through an acnode (left) and a crunode (right).}
	\label{pic:acnode crunode}
\end{figure}

Now an easy calculation shows that the coefficients of the binary forms $f_4(x,y)$ and $f_6(x,y)$ from Table \ref{table: dP 1} can be chosen in such a way that the number of crunodal curves in (\ref{eq: dP1}) is greater than acnodal ones, so $\widetilde{X}(\RR)<0$.

So, all the groups $G$ from Table \ref{table: dP 1} do embed into $\Cr_2(\RR)$. Also note that any del Pezzo surfaces of degree 1 is $G$-superrigid (see \cite[Corollary 7.11]{di} and \cite[Theorem 2.6]{isk-1}). In particular, none of the groups listed in Proposition \ref{prop: dP1 classification} is linearizable.

\appendix

\addtocontents{toc}{\protect\setcounter{tocdepth}{1}}
\section{Simple groups and $p$-groups acting on real rational surfaces}\label{app: simple groups}

In this Appendix we show that classification of subgroups of some {\it particular types} in $\Cr_2(\RR)$ can be much simpler than the analogous question in complex settings. Also, these results can be obtained directly, i.e. avoiding the complete classification. 

In this section $X$ denotes a real smooth {\it geometrically rational} (not necessarily $\RR$-rational) surface. Let $G\subset\Bir(X)$ be a finite group. Then, applying $G$-equivariant minimal model program to $X$, we can assume that $X$ is either a real del Pezzo surface with $\Pic(X)^G\cong\ZZ$, or a real surface with $G$-equivariant conic bundle structure and $\Pic(X)^G\cong\ZZ^2$ \cite[Theorem 5]{di-perf}.

Our goal is to classify simple groups acting on real geometrically rational surfaces. Let us first recall the situation in the case $\kk=\CC$. As a by-product result of \cite{di} one has the next theorem.

\begin{thm}
	Let $G\subset\Cr_2(\CC)$ be a finite non-abelian simple group. Then $G$ is isomorphic to one of the following groups:
	\[
	\Alt_5,\ \ \ \Alt_6,\ \ \ \PSL_2(\FF_7).
	\]
	More precisely, we have the following characterization of these groups.
	\begin{itemize}
		\item There are 2 conjugacy classes of subgroups isomorphic to $\PSL_2(\FF_7)$. First,
		$\PSL_2(\FF_7)$ embeds into $\PGL_3(\CC)$ and preserves the Klein quartic $x^3y+y^3z+z^3x=0$. Second, it embeds as a group of automorphisms of the double cover of 
		$\PP_\CC^2$, ramified along that Klein quartic (i.e. a del Pezzo surface of degree 2).
		
		\item There are 3 embeddings of $\Alt_5$ into $\Cr_2(\CC)$, up to conjugacy. The first is in $\PGL_2(\CC)$, the second is in $\PGL_3(\CC)$, and the third is in the group of automorphisms of a del Pezzo surface	of degree 5.
		
		\item Up to conjugacy, there is a unique copy of the Valentiner group $\Alt_6$, acting linearly on $\PP_\CC^2$ and preserving the sextic curve
		\[
		10x^3y^3+9x^5z+9y^5z+27z^6-45x^2y^2z^2-135xyz^4=0.
		\]
	\end{itemize}
\end{thm}

\begin{rem}
	Although a complete classification of finite subgroups in $\Cr_3(\CC)$ seems to be out of reach, Yu. Prokhorov managed to find all finite simple non-abelian subgroups of $\Cr_3(\CC)$. Besides $\Alt_5,\ \Alt_6$ and $\PSL_2(\FF_7)$, we have three new simple groups
	\[
	\SL_2(\FF_8),\ \ \ \Alt_7,\ \ \ \PSp_4(\FF_3).
	\]
	in $\Cr_3(\CC)$, see \cite{ProSimple} for details.
\end{rem}

By contrast, over $\RR$ the following holds.

\begin{thm}\label{thm: simple subgr}
	Let $X$ be a real geometrically rational surface with $X(\RR)\ne\varnothing$ acted by a simple non-abelian group $G$. Then $G\cong\Alt_5$. It has 3 embeddings into $\Cr_2(\RR)$ up to conjugacy. One in $\PGL_3(\RR)$, one in $\Aut(Q_{3,1})\cong\PO(3,1)$, and one in the group of automorphisms of a del Pezzo surface of degree 5 obtained by blowing up $\PP_\RR^2$ at the points $[1 : 0 : 0],\ [0 : 1 : 0],\ [0 : 0 : 1],\ [1:1:1]$.
\end{thm}
\begin{proof}
	Assume first that $G$ is minimally regularized on a real conic bundle $\pi: X\to\PP_\RR^1$. Since $X(\RR)\ne\varnothing$ one has $C(\RR)\ne\varnothing$, so $C\cong\PP_\RR^1$. The homomorphism $G\to\Aut(\PP_\RR^1)\cong\PGL_2(\RR)$ is either injective or trivial. In both cases $G$ embeds into automorphism group of some $\PP_\RR^1$ (which is either a base or a fiber), hence must be cyclic or dihedral by Proposition \ref{prop: PGL2 and PGL3} (i). 
	
	Now assume that $G$ is minimally regularized on a real del Pezzo surface $X$ of degree $d=K_X^2\ne 7$. We consider each $d$ separately.
	\begin{description}
		\item[$d=9$] Then $X$ is a Severi-Brauer variety. As $X(\RR)\ne\varnothing$, we have $X\cong\PP_\RR^2$ and $\Aut(X)\cong\PGL_3(\RR)$, so $G\cong\Alt_5$ by Proposition \ref{prop: PGL2 and PGL3} (ii). This is where the Valentiner group $\Alt_6$ is excluded (it does not embed into $\PGL_3(\RR)$).
		\item[$d=8$] The surface $\PP_\RR^2(1,0)$ is never $G$-minimal. If $X\cong\PP_\RR^1\times\PP_\RR^1$, we argue as in the conic bundle case. If $X\cong Q_{3,1}=\{x^2+y^2+z^2=w^2\}$, then $G\cong\Alt_5$ realized as the automorphism group of an icosahedron inscribed in the sphere
		\[
		\left (\frac{x}{w}\right )^2+\left( \frac{y}{w}\right )^2+\left(\frac{z}{w}\right )^2=1.
		\]
		The action is minimal since $\Pic(Q_{3,1})\cong\ZZ$. 
		\item[$d=6$] Then $G$ is a subgroup of $\Aut(X_\CC)\cong (\CC^*)^2\rtimes\Dih_{6}$, so it maps isomorphically to a subgroup of $\Dih_6$. So, this case does not occur.
		\item[$d=5$] Then $G$ is a subgroup of $\Aut(X_\CC)\cong\Sym_5$, so $G\cong\Alt_5$. The action of this group can be defined over $\RR$ and is always minimal, since $G$ contains a minimal element of order 5, see \cite[4.6]{Yas} or Section \ref{sec: dP5}
		\item[$d=4$] Then $X=Q_1\cap Q_2\subset\PP_\RR^4$ is an intersection of two quadrics and $G$ acts on the pencil $\QQ=\langle Q_1,Q_2\rangle\cong\PP_\RR^1$. Thus $G$ acts trivially on $\QQ$ and fixes the vertices of its singular members. But these vertices generate $\PP_\CC^4$, hence $G$ is abelian, a contradiction.
		\item[$d=3$] All possible groups of automorphisms of complex cubic surfaces are well-known, see \cite[Table 9.6]{cag} or  Section \ref{sec: dP3}. The only group to consider is $G\cong\Alt_5$. This group acts faithfully on $H^0(X,-K_X)\cong\RR^4$. It is known that there exists only one real 4-dimensional irreducible representation\footnote{Namely, the number of real irreducible representations of a finite group $G$ equals to the number of equivalence classes under real conjugacy. By definition, two elements are equivalent if they are either in the same conjugacy class or if the inverse of one element is in the conjugacy class of the other. It is known that for $\Alt_5$ (as a particular case of an ambivalent group) the number of such equivalence classes equals to the number of usual conjugacy classes, i.e. the number of  irreducible complex representations (and all these representations can be defined over $\RR$), see \cite[Chapter 16]{representation}.} of $\Alt_5$. Thus there exists a unique $\Alt_5$-invariant cubic surface in $\PP H^0(X,-K_X)$; we may assume that it is given by \[x_0+x_1+x_2+x_3+x_4=x_0^3+x_1^3+x_2^3+x_3^3+x_4^3=0\] in $\Proj\RR[x_0,x_1,x_2,x_3,x_4]$ and the set $S$ of $(-1)$-curves on $X$ consists of 27 real lines (real forms of the Clebsch cubic were described in \S \ref{sec: Clebsch}). Moreover $S=S_6\sqcup S_6'\sqcup S_{15}$, $|S_k|=k$, where the lines inside both $S_6$ and $S_6'$ are disjoint. Further, there exists a commutative diagram
		
		\[
		\xymatrix{
			& X\ar[ld]_{\pi}\ar[rd]^{{\pi'}} & \\
			\PP_\RR^2 \ar@{-->}[rr]_{} & & \PP_\RR^2 }
		\]
		such that $\pi$ (resp. $\pi'$) is a birational $\Alt_5$-morphism contracting $S_6$ (resp. $S_6'$) to the unique $\Alt_5$-orbit of length 6 in $\PP_\RR^2$. It follows that $\rk\Pic(X)^G=2$, so $X$ is not strongly $G$-minimal (in fact, it is not $G$-minimal either, as the conic bundle structures on a cubic surface are given by projecting away from a line).
		
		\item[$d=2$] Then $G$ embeds into $\Aut(B)\subset\PGL_3(\RR)$, where $B$ is a smooth quartic curve. By Proposition \ref{prop: PGL2 and PGL3} (ii), we need to consider only $G\cong\Alt_5$. But, as is well known, a genus 3 curve has no automorphisms of order 5, so this case does not occur.
		\item[$d=1$] Note that any group $G\subset\Aut(X)$ fixes a unique base point $p\in X(\RR)$ of an elliptic pencil $|-K_X|$. Thus we have a faithul representation $G\to\GL(T_pX)\cong\GL_2(\RR)$, and $G$ cannot be simple by Proposition \ref{prop: PGL2 and PGL3} (i).
	\end{description}
	Finally, let us remark that all three conjugacy classes of $\Alt_5$ in $\Cr_2(\RR)$ (corresponding to $d=5,8,9$) are different, since they are different in $\Cr_2(\CC)$, see e.g. \cite[Theorem B2]{Cheltsov}.
\end{proof}

We can generalize Theorem \ref{thm: simple subgr} a little bit. Recall that a group $G$ is called {\it quasisimple} if $[G,G]=G$ and $G/\Center(G)$ is a simple group. It was proved in \cite{Quasi} that every finite quasisimple non-simple subgroup of $G\subset\Cr_2(\CC)$ is isomorphic to $2_\bullet\Alt_5\cong\SL_2(\FF_5)$ and the embedding $G\subset\Cr_2(\CC)$ is induced by action either on $\PP^2$, or on a conic bundle. In contrast with this situation, we have

\begin{prop}
	Every quasisimple subgroup of $\Cr_2(\RR)$ is simple (and is described in Theorem \ref{thm: simple subgr}). 
\end{prop}
\begin{proof}
	Let $G\subset\Cr_2(\RR)$ be a finite quasisimple non-simple group. As usual, we assume that $G$ acts biregularly on some $\RR$-rational surface $X$. The simple group $H=G/\Center(G)$ acts on $Y=X/\Center(G)$. The surface $Y$ is clearly unirational over $\CC$, hence is $\CC$-rational by Castelnuovo's theorem. Thus $H\cong\Alt_5$ by Theorem \ref{thm: simple subgr}. Same group-theoretic arguments as in \cite[Proposition 2.1]{Quasi} imply that $\Center(G)\cong\ZZ/2$ and $G$ is the binary icosahedral group $2_\bullet\Alt_5$.
	
	Suppose that $X$ is a $G$-equivariant conic bundle over $B\cong\PP_\RR^1$. The kernel of the homomorphism $G\to\Aut(B)$ coincides with $\Center(G)=\ZZ/2$, as this is the only proper normal subgroup of $2_\bullet\Alt_5$. Thus $\Alt_5$ acts faithfully on the general fiber, which is impossible. 
	
	Now let $X$ be a del Pezzo $G$-surface with $\rk\Pic(X)^G=1$. We then argue as in the proof of Theorem \ref{thm: simple subgr}. Note that the image of every nontrivial homomorphism from $2_\bullet\Alt_5$ either contains $\Alt_5$, or coincides with  the whole group. This observation helps us to exclude all the cases\footnote{It follows from \cite{lorentz} that $\PO(3,1)$ does not contain $2_\bullet\Alt_5$. Alternative proof: assume that $G=2_\bullet\Alt_5\subset\Aut(\Quad_{3,1})$. As $G$ has no index 2 subgroups, it faithfully acts by orientation-preserving diffeomorphisms of $\Sph^2$, and hence embeds into $\SO_3(\RR)$, see Remark \ref{rem: smooth actions}. But this is  impossible by Lemma \ref{lem: PGL subg}.} except $d=3$. It remains to notice that $2_\bullet\Alt_5$ does not act on any cubic surface \cite[Table 9.6]{cag}.
\end{proof}

\subsection{$p$-subgroups in $\Cr_2(\RR)$}

Recall that a {\it $p$-group} is a finite group of order $p^k$, where $p$ is a prime. From the group-theoretic point of view, these groups are somewhat opposite to simple non-abelian groups. It follows from \cite{Yas} that for $p\geqslant 3$, every $p$-subgroup $G\subset\Cr_2(\RR)$ is conjugate either to a direct product of at most two cyclic groups, regularized on $X=\PP_\RR^1\times\PP_\RR^1$ with $\rk\Pic(X)^G=2$, or to a cyclic subgroup of $\PGL_3(\RR)$, or to $(\ZZ/3^k\times\ZZ/3^l)\rtimes(\ZZ/3)$ acting on a del Pezzo surface of degree 6, or to $\ZZ/5$ acting on a del Pezzo surface of degree 5 (with invariant Picard numbers equal to one). As the reader can see from present paper, the classification of 2-subgroups of $\Cr_2(\RR)$ is much more extensive. We leave it to the interested reader to extract this classification for del Pezzo surfaces.

Instead we give a bound on the number of generators of an abelian $p$-subgroup $G\subset\Cr_2(\RR)$ in the spirit of \cite{Bea} (where it was done for $\kk=\CC$; note that a priori Beauville's bound might fail to be sharp over $\RR$). 

\begin{prop}\label{thm: p-groups}
	Let $G\subset\Cr_2(\RR)$ be an abelian $p$-subgroup. Then it is generated by at most $r$ elements, where
	\[
	r\leqslant
	\begin{cases}
	4\ \text{if}\ p=2,\\
	3\ \text{if}\ p=3,\\
	2\ \text{if}\ p\geqslant 5.
	\end{cases}
	\]
	If $G$ is elementary, then $r\leqslant 2$ for $p=3$. For any $p$, these bounds are attained by some abelian $p$-~subgroups $G\subset\Cr_2(\RR)$.
\end{prop}
\begin{proof}
	If $G$ is minimally regularized on a real conic bundel $X\to B$, then $G$ fits into the short exact sequence
	\begin{equation}\label{eq: CB short exact}
	1\to G_F\to G\to G_B\to 1,
	\end{equation}
	where $G_B\subset\Aut(B)\cong\PGL_2(\RR)$, and $G_F$ acts by automorphisms of the generic fiber $F$. Since $G$ is finite, $G_F$ is a subgroup of $\PGL_2(\RR)$. So, both $G_F$ and $G_B$ are cyclic or dihedral, and hence $G$ is generated by at most 4 elements. Note that if $G\cong(\ZZ/p)^r$ and $p>2$, then $r=1$ or $2$.
	
	The remaining cases directly follow from the results of this paper. Note that for $p=2$ the value $r=4$ is achieved for a del Pezzo quartic surface isomorphic to $\PP_\RR^2(5,0)$ or $\Quad_{2,2}(0,2)$. The bound $r=3$ for $p=3$ is attained on a del Pezzo surface of degree 6 isomorphic to $\Quad_{2,2}(0,1)$ (so $G$ is a group of type 2b).
\end{proof}

\section{Non-solvable groups}\label{app: non-solvable}

Another interesting class of subgroups of the Cremona group is non-solvable groups. As was shown in \cite{di} and \cite{Tsygankov} the plane Cremona group (over $\CC$) already contains eight sporadic insoluble subgroups
\[
\Sym_5,\ \ \PSL_2(\FF_7),\ \ \PSL_2(\FF_7)\times\ZZ/2,\ \ \Alt_6,\ \ \Alt_5\times\Alt_4,\ \ \Alt_5\times\Sym_4,\ \ \Alt_5\times\Alt_5,\ \ (\Alt_5\times\Alt_5)\rtimes\ZZ/2,
\]
and four infinite series $\Alt_5\times\ZZ/n$, $\Alt_5\times\Dih_n$, $\SL_2(\FF_5)\times\ZZ/n$ and $\SL_2(\FF_5)\times\Dih_n$. By contrast, the following holds over $\RR$.

\begin{thm}\label{thm: non solvable}
	Let $X$ be a real geometrically rational surface with $X(\RR)\ne\varnothing$, and $G$ be a finite non-solvable group acting on it. Then the pair $(X,G)$ is isomorphic to one (and only one) of the following pairs
	\begin{itemize}
		\item $\big(\PP_\RR^2, \Alt_5\big)$;
		\item $\big(\Quad_{3,1},\Alt_5\big)$ or $\big(\Quad_{3,1},\Alt_5\times\ZZ/2\big)$;
		\item $\big(\PP^2_\RR(4,0),\Alt_5\big)$ or $\big(\PP^2_\RR(4,0),\Sym_5\big)$;
		\item $\big(Y,\Sym_5\big)$, where $Y$ is the Clebsch diagonal cubic. 
	\end{itemize}
\end{thm}
\begin{proof}
	If $G$ is minimally regularized on a conic bundle, then we again have the short exact sequence (\ref{eq: CB short exact}) with both $G_F$ and $G_B$ cyclic or dihedral. Thus $G$ is solvable.
	
	So, we may assume that $G$ acts on a real del Pezzo surface $X$ of degree $d$ with $\Pic(X)^G\cong\ZZ$. If $d=9$, then $G\cong\Alt_5$. If $d=8$ and $X\cong\Quad_{2,2}$, then $G\cong H_\bullet(\ZZ/2)^r$, where $r\in\{0,1\}$ and $H$ is a subgroup of $H_1\times H_2$ with $H_1$ and $H_2$ being cyclic or dihedral. Clearly $G$ is solvable in this case. If $d=8$ and $X\cong\Quad_{3,1}$ we have $G\subset\PO(3,1)$, so $G\cong\Alt_5$ or $\Alt_5\times\ZZ/2$, see \cite{lorentz}. When $d=6$, Proposition \ref{prop: dP6} tells us that $G\cong H_\bullet N$, where $H$ is an abelian group and $N$ is a group of order at most 6, so $G$ is solvable. For $d=5$ we have either $G\cong\Alt_5$, or $G\cong\Sym_5$ by Proposition \ref{prop: dP5}. Let $d=4$. Then $G$ is a subgroup of $\Weyl(\Dih_5)\cong(\ZZ/2)^4\rtimes\Sym_5$, so $G\cong A_\bullet H$, where $A$ is an abelian group, and $H\subset\Sym_5$. In fact it is known that $|H|<10$ \cite[Theorem 8.6.8]{cag}, so $G$ is solvable.
	
	Let $d=3$. Then \cite[Table 9.6]{cag} shows that $G\cong\Alt_5$ or $\Sym_5$. Moreover, we already know (see the proof of Theorem \ref{thm: simple subgr}) that $X$ is the Clebsch cubic and it is not $\Alt_5$-minimal. Further, in the notation of that proof let $\ell_i$, $i=1,\ldots,6$, be the elements of $S_6$ (note that all the lines on the Clebsch cubic are real). It is known that the divisor classes of $\ell_i$ and $K_X$ span $\Pic(X)\otimes\RR$, so $\Pic(X)^{\Alt_5}\otimes\RR$ is spanned by $K_X$ and the sum $\sum \ell_i$. Since $\Sym_5$ does not leave this sum invariant, the group $\Sym_5$ acts minimally on $X$.

	For $d=2$ we have $G\cong (\ZZ/2)^r\times H$, where $r\in\{0,1\}$, and $H$ is either cyclic, or dihedral, or has order $<60$. So, $G$ is solvable. When $d=1$ the group $G$ embeds into $\GL(T_qX)\cong\GL_2(\RR)$ (see Section \ref{section: dP1}), so it is solvable.
	
	Finally as we noted in Theorem \ref{thm: simple subgr} the pairs $(\PP_\RR^2,\Alt_5)$, $(\Quad_{3,1},\Alt_5)$ and $(\PP_\RR^2(4,0),\Alt_5)$ are pairwise non-isomorphic. The pair $(\PP_\RR^2(4,0),\Sym_5)$ is known to be superrigid, see \cite[8.1]{di}.
\end{proof}

\addtocontents{toc}{\protect\setcounter{tocdepth}{1}}
\section{Real invariants of some finite groups}\label{appendix}

In this appendix we collect some results concerning invariant theory of finite groups over $\RR$. They should be known to experts, but we decided to include them because we do not know proper references.

Let $V$ be a real $m$-dimensional vector space and $x_1,\ldots,x_m$ be a standard dual basis of $V^*$. Let $\rho: G\to\GL(V)$ be a faithful linear representation of a finite group $G$ and $\eta: G\to\GL(V\otimes \CC)$ be some faithful complex representation equivalent to $\rho$, i.e. $\rho(g)=T\circ\eta(g)\circ T^{-1}$ for each $g\in G$ and some $T\in\GL_m(\CC)$.

Recall that every finite subgroups of $\GL_2(\RR)$ is either cyclic $\ZZ/n\cong\langle R_n\rangle$ or dihedral \[\Dih_n\cong\langle R_n,S\ |\ R_n^n=S^2=1,SR_nS^{-1}=R_n^{-1} \rangle.\] In the sequel by {\it standard representations} of $\ZZ/n$ and $\Dih_n$ we mean
\[
\rho:\ \ R_n\mapsto
\begin{pmatrix}
\cos(2\pi/n) & -\sin(2\pi/n)\\
\sin(2\pi/n) & \cos(2\pi/n)
\end{pmatrix},\ \ \
S\mapsto
\begin{pmatrix}
1 & 0\\
0 & -1
\end{pmatrix}.
\]

In order to construct $G$-invariant del Pezzo surfaces in Sections \ref{section: dP2} and \ref{section: dP1} we need to know $G$-invariant binary forms $f_k(x,y)$ of degrees $k=2,4$ and $6$. They are listed below for different groups (in their standard representation).

\subsection*{Cyclic groups}

Let $G=\ZZ/n$. We claim that
\begin{equation}\label{eq: inv of cyclic group}
\RR[x,y]^{\rho(\ZZ/n)}=\RR[x^2+y^2,\Real(x+iy)^n,\Imag(x+iy)^n]
\end{equation}

Denote by $\omega$ a primitive $n$th root of unity. The representation $\rho$ is equivalent to $\eta: R_n\mapsto\diag\{\omega,\overline{\omega}\}$
via the map $T: x\mapsto z=x+iy$, $y\mapsto w=x-iy$. For each $g\in G$ we have $\eta(g)(Tf)=T\rho(g) T^{-1}Tf=Tf$, so $Tf\in\CC[x,y]^{\eta(\ZZ/n)}$. It is well known that $\CC[z,w]^{\eta(\ZZ/n)}=\CC[z^n,zw,w^n]$, so
\[
Tf=\sum c_{jkl}z^{nj}(zw)^k w^{nl},\ \ \ \ \text{and}\ \ \
f= \sum c_{jkl}(x+iy)^{nj}(x^2+y^2)^k (x-iy)^{nl}.
\]
Separating the real part, we get the list of basic invariants. Below we use (\ref{eq: inv of cyclic group}) mostly as a starting point for finding a nicer list of generators.

\begin{enumerate}
	\item[$\boxed{\ZZ/2}$] $\RR[x,y]^{\ZZ/2}=\RR[x^2+y^2,x^2-y^2,2xy]=\RR[x^2,xy,y^2]$
	
	\vspace{0.3cm}
	
	\begin{itemize}
	\item[$f_{2k}(x,y)$] is invariant for all $k\geqslant 1$.
	\end{itemize}
	
	\vspace{0.3cm}
	
	\item[$\boxed{\ZZ/4}$] $\RR[x,y]^{\ZZ/4}=\RR[x^2+y^2,x^4-6x^2y^2+y^4,4x^3y-4xy^3]=\RR[x^2+y^2,x^2y^2,x^3y-xy^3]$
	
	\vspace{0.3cm}
	
	\begin{itemize}
		\item[$f_2(x,y):$] $a(x^2+y^2)$;
		\item[$f_4(x,y):$] $ax^4+bx^2y^2+ay^4+cxy(x^2-y^2)$;
		\item[$f_6(x,y):$] $(x^2+y^2)(ax^4+dx^3y+bx^2y^2-dxy^3+ay^4)$
	\end{itemize}
	
	\vspace{0.3cm}
	
	\item[$\boxed{\ZZ/8}$] $\RR[x,y]^{\ZZ/8}=\RR[x^2+y^2,x^8-28x^6y^2+70x^4y^4-28x^2y^6+y^8,8x^7y-56x^5y^3+56x^3y^5-8xy^7]=\RR[x^2+y^2,xy(x^2-y^2)(x^4-6x^2y^2+y^4),x^2y^2(x^2-y^2)^2]$
	
	\vspace{0.3cm}
	
	\begin{itemize}
		\item[$f_2(x,y):$] $a(x^2+y^2)$;
		\item[$f_4(x,y):$] $a(x^2+y^2)^2$;
		\item[$f_6(x,y):$] $a(x^2+y^2)^3$.
	\end{itemize}
\end{enumerate}

\subsection*{Dihedral groups} One has
\[
\RR[x,y]^{\rho(\Dih_n)}=\RR[x^2+y^2,\Real(x+iy)^n].
\]

\begin{enumerate}
	\item[$\boxed{\Dih_2}$] $\RR[x,y]^{\Dih_2}=\RR[x^2+y^2,x^2-y^2]=\RR[x^2,y^2]$
	\begin{itemize}
		
		\vspace{0.3cm}
		
		\item[$f_2(x,y):$] $ax^2+by^2$;
		\item[$f_4(x,y):$] $ax^4+bx^2y^2+cy^4$;
		\item[$f_6(x,y):$] $ax^6+bx^4y^2+cx^2y^4+dy^6$.
	\end{itemize}

	\vspace{0.3cm}

	\item[$\boxed{\Dih_4}$] $\RR[x,y]^{\Dih_4}=\RR[x^2+y^2,x^4-6x^2y^2+y^4]=\RR[x^2+y^2,x^2y^2]$
	
	\vspace{0.3cm}
	
	\begin{itemize}
	\item[$f_2(x,y):$] $a(x^2+y^2)$;
	\item[$f_4(x,y):$] $ax^4+bx^2y^2+ay^4$;
	\item[$f_6(x,y):$] $(x^2+y^2)(ax^4+bx^2y^2+ay^4)$.
	\end{itemize}

	\vspace{0.3cm}

	\item[$\boxed{\Dih_8}$] $\RR[x,y]^{\Dih_8}=\RR[x^2+y^2,x^8-28x^6y^2+70x^4y^4-28x^2y^6+y^8]=\RR[x^2+y^2,x^2y^2(x^2-y^2)^2]$

	\vspace{0.3cm}

	\begin{itemize}
	\item[$f_2(x,y):$] $a(x^2+y^2)$;
	\item[$f_4(x,y):$] $a(x^2+y^2)^2$;
	\item[$f_6(x,y):$] $a(x^2+y^2)^3$.
	\end{itemize}
\end{enumerate}

\def\bibindent{2.5em}

\end{document}